\documentclass[11pt]{amsart}

\usepackage{graphicx}
\usepackage{caption}
\usepackage{subcaption}
\usepackage{graphicx}
\usepackage{amssymb}
\usepackage{subfig}
\usepackage{enumerate}
\usepackage[utf8]{inputenc}

\def\RR{\mathbb R}

\def\ZZ{\mathbb Z}

\DeclareMathOperator{\dete}{det}
\DeclareMathOperator{\diam}{diam}

\newcommand{\set}[1]{\left\lbrace #1\right\rbrace}
\providecommand{\abs}[1]{\left\lvert#1\right\rvert}
\providecommand{\norm}[1]{\left\lVert#1\right\rVert}
\newcommand{\remove}[1]{ }

\newtheorem{theorem}{Theorem}[section]

\newtheorem{lemma}[theorem]{Lemma}
\newtheorem{remark}[theorem]{Remark}
\newtheorem{remarks}[theorem]{Remarks}

\numberwithin{equation}{section}

\begin{document}

\title[Ingham type inequalities in lattices]{Ingham type inequalities\\  in lattices}

\author[V. Komornik]{Vilmos Komornik} 
\address{16 rue de Copenhague, 67000 Strasbourg, France}
\email{vilmos.komornik@gmail.com}
\author[A. C. Lai]{Anna Chiara Lai}
\address{Sapienza Università di Roma,
Dipartimento di Scienze di Base
e Applicate per l'Ingegne\-ria,
via A. Scarpa n. 16,
00161 Roma, Italy}      
\email{anna.lai@sbai.uniroma1.it}
\author[P. Loreti]{Paola Loreti}
\address{Sapienza Università di Roma,
Dipartimento di Scienze di Base
e Applicate per l'Ingegne\-ria,
via A. Scarpa n. 16,
00161 Roma, Italy}
\email{paola.loreti@sbai.uniroma1.it}

\maketitle
\begin{abstract}  
A classical theorem of Ingham extended Parseval's formula of the trigonometrical system to arbitrary families of exponentials satisfying a uniform gap condition. 
Later his result was extended to several dimensions, but the optimal integration domains have only been determined in very few cases. 
The purpose of this paper is to determine the optimal connected integration domains for all regular two-dimensional lattices.
\end{abstract}

\section{Introduction}\label{s1}

A classical theorem of Ingham \cite{Ing1936} extended the Parseval's formula of the trigonometrical system to arbitrary families of exponentials satisfying a uniform gap condition. Later Beurling \cite{Beu} determined the critical length of the intervals on which these estimates hold. 

Kahane \cite{Kah1962} extended these results to several dimensions. 
His theorem was improved and generalized in \cite{BaiKomLor103} (see also \cite{KomLorbook}), but the optimal integration domains have only been determined in very particular cases. 

The purpose of this paper is to determine the optimal connected integration domains for all regular two-dimensional lattices.

\section{A general framework}\label{s2}

Consider $M$ disjoint translates of $\ZZ^N$ by vectors $u_1,\ldots, u_M\in\RR^N$, and consider the functions of the form 
\begin{equation*}
f(x)=\sum_{j=1}^M\sum_{k\in\ZZ^N}a_{jk}e^{i(u_j+k,x)}=:\sum_{j=1}^Me^{i(u_j,x)}f_j(x)
\end{equation*}
with square summable complex coefficients $a_{jk}$.

Let us observe that the functions
\begin{equation}\label{21}
f_j(x)=\sum_{k\in\ZZ^N}a_{jk}e^{i(k,x)}
\end{equation} 
are $2\pi$-periodical in each variable, so that 
\begin{equation}\label{22}
\frac{1}{\abs{\Omega_0}}\int_{\Omega_0}\abs{f_j(x)}^2\ dx=\sum_{k\in\ZZ^N}\abs{a_{jk}}^2
\end{equation}
on $\Omega_0:=(0,2\pi)^N$ by Parseval's equality for multiple Fourier series, where  and $\abs{\Omega_0}=(2\pi)^N$ denotes the volume of the cube $\Omega_0$.

Next we consider $M$ vectors $v_1,\ldots, v_M\in\RR^N$
satisfying the following two conditions:
\begin{itemize}
\item[(A1)] the coordinates of each $v_k$ are multiples of $2\pi$;
\item[(A2)] the matrix $E:=\left(e^{i(u_j,v_k)}\right)_{j,k=1}^M$ is invertible.
\end{itemize}
It follows from (A1) that the translated sets
\begin{equation*}
\Omega_k:=v_k+\Omega_0,\quad k=1,\ldots,M
\end{equation*}
are non-overlapping, i.e., their interiors are pairwise disjoint.

Finally we fix an invertible linear transformation $L$ of $\RR^N$, we introduce the lattice 
\begin{equation*}
\Lambda:=\bigcup_{j=1}^M L^*\left(u_j+\ZZ^N\right)\subset \RR^N
\end{equation*}
(here $L^*$ denotes the adjoint of $L$) and the set 
\begin{equation*}
\Omega:=L^{-1}\left(\Omega_1\cup\cdots\cup\Omega_M\right)\subset \RR^N.
\end{equation*}

\begin{remark}\label{r21}
Let us emphasize that the volume of $\abs{\Omega}$ of $\Omega$ does not depend on the particular choice of $M$ and the vectors $v_1,\ldots, v_M\in\RR^N$ satisfying (A1).
\end{remark}

We prove the following Ingham type generalization of Parseval's formula:

\begin{theorem}\label{t22}
Assume (A1) and (A2). \mbox{}

\begin{enumerate}[\upshape (i)]
\item There exist two positive constants $c_1, c_2$ such that 
\begin{equation}\label{23}
c_1\sum_{\lambda\in\Lambda}\abs{a_{\lambda}}^2 
\le \int_{\Omega}\abs{\sum_{\lambda\in\Lambda}a_{\lambda}e^{i(\lambda,x)}}^2\ dx 
\le c_2\sum_{\lambda\in\Lambda}\abs{a_{\lambda}}^2 
\end{equation}
for all square summable families $(a_{\lambda})_{\lambda\in\Lambda}$ of complex coefficients.

\item The estimates fail if we remove any non-empty open subset from $\Omega$.
\end{enumerate}
\end{theorem}

\begin{proof}
Let us first consider the case where $L$ is the identity map. 

Using (A1) we have 
\begin{align*}
\int_{\Omega}\abs{f(x)}^2\ dx 
&=\sum_{k=1}^M\int_{\Omega_0}\abs{f(v_k+x)}^2\ dx\\ 
&=\sum_{k=1}^M\int_{\Omega_0}\abs{\sum_{j=1}^Me^{i(u_j,v_k+x)}f_j(v_k+x)}^2\ dx\\
&=\sum_{k=1}^M\int_{\Omega_0}\abs{\sum_{j=1}^Me^{i(u_j,v_k)}\cdot e^{i(u_j,x)}f_j(x)}^2\ dx.
\end{align*}

Furthermore, by (A2) there exist two positive constants $c_1', c_2'$ such that 
\begin{multline*}
c_1'\sum_{j=1}^M\abs{e^{i(u_j,x)}f_j(x)}^2
\le \sum_{k=1}^M\abs{\sum_{j=1}^Me^{i(u_j,v_k)}\cdot e^{i(u_j,x)}f_j(x)}^2\\
\le c_2'\sum_{j=1}^M\abs{e^{i(u_j,x)}f_j(x)}^2,
\end{multline*}
or equivalently 
\begin{equation*}
c_1'\sum_{j=1}^M\abs{f_j(x)}^2
\le \sum_{k=1}^M\abs{\sum_{j=1}^Me^{i(u_j,v_k)}\cdot e^{i(u_j,x)}f_j(x)}^2
\le c_2'\sum_{j=1}^M\abs{f_j(x)}^2
\end{equation*}
for all $x$. 
Integrating over $\Omega_0$ and using the last equality we obtain the estimates 
\begin{equation*}
c_1'\sum_{j=1}^M\int_{\Omega_0}\abs{f_j(x)}^2\ dx
\le \int_{\Omega}\abs{f(x)}^2\ dx\le c_2'\sum_{j=1}^M\int_{\Omega_0}\abs{f_j(x)}^2\ dx
\end{equation*}
Since $\abs{\Omega}=M\abs{\Omega_0}$ by (A2), using \eqref{22} the estimates \eqref{23}  follow with $c_1=c_1'\abs{\Omega_0}$ and $c_2=c_2'\abs{\Omega_0}$. 

Now we show that the above estimates fail if we remove from $\Omega$ a non-empty open subset $\omega$.  
We may assume that $\omega\subset\Omega_k$ for some $k\in\set{1,\ldots,M}$.

Let $f\in L^2(\Omega)$ be the characteristic function of $\omega$. 
Thanks to Assumption (A2) the linear system
\begin{equation*}
f(v_k+x)=\sum_{j=1}^Me^{i(u_j,v_k)}e^{i(u_j,x)}f_j(x),\quad k=1,\ldots, M
\end{equation*}  
has a unique solution 
\begin{equation*}
e^{i(u_j,x)}f_j(x),\quad j=1,\ldots, M
\end{equation*}
for each $x\in\Omega_0$, and $f_1,\ldots, f_M\in L^2(\Omega_0)$. 
Extending the functions $f_j$ by $2\pi$-periodicity in each variable, we get \eqref{21} for each $j$ with square summable coefficients $a_{jk}$. 
Since, furthermore, 
\begin{equation*}
f(x)=\sum_{j=1}^Me^{i(u_j,x)}f_j(x)
\end{equation*}
by Assumption (A1), we conclude that
\begin{equation*}
f(x)=\sum_{j=1}^M\sum_{k\in\ZZ^N}a_{jk}e^{i(u_j+k,x)}
\end{equation*}
in $\Omega$. 

Since $\omega$ has a positive measure, the coefficients $a_{jk}$ do not vanish identically. 
On the other hand, 
\begin{equation*}
\int_{\Omega\setminus\omega}\abs{f(x)}^2\ dx=0,
\end{equation*}
so that the first estimate of \eqref{22} fails.

In order to complete the proof of the first part of the theorem it suffices to show that if the estimates \eqref{23} hold for some $\Lambda$ and $\Omega$, and $L$ is an invertible linear transformation of $\RR^N$, then the estimates 
\begin{equation*}
c_1\sum_{\lambda\in\Lambda}\abs{a_{\lambda}}^2 
\le \frac{1}{\abs{L^{-1}(\Omega)}}\int_{L^{-1}(\Omega)}\abs{\sum_{\lambda\in\Lambda}a_{\lambda}e^{i(L^*\lambda,x)}}^2\ dx 
\le c_2\sum_{\lambda\in\Lambda}\abs{a_{\lambda}}^2 
\end{equation*}
also hold. 
This follows from the change of variable formula: if $x=Lx'$, then 
\begin{equation*}
\int_{\Omega}\abs{\sum_{k\in\ZZ^N}a_ke^{i(k,x)}}^2\ dx=
\abs{\det L}\int_{L^{-1}(\Omega)}\abs{\sum_{k\in\ZZ^N}a_ke^{i(L^*k,x')}}^2\ dx',
\end{equation*}
where $\det L$ denotes the determinant of $L$, and 
\begin{equation*}
\abs{L^{-1}(\Omega)}=\frac{\abs{\Omega}}{\abs{\det L}}.
\end{equation*}

Since $L$ transforms non-empty open sets into non-empty open sets, the second part of the theorem also holds in the general case.
\end{proof}

\begin{remarks}\label{r23}\mbox{}

\begin{enumerate}[\upshape (i)]
\item A standard application of the triangle inequality implies that the assumptions (A1) and (A2) are not necessary for the second inequality in \eqref{23}.
\item The assumption (A2) is not necessary for Part (ii) of the theorem. 
This may be shown by taking a maximal subset of the vectors for which the corresponding columns of the matrix $E$ are linearly independent, and by completing this subset to a new set of vectors satisfying  (A1) and (A2).
\end{enumerate}
\end{remarks}

Given a lattice
\begin{equation}\label{24}
\Lambda:=\bigcup_{j=1}^M L^*\left(u_j+\ZZ^N\right)\subset \RR^N
\end{equation}
we may wonder whether we there exists another representation 
\begin{equation}\label{25}
\Lambda=\bigcup_{j=1}^{M_0} L_0^*\left(\tilde u_j+\ZZ^N\right)\subset \RR^N
\end{equation}
with another invertible matrix $L_0$ and a smaller integer $M_0$. 
As we will see in the rest of the paper choosing the minimal $M$ may substantially simplify the study of optimal integration domains. 

The following simple condition will allow us to determine the smallest $M$ in all but one of the examples in this work. 
Given two points $a,b\in\RR^N$, let us introduce the lattice
\begin{equation*}
\Lambda(a,b):=\set{a+k(b-a)\ : k\in\ZZ}
\end{equation*}
generated by $a$ and $b$.

\begin{lemma}\label{l24}
If there exist $M$ points $a_1,\ldots,a_M\in \RR^N$ such that $\Lambda(a_i,a_k)\not\subset\Lambda$ for all $i\ne k$, then the number $M$ in the representation \eqref{24} of $\Lambda$ is the smallest possible.
\end{lemma}

\begin{proof}
If two points $a_i$ and $a_k$ belong to the same set $L_0^*\left(\tilde u_j+\ZZ^N\right)$ in another representation \eqref{25}, then 
\begin{equation*}
\Lambda(a_i,a_k)\subset L_0^*\left(\tilde u_j+\ZZ^N\right)\subset\Lambda,
\end{equation*}
contradicting our hypothesis. 
Therefore each point $a_i$ corresponds to a different $j$, and thus $M\le M_0$.
\end{proof}

\section{Triangular and hexagonal lattices}\label{s3}

We illustrate Theorem \ref{t22} by two examples.

\subsection{Regular triangular lattice}\label{ss31}
Choosing\footnote{We write the vectors as row vectors but we consider them as column vectors in the computations with matrices.}
\begin{equation*}
N=2,\quad M=1,\quad u_1=v_1=(0,0)
\end{equation*}
and 
\begin{equation*}
L= 
\begin{pmatrix}
1&0\\ \frac{1}{2}&\frac{\sqrt{3}}{2}
\end{pmatrix}
\end{equation*}
(as usual, we identify the linear transformations with their matrices in the canonical basis of $\RR^2$),
\begin{equation*}
\Lambda=\set{\left(k_1+\frac{k_2}{2},\frac{\sqrt{3}k_2}{2}\right)\ :\ k\in\ZZ^2}
\end{equation*}
is a triangular lattice formed by equilateral triangles of unit side. Furthermore, since
\begin{equation*}
L^{-1}= 
\begin{pmatrix}
1&0 \\ 
\frac{-1}{\sqrt{3}}&\frac{2}{\sqrt{3}}
\end{pmatrix}
,
\end{equation*}
$\Omega=L^{-1}(\Omega_0)$ is a parallelogram of vertices 
\begin{equation*}
(0,0),\quad \left(2\pi,-\frac{2\pi}{\sqrt{3}}\right),\quad \left(0,\frac{4\pi}{\sqrt{3}}\right),\quad \left(2\pi,\frac{2\pi}{\sqrt{3}}\right).
\end{equation*}
Its area is equal to $\frac{8\pi^2}{\sqrt{3}}$; see Figure \ref{triangular_fig}. 
\begin{figure}
 \begin{subfigure}{0.4\textwidth}
  {\includegraphics[scale=0.4]{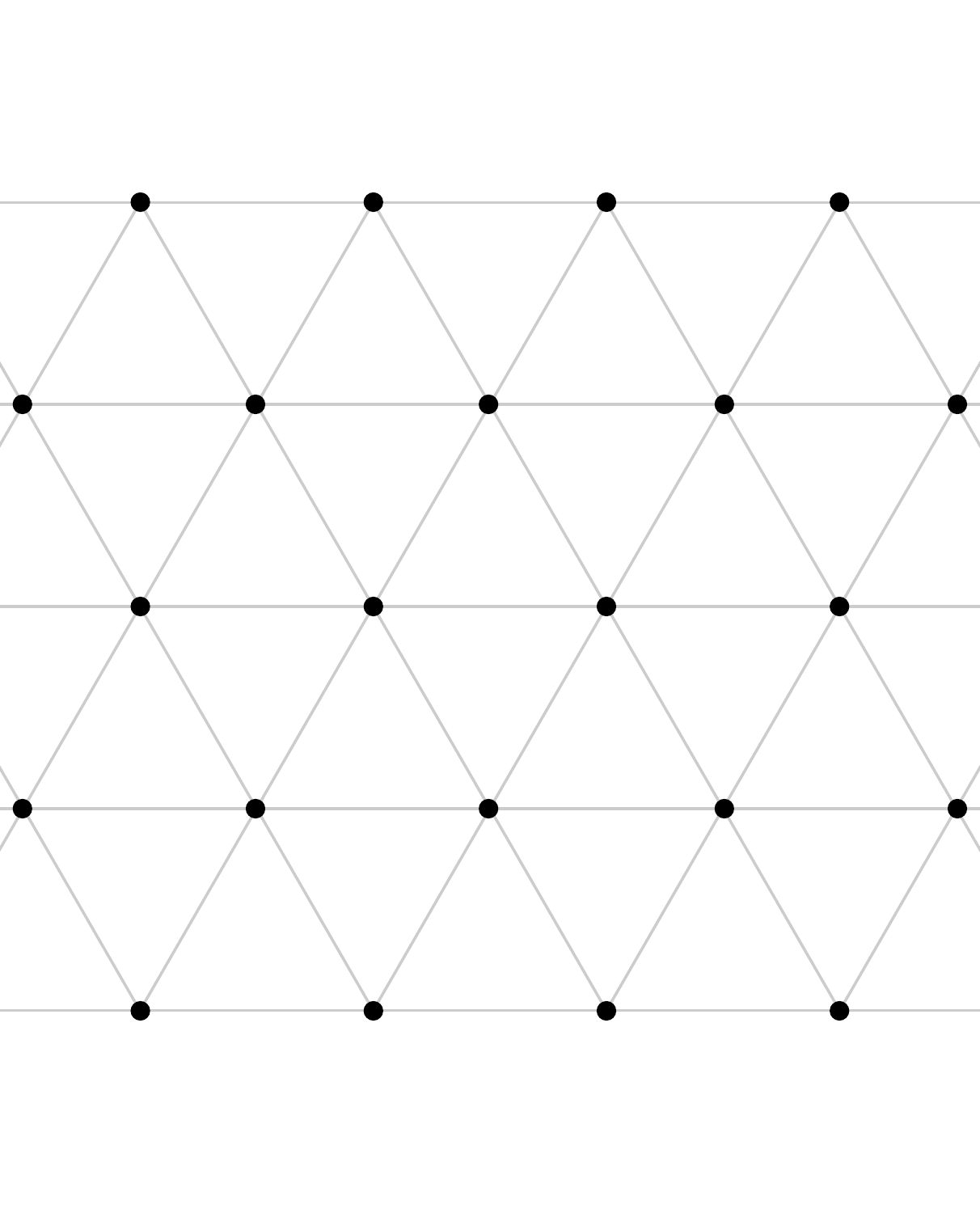}}
 \caption{$\Lambda$}
 \end{subfigure}
 \hskip1cm
 \begin{subfigure}{0.4\textwidth}
 {\includegraphics[scale=0.4]{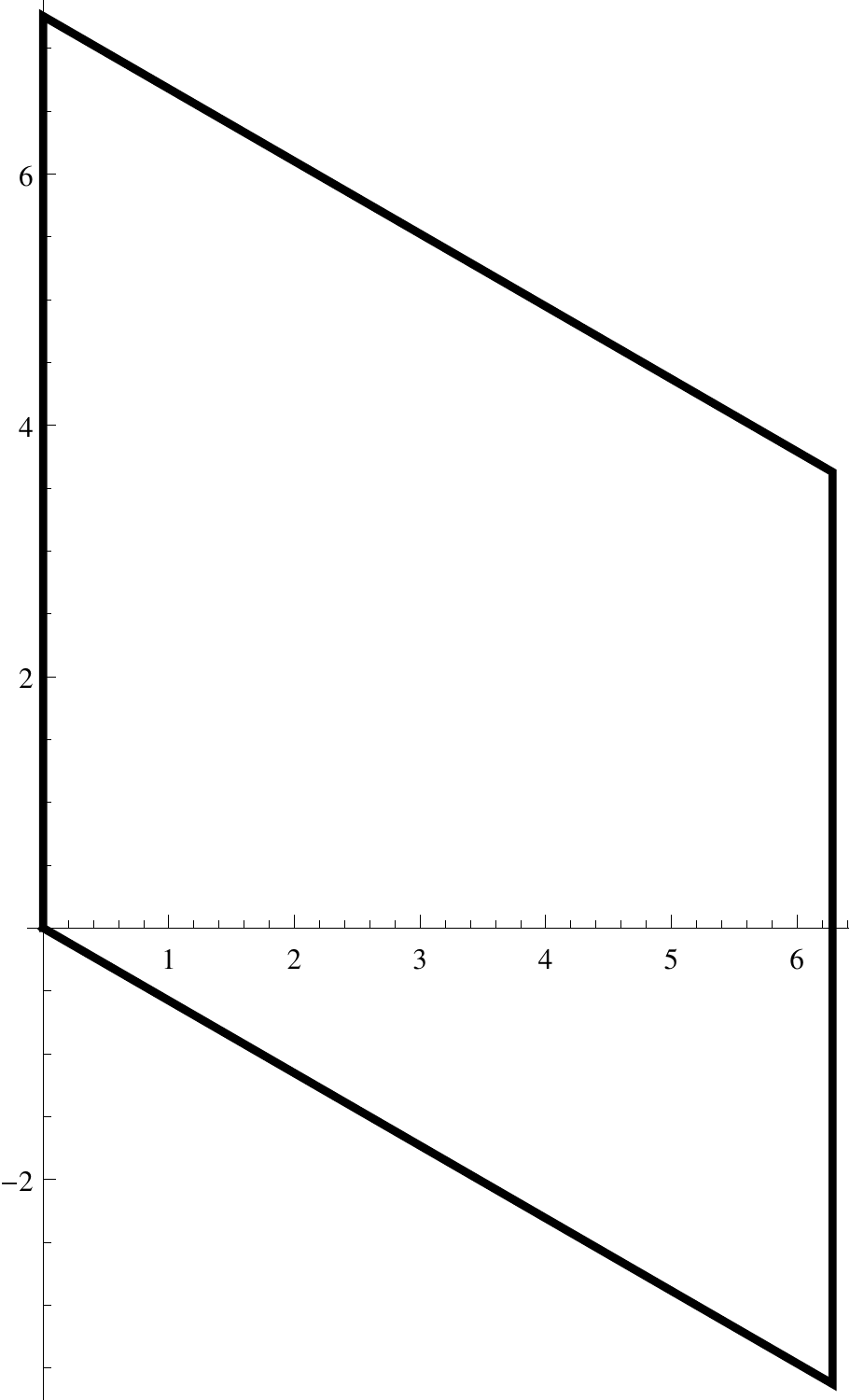}}
 \caption{$L^{-1}(\Omega_0)$}
 \end{subfigure}
   \caption{\label{triangular_fig}Hexagonal lattice} 
\end{figure}

Since every disk $D_R$ of radius
\begin{equation*}
R>\frac{1}{2}\diam(\Omega)=\frac{1}{2}\norm{ \left(0,\frac{4\pi}{\sqrt{3}}\right)-\left(2\pi,-\frac{2\pi}{\sqrt{3}}\right)}
=2\pi\approx 6.28
\end{equation*}
contains a translate of $\Omega$, Theorem \ref{t22} implies that if $R>2\pi$, then
\begin{equation}\label{31}
c_1(R)\sum_{\lambda\in\Lambda}\abs{a_{\lambda}}^2 
\le \int_{D_R}\abs{\sum_{\lambda\in\Lambda}a_{\lambda}e^{i(\lambda,x)}}^2\ dx 
\le c_2(R)\sum_{\lambda\in\Lambda}\abs{a_{\lambda}}^2 
\end{equation}
with two positive constants $c_1(R), c_2(R)$, for all square summable families $(a_{\lambda})_{\lambda\in\Lambda}$ of complex coefficients.

In fact, these estimates hold under the weaker condition $R>2\rho_2\approx 4.8096$, where $\rho_2\approx 2.4048$ denotes the smallest positive root of the Bessel function $J_0(x)$. 
This follows by applying \cite[Theorem 8.1]{KomLorbook} and a following remark on the same page with $p=2$ and $\gamma=1$. 

On the other hand, it follows from Remark \ref{r21} and Remark \ref{r23} (ii) that if \eqref{31} holds for some disk $D_R$ of radius $R$, then the area of this disk is bigger than equal to the area of $\Omega$: 
\begin{equation*}
R^2\pi\ge \frac{8\pi^2}{\sqrt{3}}\Longleftrightarrow R\ge \frac{2\sqrt{2\pi}}{3^{1/4}}\approx 3.8.
\end{equation*}
Indeed, a smaller disk could be covered by a set 
\begin{equation*}
\Omega:=L^{-1}\left(\Omega_1\cup\cdots\cup\Omega_M\right)\subset \RR^2
\end{equation*}
for a sufficiently large number of vectors $v_1,\ldots, v_M$ satisfying (A1).

It would be interesting to determine the critical radius $R$ for the validity of \eqref{31}. 

\subsection{Regular hexagonal lattice}\label{ss32}
Now we choose
\begin{equation*}
N=M=2,\quad u_1=(0,0),\quad u_2=(2/3,-1/3)
\end{equation*}
and
\begin{equation*}
L=\sqrt{3}
\begin{pmatrix}
\frac{\sqrt{3}}{2}&\frac{1}{2}\\0&1
\end{pmatrix}
.
\end{equation*}
Now $\Lambda$ is the honeycomb lattice of unit side, see Figure \ref{honey}.
Furthermore, since 
\begin{equation*}
L^{-1}
=
\frac{1}{\sqrt{3}}
\begin{pmatrix}
\frac{2}{\sqrt{3}}&\frac{-1}{\sqrt{3}}\\0&1
\end{pmatrix}
,
\end{equation*}
$L^{-1}(\Omega_0)$ is the parallelogram of vertices
\begin{equation*}
(0,0),\quad \left(-\frac{2\pi}{3},\frac{2\pi}{\sqrt{3}}\right),\quad \left(\frac{2\pi}{3},\frac{2\pi}{\sqrt{3}}\right),\quad \left(\frac{4\pi}{3},0\right);
\end{equation*}
see Figure \ref{hexagonal_fig}.

If we choose $v_1=(0,0)$ and $v_2=(2\pi,0)$, then the conditions (A1) and (A2) are satisfied because 
\begin{equation*}
\dete E=
\begin{vmatrix}
1&1\\
1&e^{i \frac{4 \pi }{3}}
\end{vmatrix}\not=0.
\end{equation*}
Furthermore, 
\begin{equation*}
\Omega=L^{-1}(\Omega_0\cup(\Omega_0+v_2))
\end{equation*}
is the parallelogram of vertices
\begin{equation*}
(0,0),\left(-\frac{2\pi}{3},\frac{2\pi}{\sqrt{3}}\right),\left({2\pi},\frac{2\pi}{\sqrt{3}}\right),\left(\frac{8\pi}{3},0\right);
\end{equation*}
its area of the latter one is equal to $\frac{16\pi^2}{3\sqrt{3}}$.
See Figure \ref{hexagonal_fig}.

\begin{figure}
  {\includegraphics[scale=0.4]{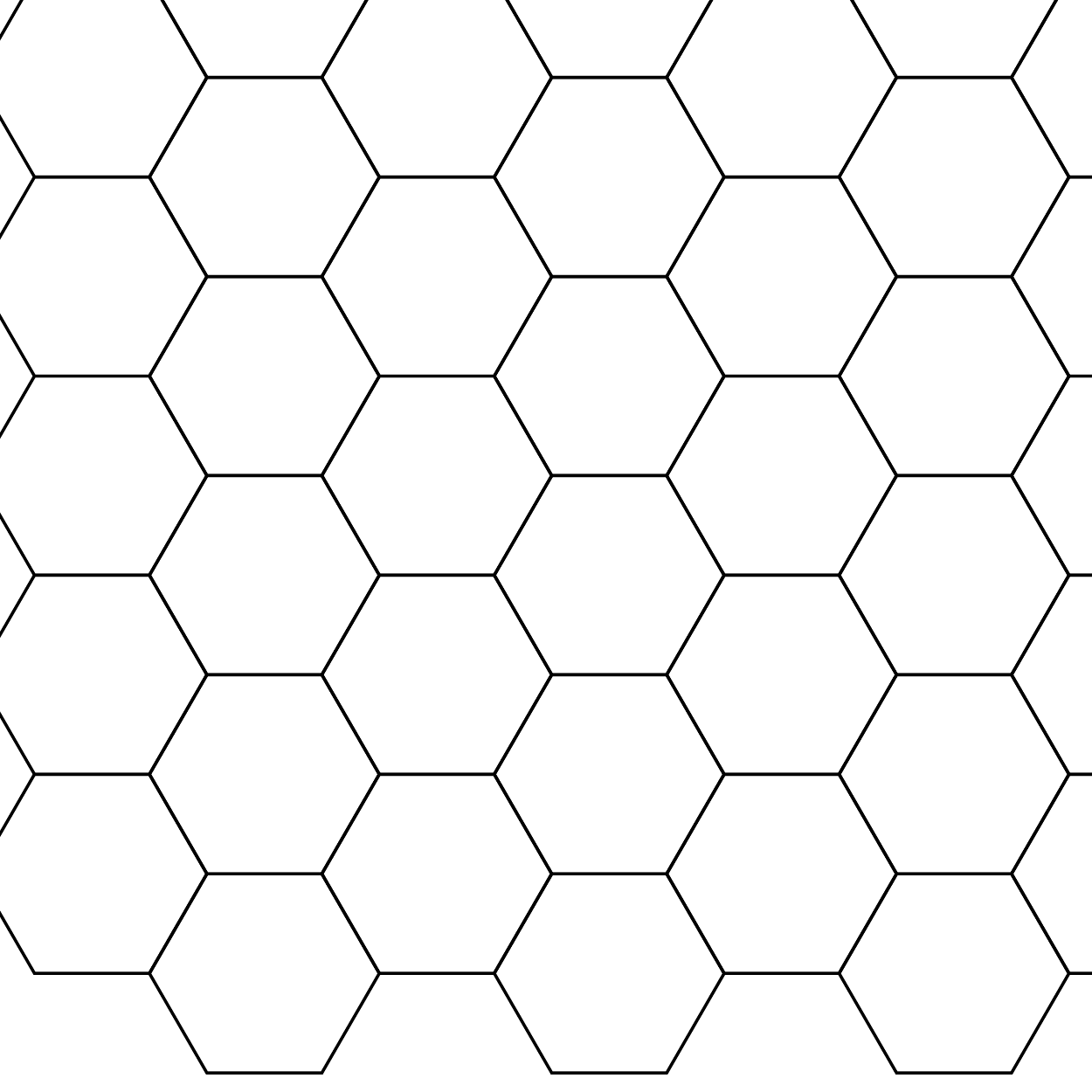}}
 \caption{The honeycomb lattice $\Lambda$\label{honey}}
\end{figure}
\begin{figure}
 {\includegraphics[scale=0.6]{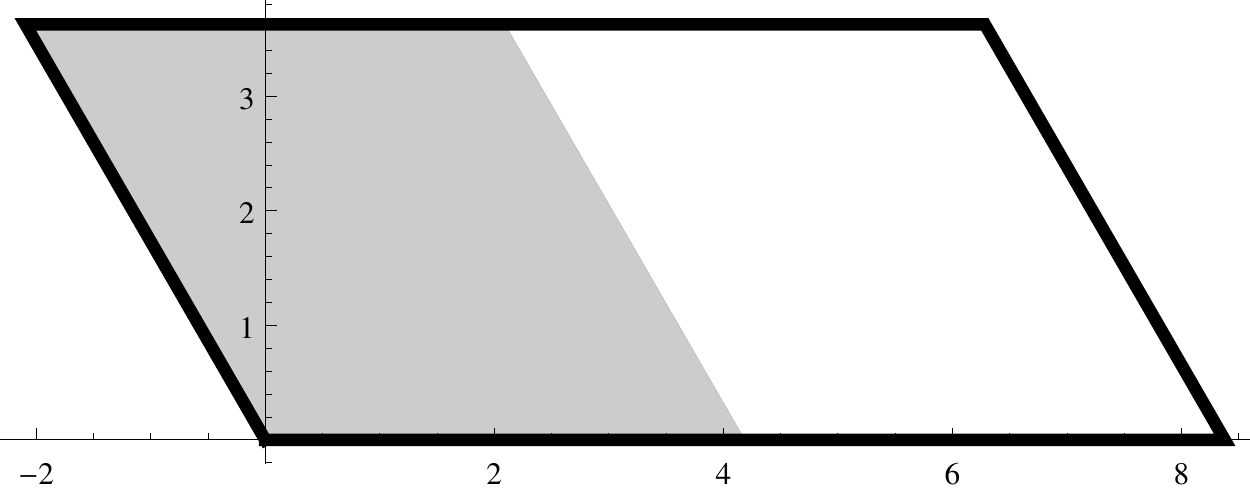}}
  \caption{Domain $\Omega$ associated to the honeycomb lattice and to $v_2=(2\pi,0)$, the shadowed area corresponds to its subset $L^{-1}(\Omega_0)$.\label{hexagonal_fig}} 
\end{figure}
 
\begin{remark}\label{r31}
If we compare the case of the parallelogram lattice and the hexagonal lattice, then we see that the integration parallelogram 
is $1.5$ times larger in the first case. 
This corresponds to the fact that the density of the corresponding lattice is also $1.5$ times larger.
\end{remark}

Since every disk $D_R$ of radius
\begin{equation*}
R>\frac{1}{2}\diam(\Omega)=\frac{1}{2}\norm{\left(-\frac{2\pi}{3},\frac{2\pi}{\sqrt{3}}\right)- \left(\frac{8\pi}{3},0\right)}
=\frac{2\pi\sqrt{7}}{3}\approx 5.54
\end{equation*}
contains a translate of $\Omega$, Theorem \ref{t22} implies that if $R>2\pi\sqrt{7}/3$, then the estimates \eqref{31} hold with two positive constants $c_1(R), c_2(R)$, for all square summable families $(a_{\lambda})_{\lambda\in\Lambda}$ of complex coefficients.

If we choose $v_1=(0,0)$ and $v_2=(0,2\pi)$ instead, then the conditions (A1) and (A2) are still satisfied because the matrix $E$ remains the same:
\begin{equation*}
E=
\begin{pmatrix}1&1\\
    1&e^{-i \frac{2 \pi }{3}}
\end{pmatrix}
=
\begin{pmatrix}
    1&1\\
    1&e^{i \frac{4 \pi }{3}}
\end{pmatrix}
.
\end{equation*}
Now
\begin{equation*}
\Omega=L^{-1}(\Omega_0\cup(\Omega_0+v_2))
\end{equation*}
is the parallelogram of vertices
\begin{equation*}
(0,0),\left(-\frac{4\pi}{3},\frac{4\pi}{\sqrt{3}}\right),\left(0,\frac{4\pi}{\sqrt{3}}\right),\left(\frac{4\pi}{3},0\right);
\end{equation*}
its area is still equal to $\frac{16\pi^2}{3\sqrt{3}}$.
See Figure \ref{hexagonal_fig2}.
\begin{figure}
 {\includegraphics[scale=0.6]{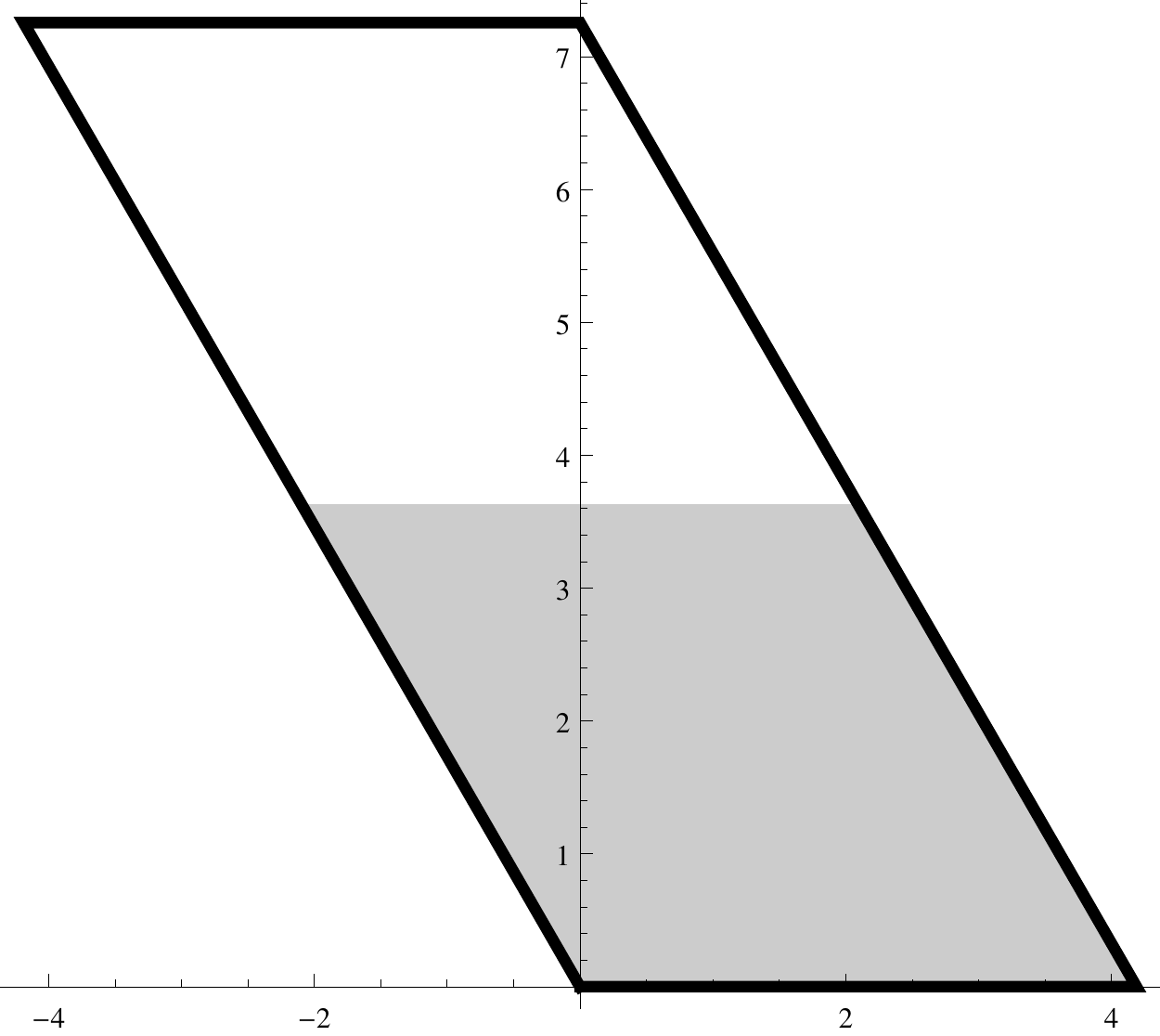}}
  \caption{Domain $\Omega$ associated to the honeycomb lattice and to $v_2=(0,2\pi)$, the shadowed area corresponds to its subset $L^{-1}(\Omega_0)$.\label{hexagonal_fig2}} 
\end{figure}
 
Since 
\begin{equation*}
\frac{1}{2}\diam(|\Omega|)=\frac{1}{2}\norm{\left(-\frac{4\pi}{3},\frac{4\pi}{\sqrt{3}}\right)- \left(\frac{4\pi}{3},0\right)}=\frac{2\pi\sqrt{7}}{3} \approx 5.54,
\end{equation*}
we obtain the same condition for the validity of \eqref{31} as before.

As in the preceding case, we may apply \cite[Theorem 8.1]{KomLorbook} $p=2$ and $\gamma=1$ to conclude that the estimates \eqref{31} hold under the weaker condition $R>2\rho_2\approx 4.8096$.
This also follows from the fact that the hexagonal lattice is a sublattice of the triangular one.

On the other hand, the validity of \eqref{31} implies that 
\begin{equation*}
R^2\pi\ge \frac{16\pi^2}{3\sqrt{3}}\Longleftrightarrow R\ge \frac{4\sqrt{\pi}}{3^{3/4}}\approx 3.11.
\end{equation*}

It would be interesting to determine the critical radius $R$ for the validity of \eqref{31}. 

\remove{Also notice that the square root of the first eigenvalue of the Laplacian on the hexagon of unit edge is $\sqrt{\mu}\sim 2.67$, see \cite{hex1} and \cite{hex2},
(then the first eigenvalue of the Laplacian on the hexagon of edge equal 1/2 is $5.34$). }

\section{Tiling of the plane by two different squares}\label{s4}

Let us consider the tiling of $\RR^2$ with two squares of different sides $R>r$ as shown on the Figure \ref{fig_twosquares}.
\begin{figure}
 \includegraphics[scale=0.5]{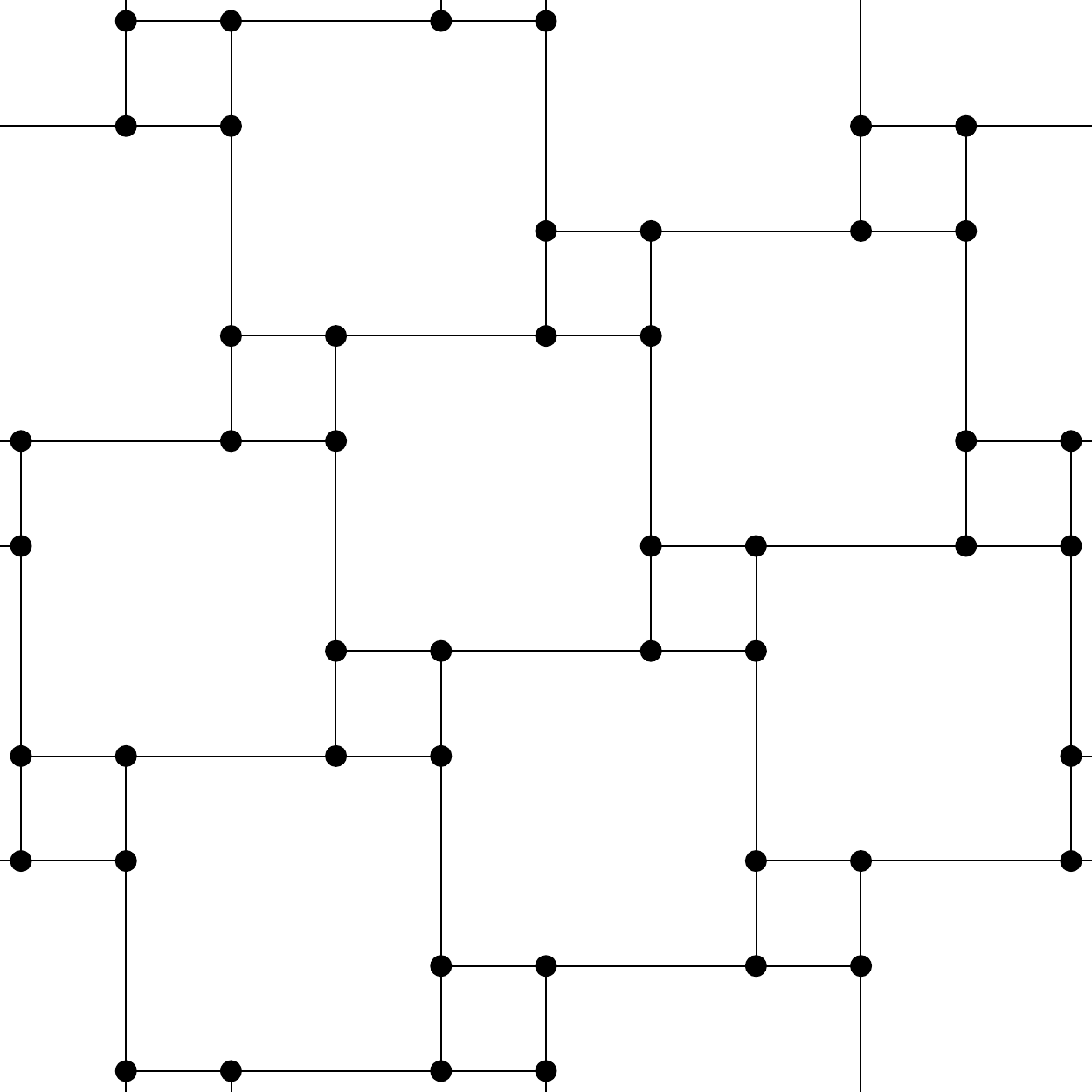}
 \caption{\label{fig_twosquares} Tiling by squares of side $r=1$ and $R=3$}
\end{figure}

Translating and rotating the tiling such that segments connecting the centers of the closest small squares are parallel to the coordinate axes and that the origin is one of these centers, we have 
\begin{equation*}
\Lambda=L^*\left(\sum_{j=1}^4\left(u_j+\ZZ^2\right)\right)
\end{equation*}
where $L$ is the homothety of coefficient $\sqrt{R^2+r^2}$, and the vectors $u_j$ are defined by the formulas $\alpha:=\arctan \frac{r}{R}$ and
\begin{equation*}
u_j=\frac{r}{\sqrt{2(R^2+r^2)}}(\cos (-\alpha+j\pi/2),\cos (-\alpha+j\pi/2)),\quad j=1,2,3,4.
\end{equation*}
See Figure \ref{fig_twosquares_details}.
\begin{figure}
 \begin{subfigure}{0.4\textwidth}
{\includegraphics[scale=0.4]{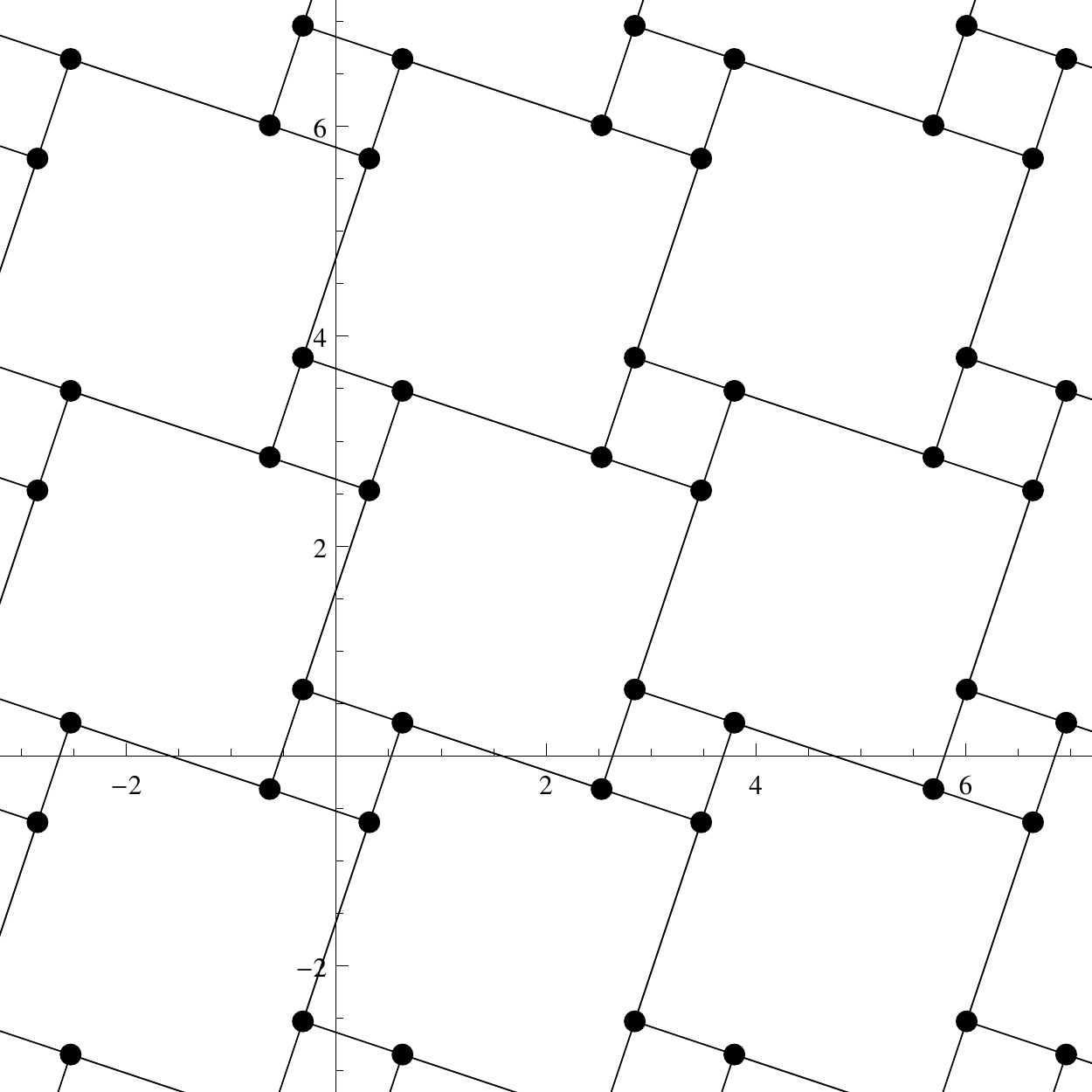}}\caption{}
   \end{subfigure}\hskip1cm
 \begin{subfigure}{0.4\textwidth}
{\includegraphics[scale=0.4]{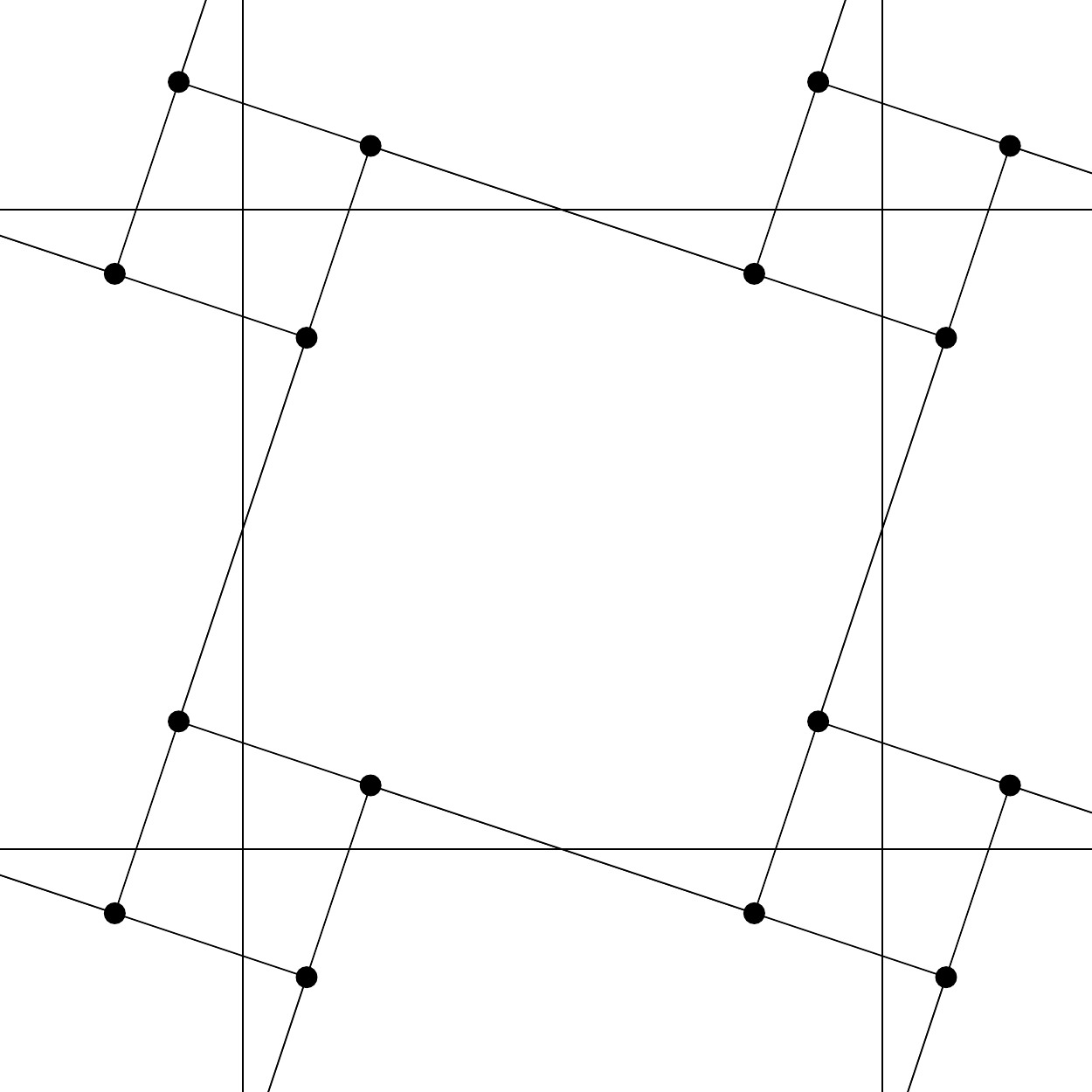}}\caption{}
   \end{subfigure}
 \caption{Geometric construction of the decomposition of $\Lambda$: the lattice is rotated by the angle $\alpha$ (see (A)) so that the centers of the small squares lay on $M\ZZ^2$ -- see (B). 
 The translation vectors $u_1,\dots,u_4$ are the vertices of the small square centered in the origin. \label{fig_twosquares_details}}
\end{figure}

Choosing 
\begin{equation*}
\set{v_1,\ldots, v_4}:=2\pi\set{(0,0), (1,0), (0,1), (1,1)}
\end{equation*}
the conditions (A1), (A2) are satisfied (see Remark \ref{rmkA2} below) and Theorem \ref{t22} may be applied with 
\begin{equation*}
\Omega=\left(0,\frac{4\pi}{\sqrt{R^2+r^2}}\right)^2.
\end{equation*}
For some examples of domains satisfying (A2) see Figures \ref{r2}--\ref{r5}. 
\begin{remark}\label{rmkA2}
Fix $r<R$ and let $A=\frac{r}{\sqrt{2(R^2+r^2)}}$ so that 
\begin{equation*}
e^{i (u_j ,v_k)}=e^{iA2\pi \cos(-\alpha+j)v_k^{(1)}+ \sin(-\alpha+j)v_k^{(2)}}.
\end{equation*}
Let $C:=e^{A i \pi \cos \alpha}$ and $D:= e^{A i \pi \sin \alpha}$. 
We have 
\begin{equation*}
E:=(e^{i (u_j ,v_k)})_{j,k}^2=
\begin{pmatrix}
1&1&1&1\\
C^{-1}&D&C&D^{-1}\\
D&C&D^{-1}&C^{-1}\\
DC^{-1}&CD&CD^{-1}&(CD)^{-1}
\end{pmatrix}
\end{equation*}
and 
\begin{equation*}
\Delta:=\dete E=(C^2-1)(D^2-1)(C^2D^2-4CD+C^2+D^2+1)
\end{equation*}
by a direct computation.

Since $\alpha\in(0,\pi/2)$ by definition, we have $A\cos\alpha, A\sin\alpha\in(0,1)$, and thus $C^2,D^2\not=1$. 
In order to prove $\Delta\not=0$, it suffices to show that
\begin{equation*}
C^2D^2-4CD+C^2+D^2+1\not=0.
\end{equation*}
We will show that even the imaginary part of this expression is different from zero.

Setting $\beta=\pi\cos\alpha$ and $\gamma=\pi\sin\alpha$ we have
\begin{align*}
\Im(C^2D^2&-4CD+C^2+D^2+1)\\&=\sin(2\beta+2\gamma)-4 \sin(\beta+\gamma)+\sin(2\beta)+\sin(2\gamma)\\
&=2\sin(\beta+\gamma)\cos(\beta+\gamma)-4 \sin(\beta+\gamma)+2\sin(\beta+\gamma)\cos(\beta-\gamma)\\
&=2\sin(\beta+\gamma)(\cos(\beta+\gamma)-2 +\cos(\beta-\gamma))\\
&=4\sin(\beta+\gamma)(\cos\beta\cos\gamma-1).
\end{align*}
Since $\cos\alpha, \sin\alpha\in(0,1)$, we have
\begin{equation*}
\cos \alpha+\sin \alpha\not=1\quad\text{and}\quad \cos\beta,\cos\gamma\in(-1,1).
\end{equation*}
They imply the inequalities
\begin{equation*}
\sin(\beta+\gamma)\not=0\quad\text{and}\quad \cos\beta\cos\gamma-1\not=0,
\end{equation*}
respectively.
This concludes the proof. 
\end{remark}

\begin{figure}
 \includegraphics[scale=0.4]{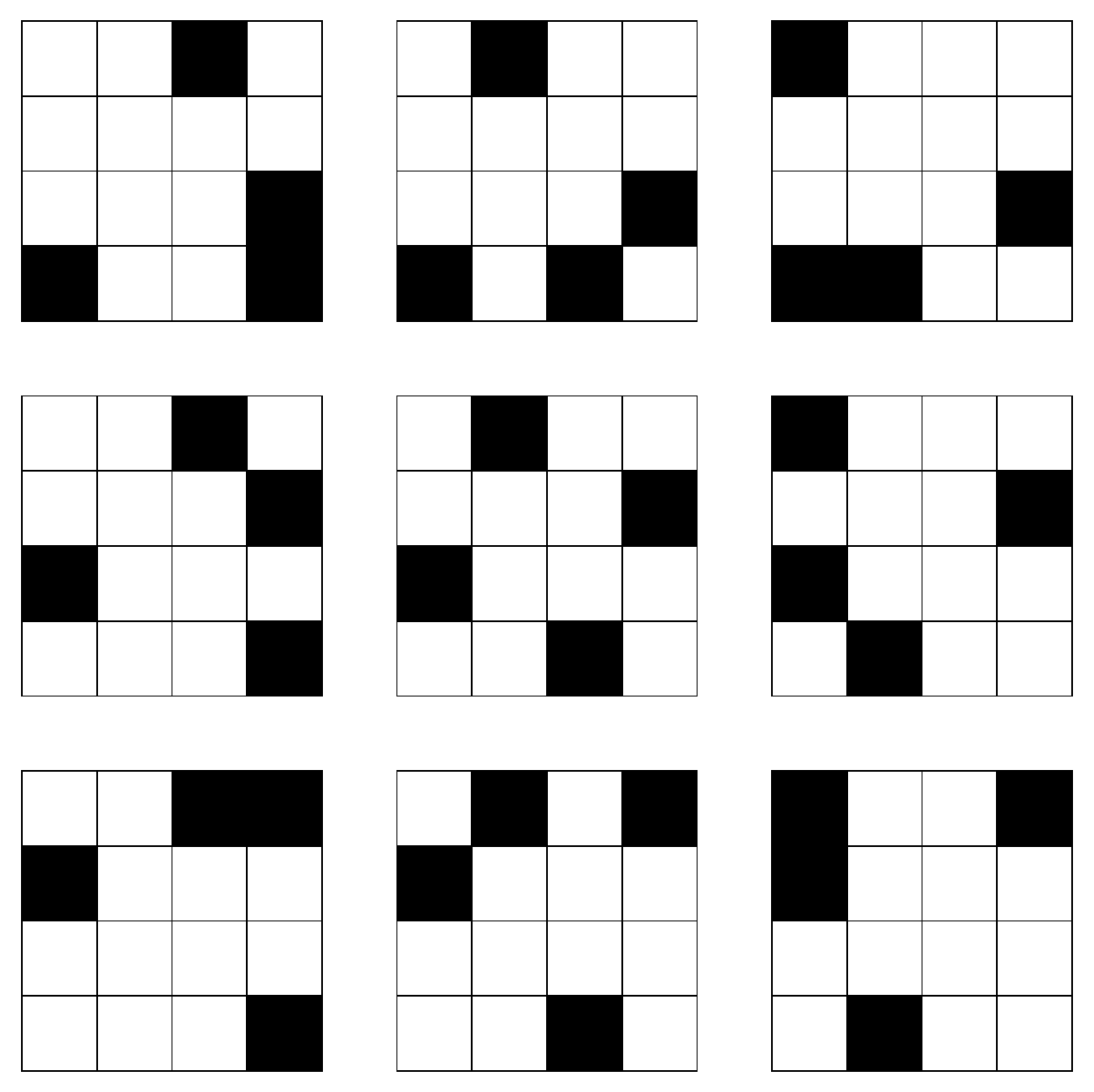}
\caption{Two square tilings: the $9$ (over 1820) domains  of the form $\cup_{k=1}^4 \Omega_0+v_k$ with $(v_k)$ not satisfying condition (A2) when $r=1$ and $R=2$.\label{r2}}
 \end{figure}

\begin{figure}
 \includegraphics{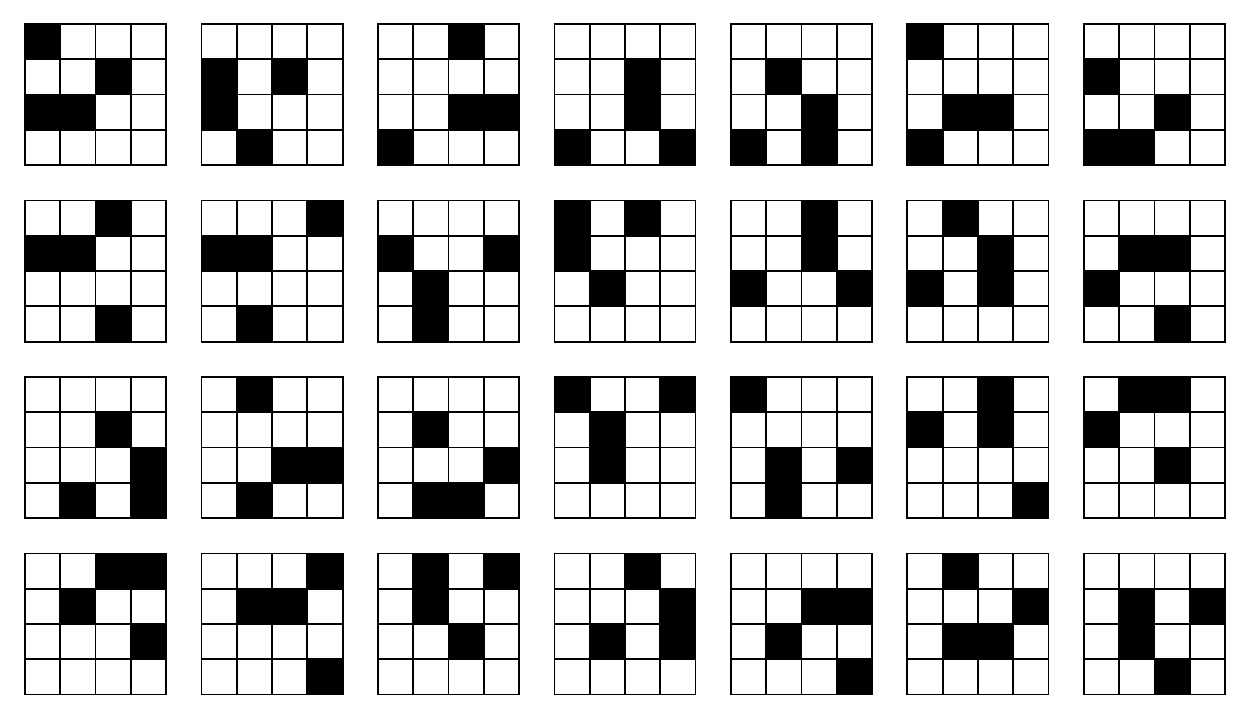}
\caption{Two square tilings: the $28$ (over 1820) domains  of the form $\cup_{k=1}^4 \Omega_0+v_k$ with $(v_k)$ not satisfying condition (A2) when $r=1$ and $R=3$.\label{r3}}
 \end{figure}

\begin{figure}
 \includegraphics{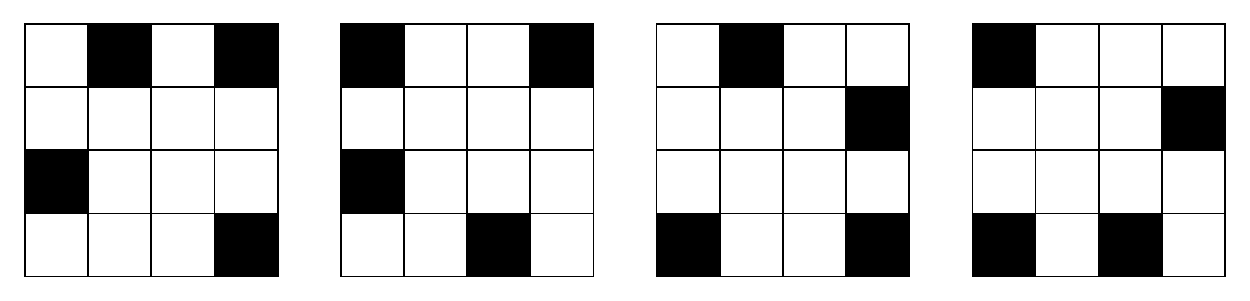}
\caption{Two square tilings: the $4$ (over 1820) domains  of the form $\cup_{k=1}^4 \Omega_0+v_k$ with $(v_k)$ not satisfying condition (A2) when $r=1$ and $R=5$.
(All domains of the form $\cup_{k=1}^4 \Omega_0+v_k$ satisfy condition $(A2)$ when $r=1$ and $R=4$.)
\label{r5}}
 \end{figure}

\begin{remark}\label{r41}
The result is not true in the limiting cases $r=0$ and  $r=R$ because then many lattice points collide.
\end{remark}

\section{Semi-regular bidimensional tilings}
\subsection{Elongated triangular tiling}\.\\
\begin{figure}[h!]
\includegraphics[scale=0.3]{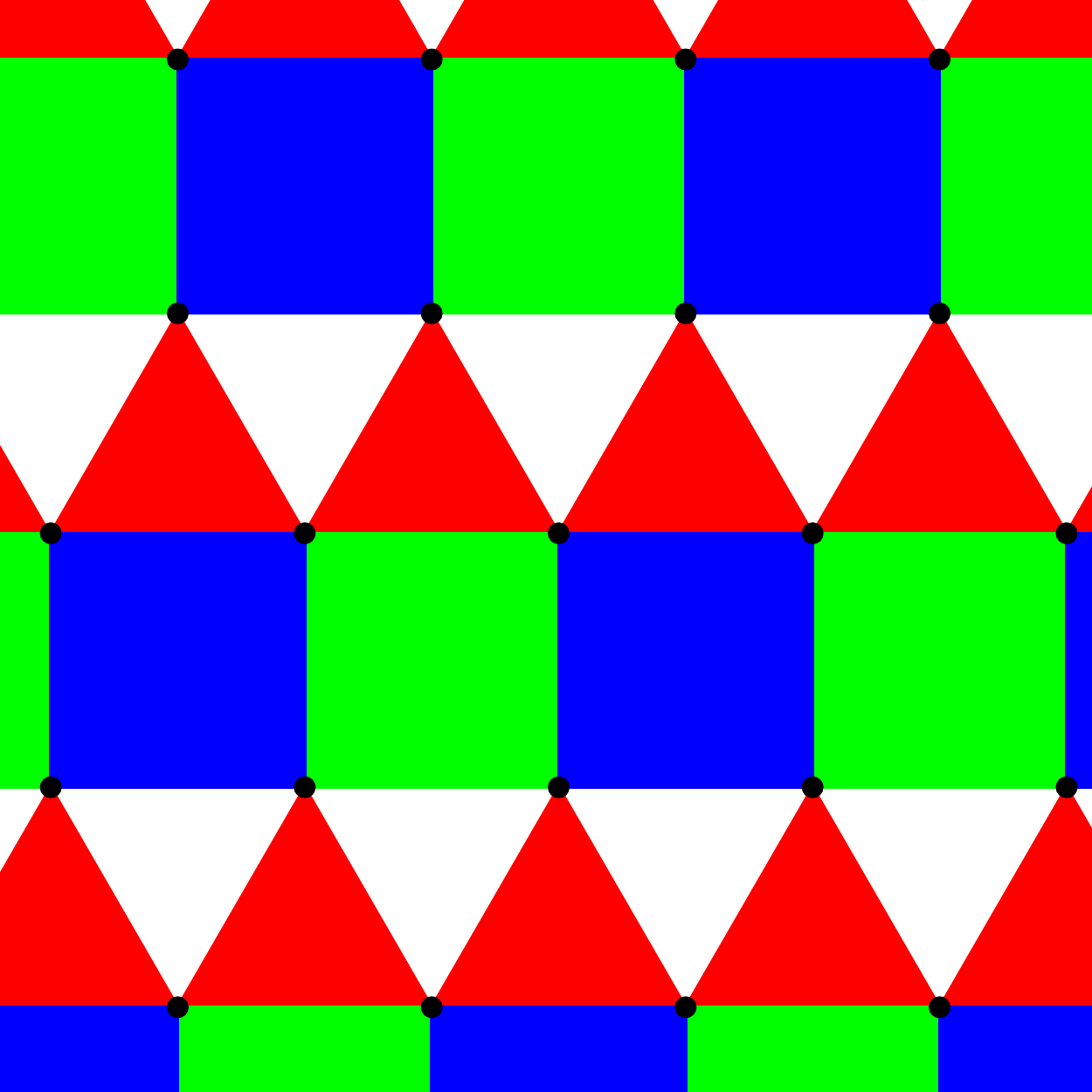}\hskip0.5cm\includegraphics[scale=0.3]{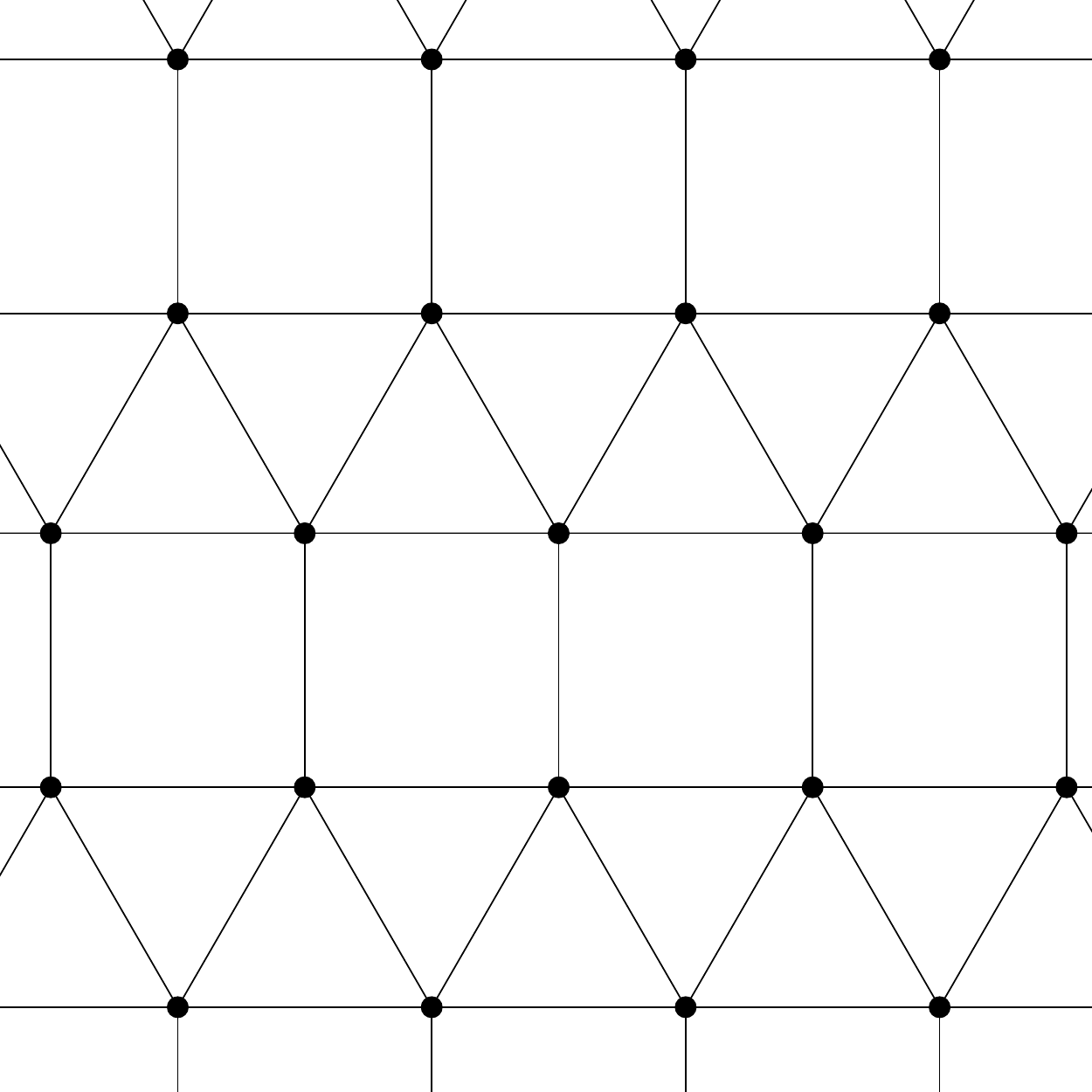}\\
 \includegraphics[scale=0.3]{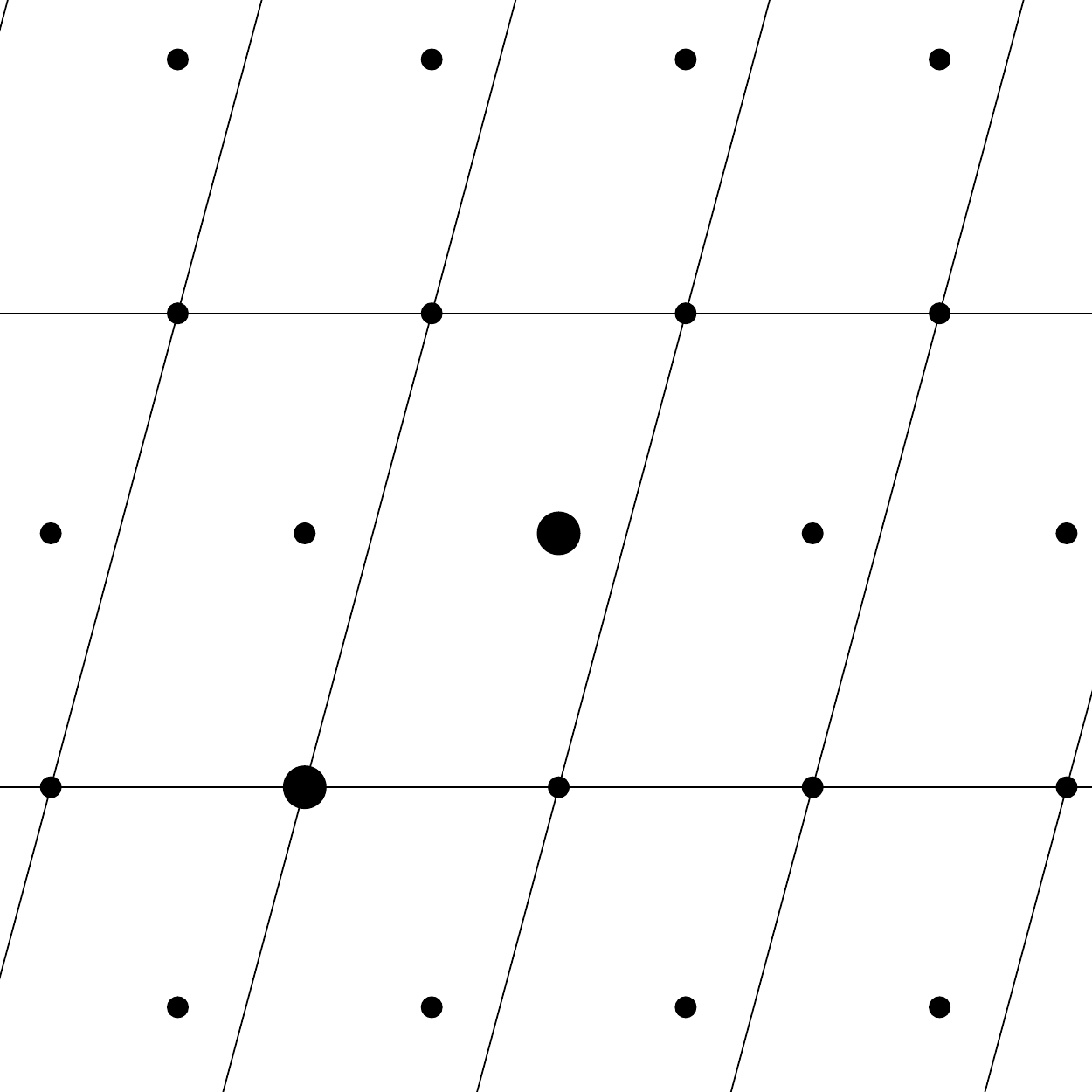}\hskip0.5cm\includegraphics[scale=0.3]{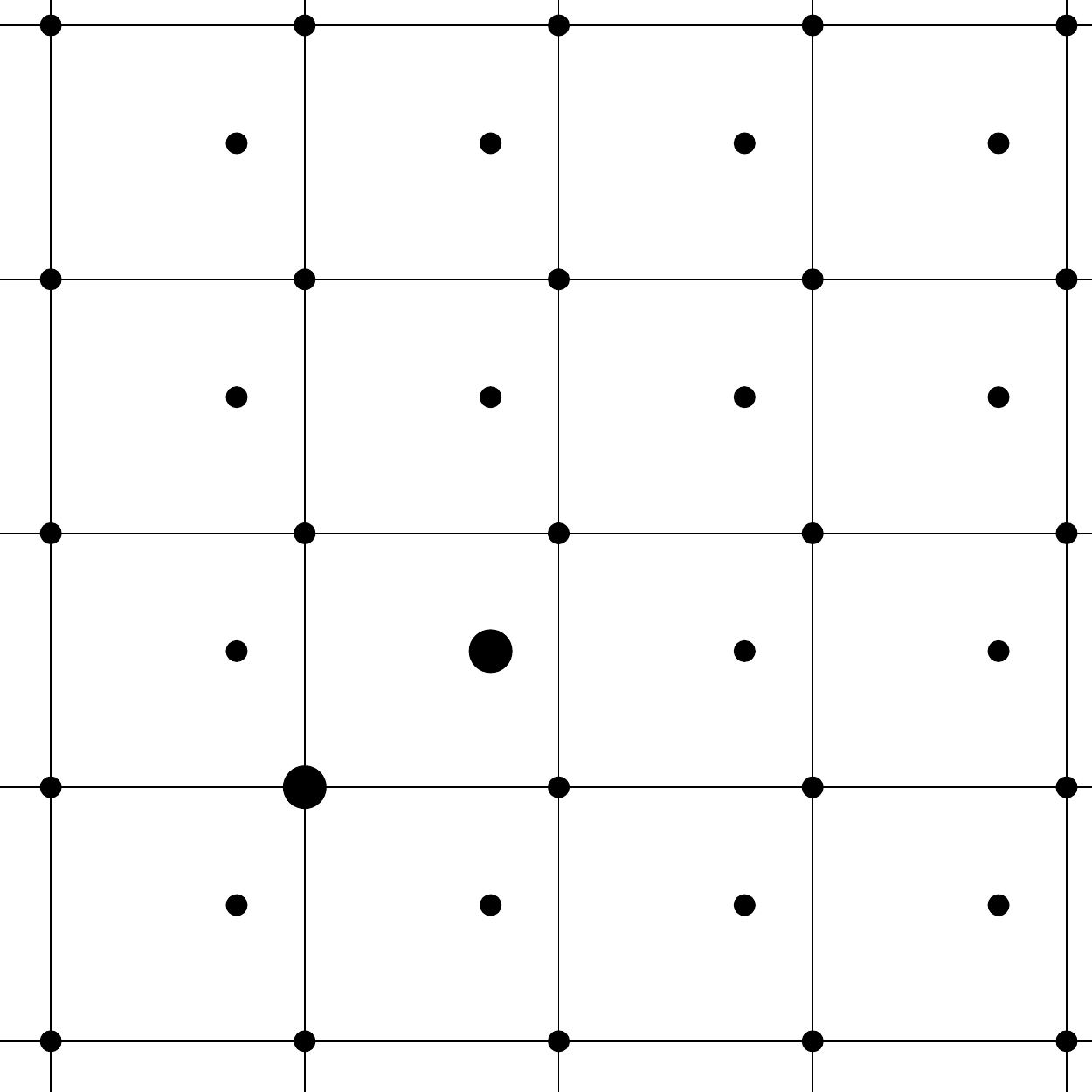}\\
  \caption{Elongated triangular tiling}\end{figure}

  $\Lambda=\cup_{j=1}^2 L^*(u_j+ \ZZ^2)$
 where 
  $$
\begin{array}{l}
u_1=( 0 , 0 )\\
 u_2=(-1+\sqrt{3} , 4-2\sqrt{3})\\
\end{array}\qquad
\text{and}\qquad
L^*=\left(
\begin{array}{cc}
 1 & \frac{1}{2} \\
 0 & 1+\frac{\sqrt{3}}{2} \\
\end{array}
\right).$$
 See Figure \ref{elongated} for the list of domains contained in $[0,2]^2$ (up to translation) and satisfying condition (A2).
\captionsetup[subfig]{\hmargin=0,\hpar=0,\checksingleline=true,}
\begin{figure}[h!]\begin{center}\hskip-0.5cm
\begin{subfigure}{0.27\textwidth}\captionsetup{width=1.2\linewidth,labelsep=newline,parindent=0cm,justification=centering}
  {\hskip0.7cm\centering \includegraphics[width=0.7\linewidth]{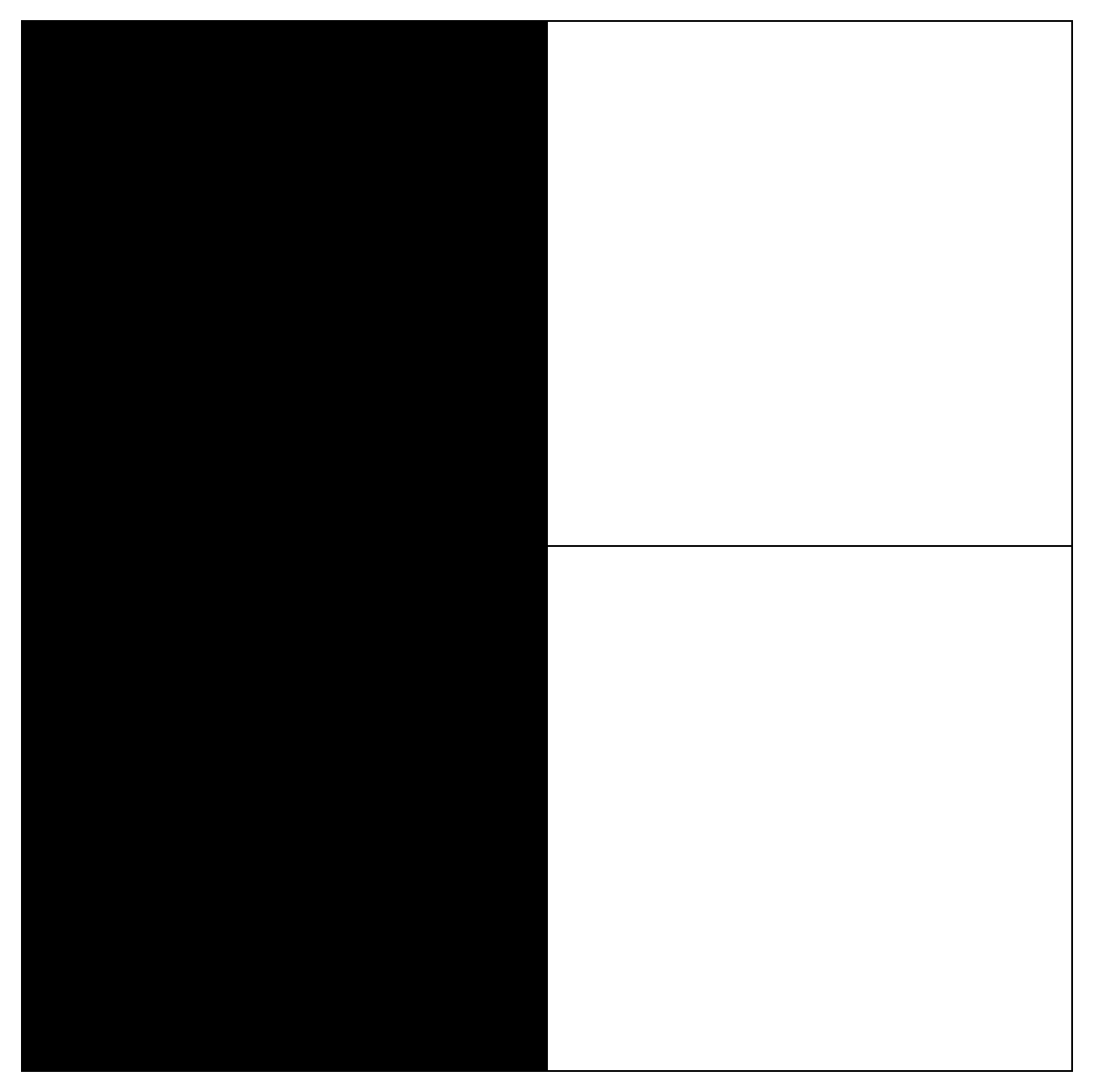}
   \caption{\small$c_1=1.77$; $c_2=2.22$}}
  \end{subfigure}
 \begin{subfigure}{0.27\textwidth}\captionsetup{width=1.2\linewidth,labelsep=newline,parindent=0cm,justification=centering}
 {\hskip1cm\centering \includegraphics[width=0.7\linewidth]{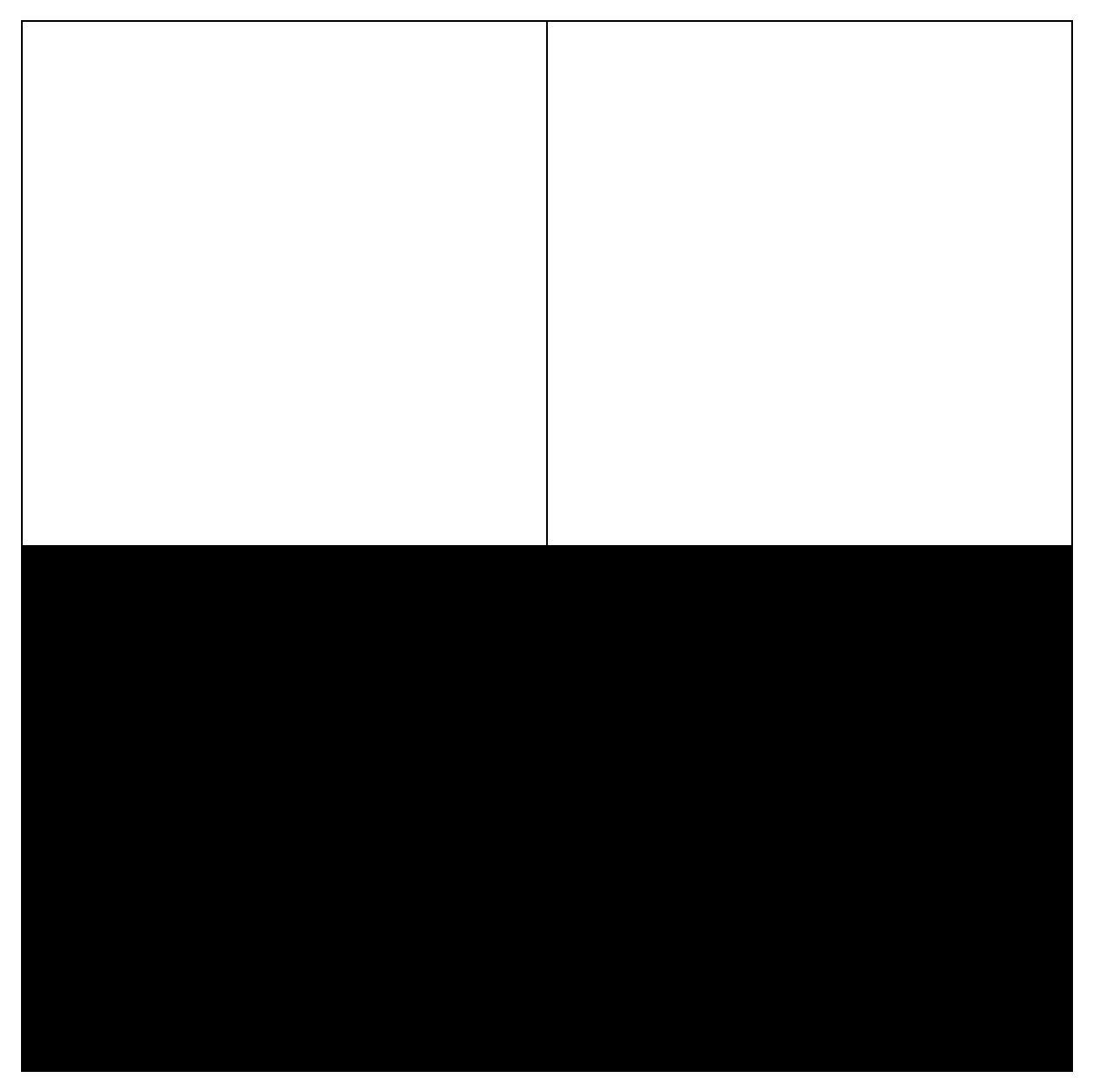}}\hskip0.2cm
 \caption{$c_1=0.66$; $c_2=3.33$}
\end{subfigure} 
\begin{subfigure}{0.27\textwidth}\captionsetup{width=1.2\linewidth,labelsep=newline,parindent=0cm,justification=centering}
{\hskip0.7cm\centering \includegraphics[width=0.7\linewidth]{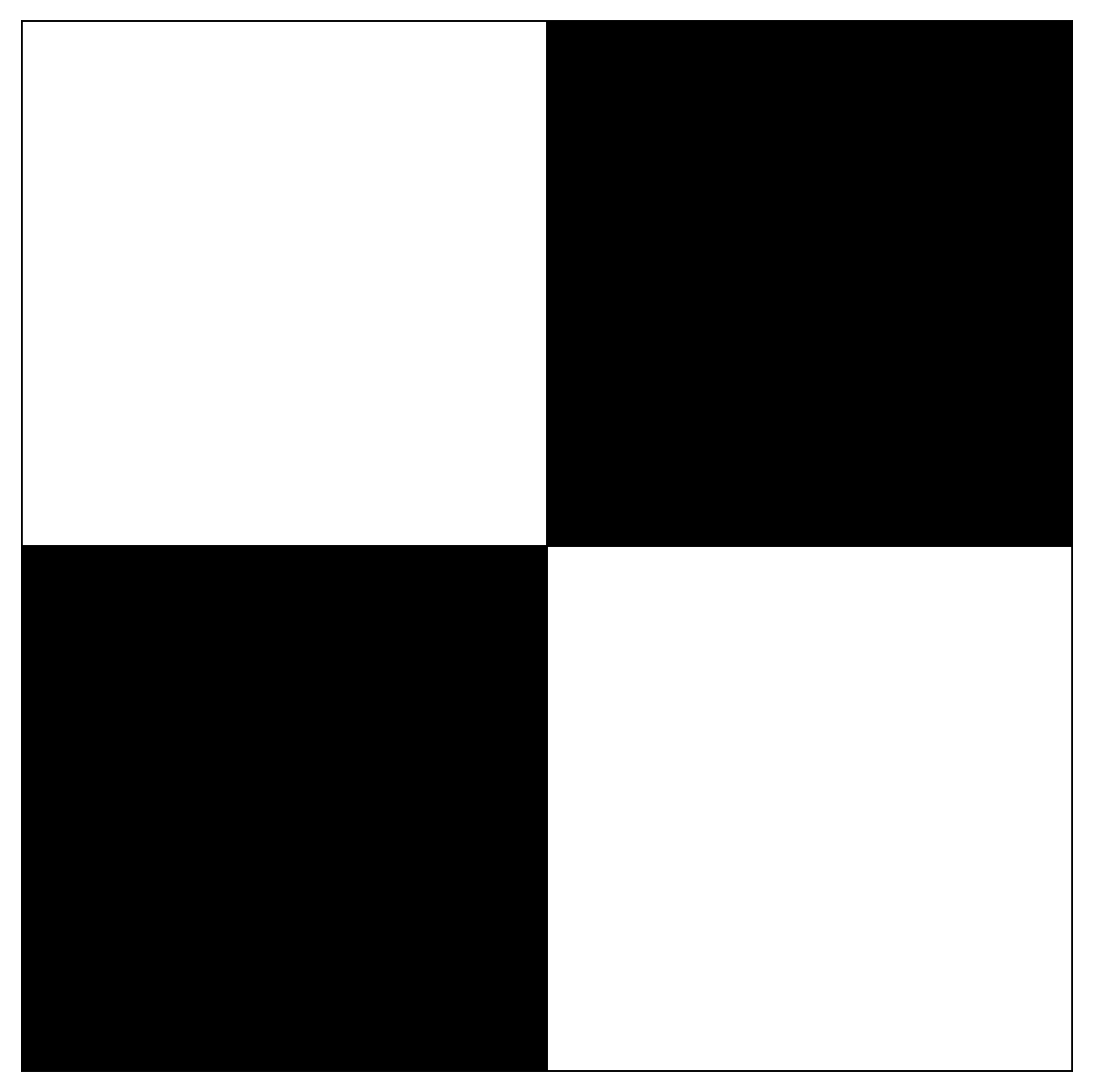}}
\caption{$c_1=0.36$; $c_2=3.63$}
\end{subfigure}
\end{center}
\caption{\label{elongated} Domains contained in $[0,2]^2$ (up to translations) satisfying condition (A2) and related constants $c_1$ and $c_2$, sorted by the increasing value of the ratio $c_2/c_1$}
\end{figure}
 
\subsection{Trihexagonal tiling}\.\\
\begin{figure}[h!]
\includegraphics[scale=0.3]{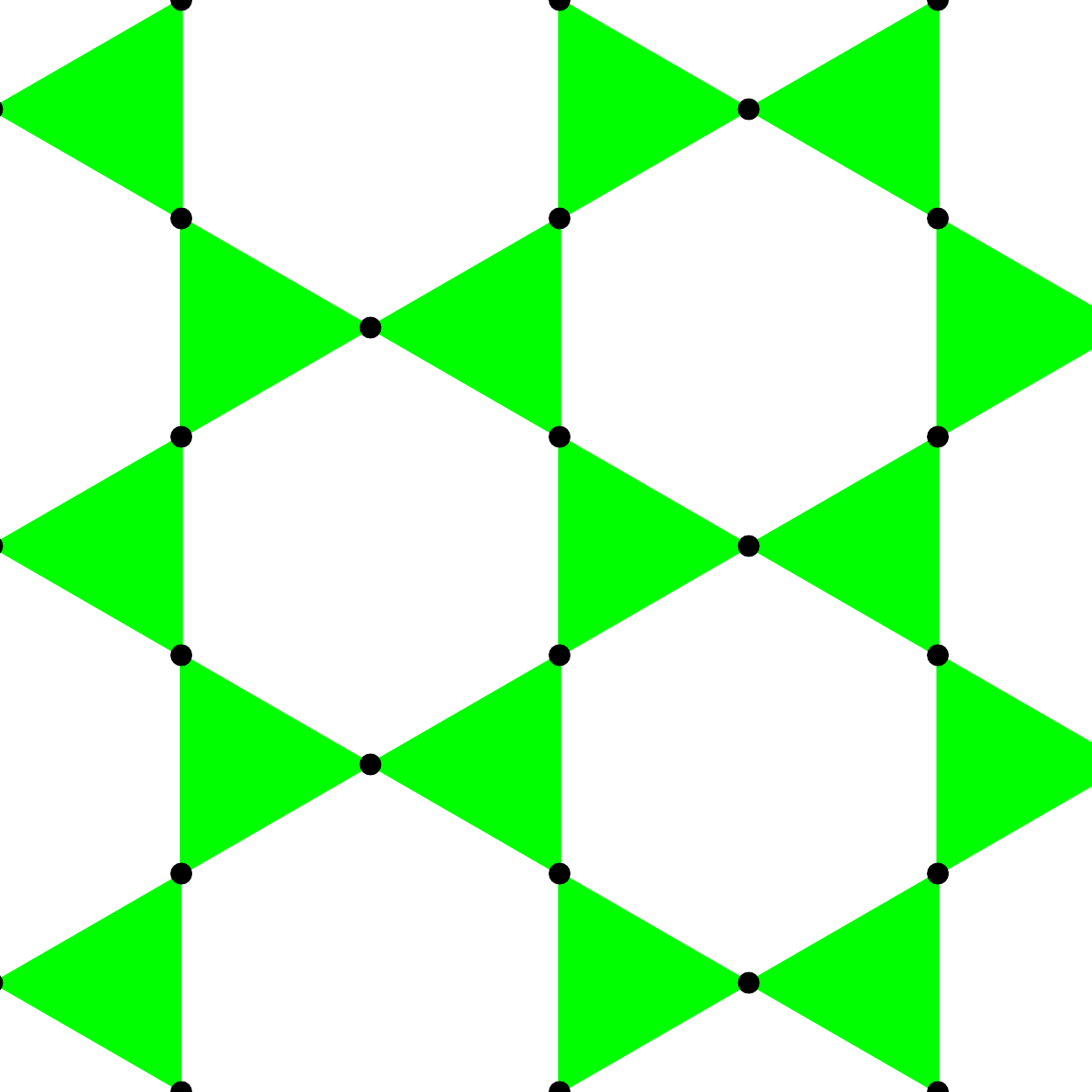}\hskip0.5cm\includegraphics[scale=0.3]{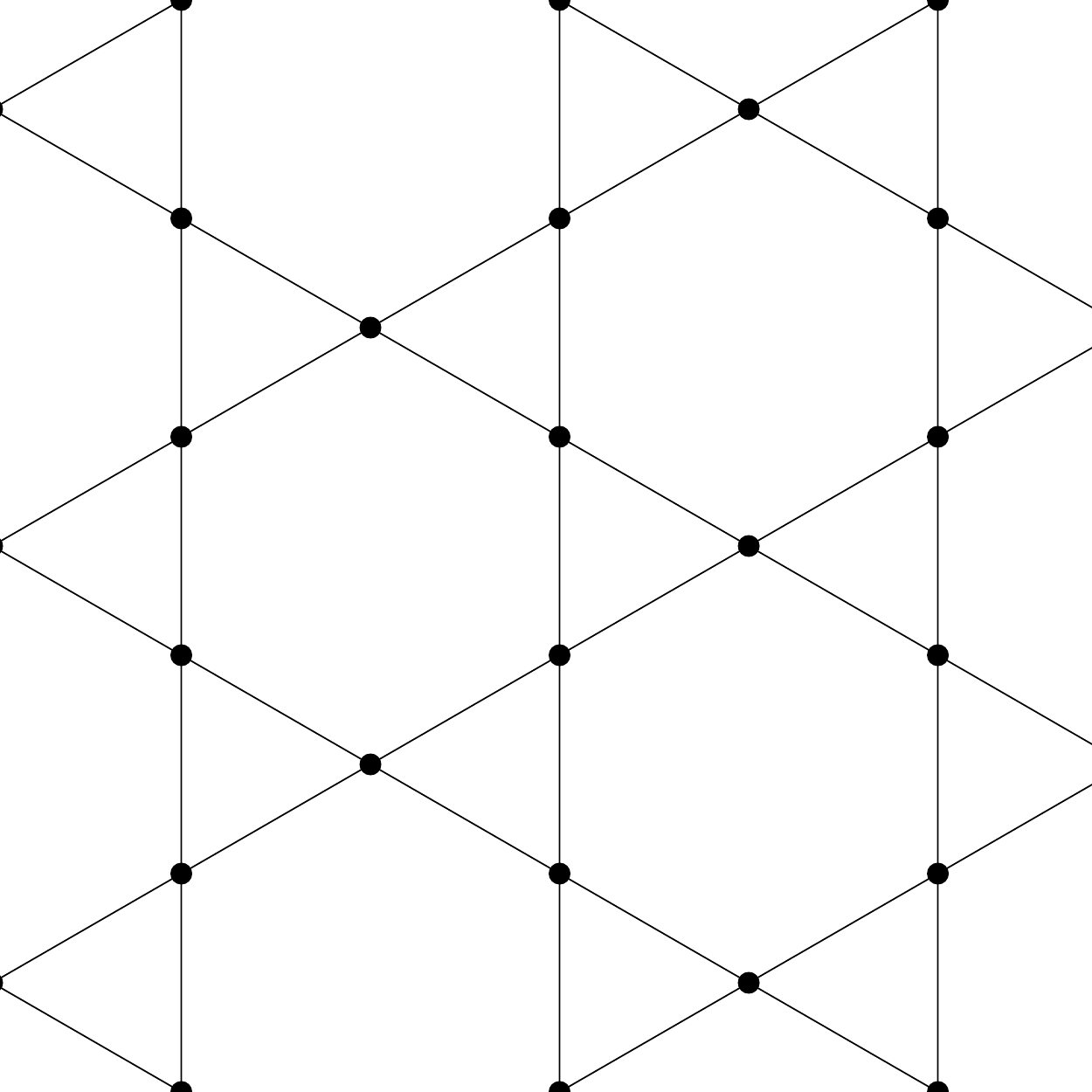}\\\vskip0.5cm
 \includegraphics[scale=0.3]{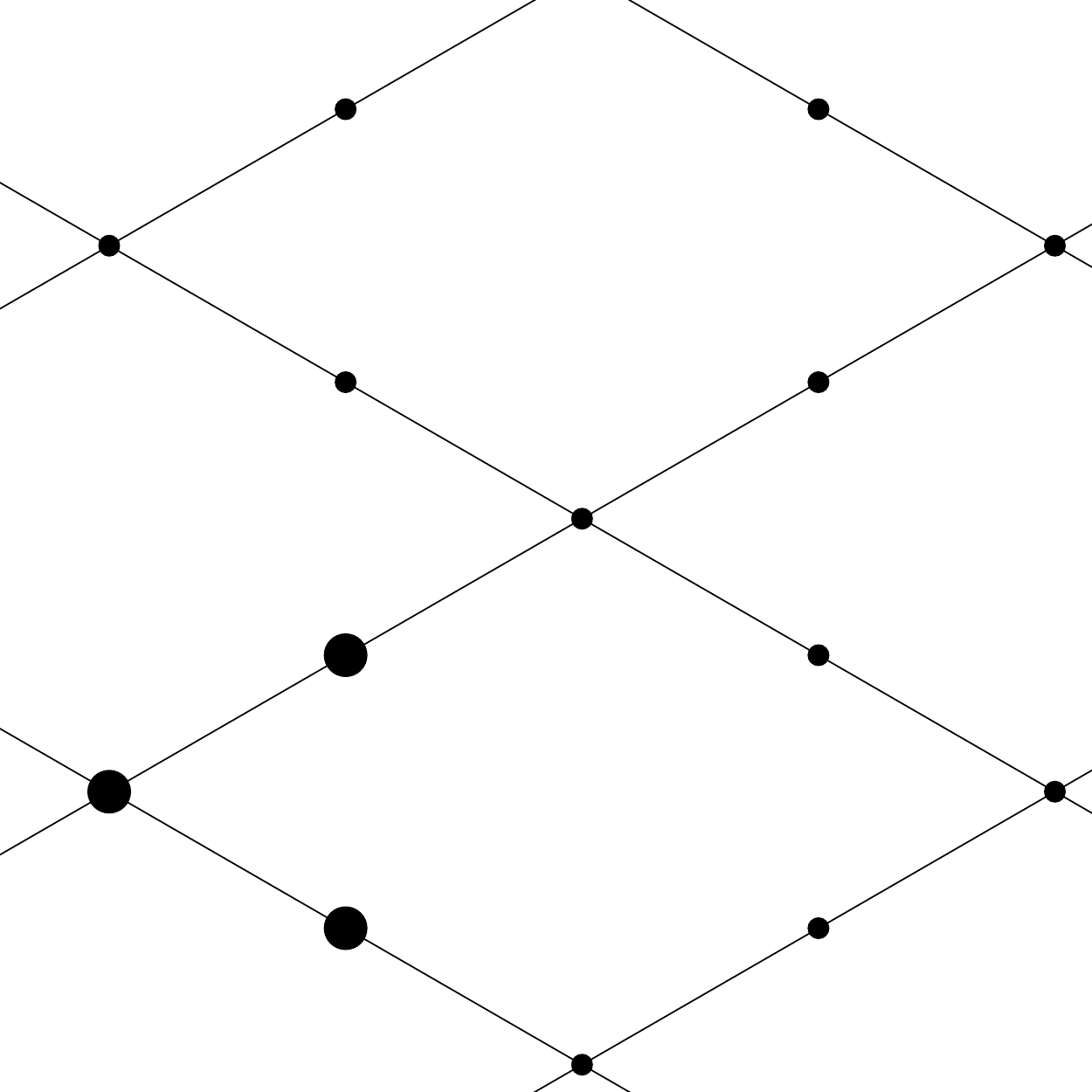}\hskip0.5cm\includegraphics[scale=0.3]{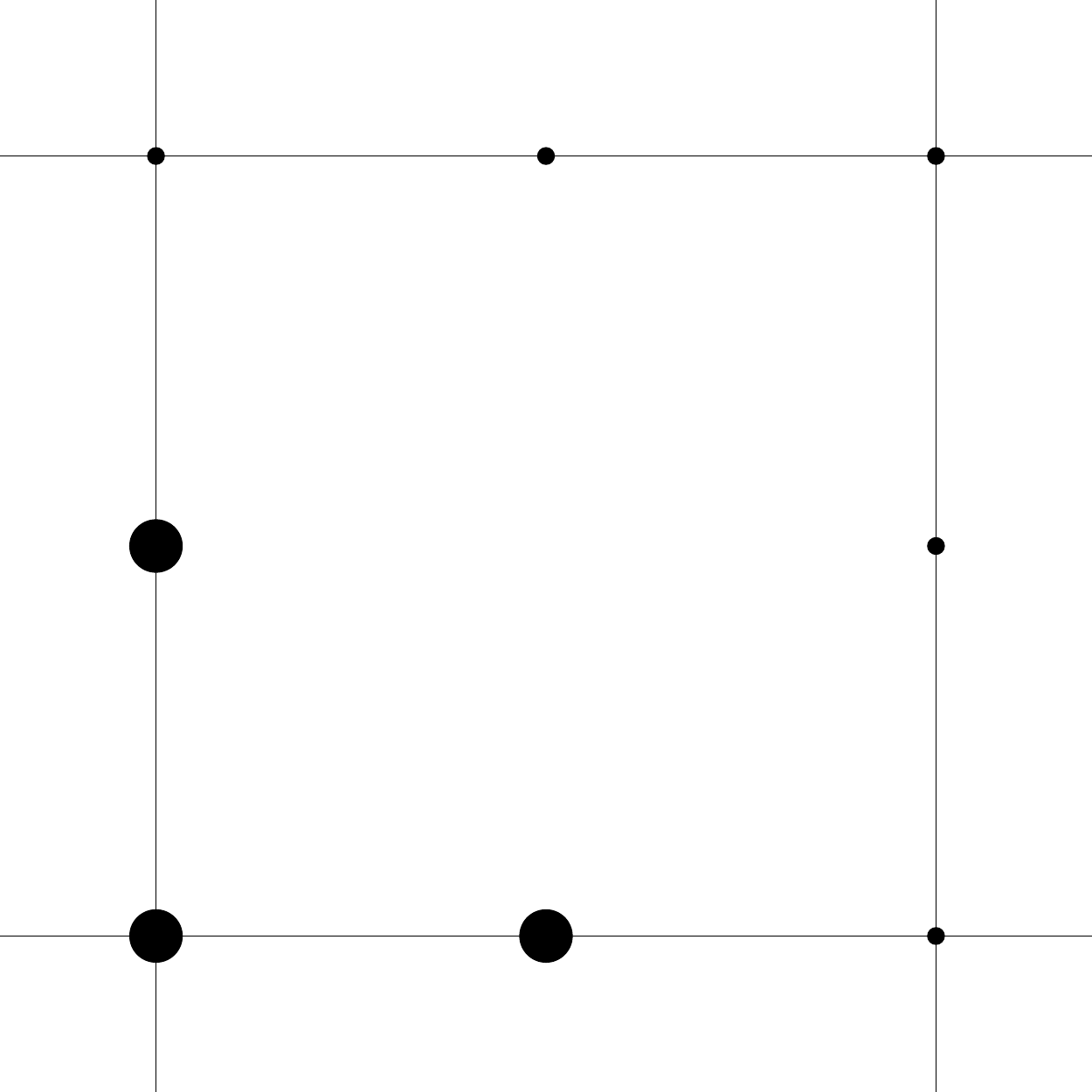}\\
 \caption{Trihexagonal tiling}\end{figure}

$\Lambda=\cup_{j=1}^3 L^*(u_j+  \ZZ^2)$
 where
 $$
\begin{array}{l}
u_1=( 0,  0 )\\
 u_2=( 0, \frac{1}{2}) \\
 u_3=( \frac{1}{2},0) \\
\end{array}
\qquad\text{and}\qquad L^*=\left(
\begin{array}{cc}
 \sqrt{3} & \sqrt{3} \\
 1 & -1 \\
\end{array}
\right).$$
See Figure \ref{good_domains_4} for the list of connected domains of the form $\cup_{k=1}^4 \Omega_0+v_k$ with $(v_k)$ satisfying condition (A2). We extended our investigation of condition (A2) to the set of domains
$\cup_{k=1}^4 \Omega_0+v_k$ with $(v_k)\in \ZZ^2\cap[0,2]^2$, see Figure \ref{general_good_domains_4}. 
By a direct computation, 36 over the 84 domains of this form satisfy condition (A2) and the associated constants $c_1$ and $c_2$ are 
constantly equal to $1$ and $4$, respectively.

\begin{figure}[h]
 \includegraphics[scale=1]{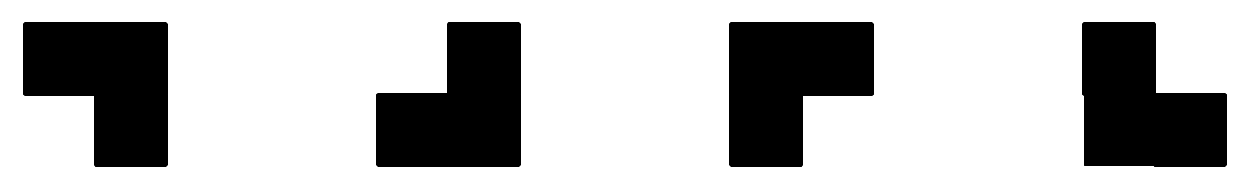}
 \caption{Trihexagonal tiling: the connected domains of the form $\cup_{k=1}^4 \Omega_0+v_k$ with $(v_k)$ satisfying condition (A2).\label{good_domains_4}}\newpage
\end{figure}

\begin{figure}[h]
 \includegraphics[scale=1]{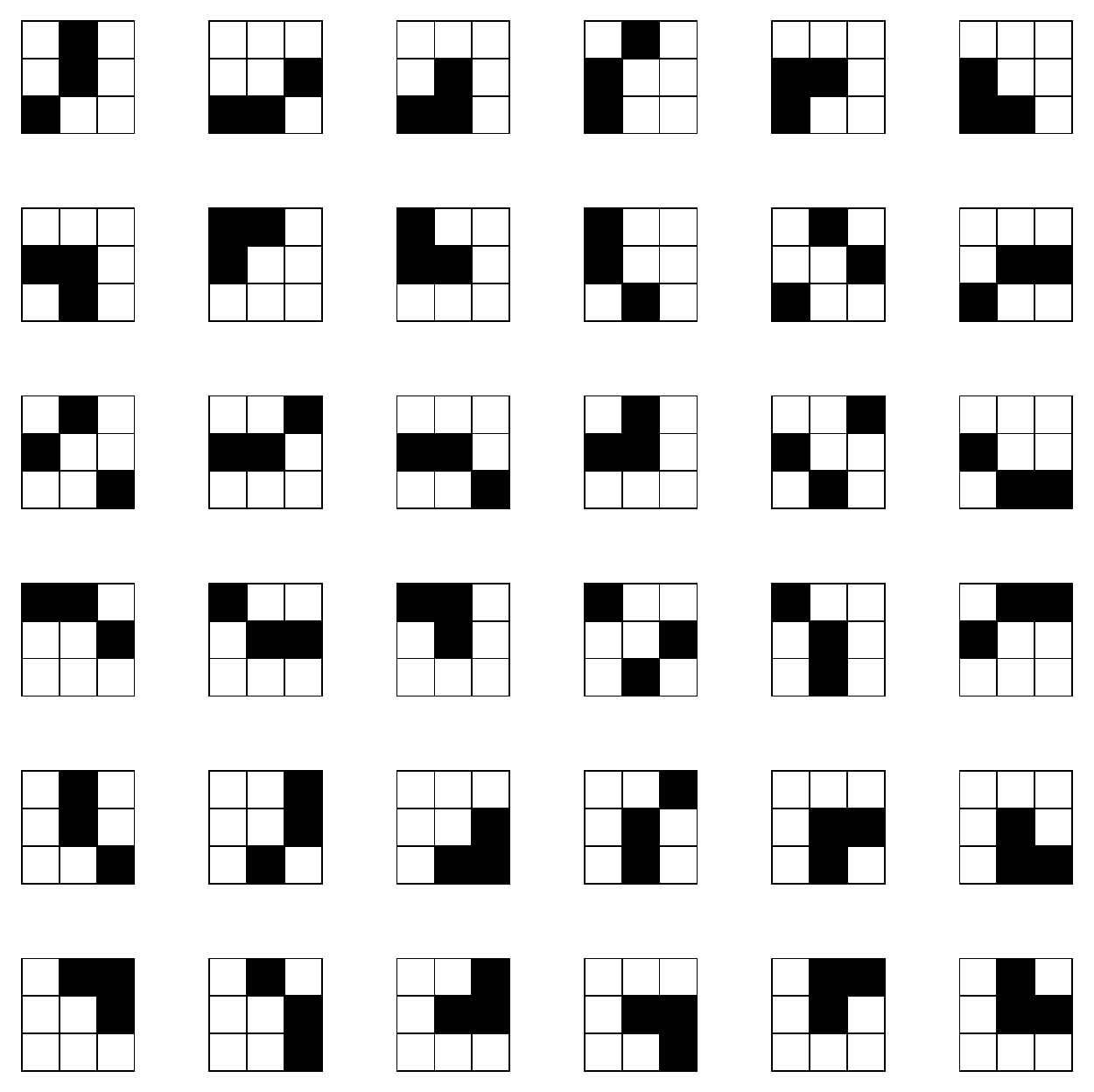}
 \caption{Trihexagonal tiling: the 36 (over 84) domains of the form $\cup_{k=1}^4 \Omega_0+v_k$ with $(v_k)\in \ZZ^2\cap[0,3]^2$ satisfying condition (A2).\label{general_good_domains_4}}\newpage
\end{figure}
\clearpage

\subsection{Snub square tiling}\.\\

\begin{figure}[h!]
 \includegraphics[scale=0.3]{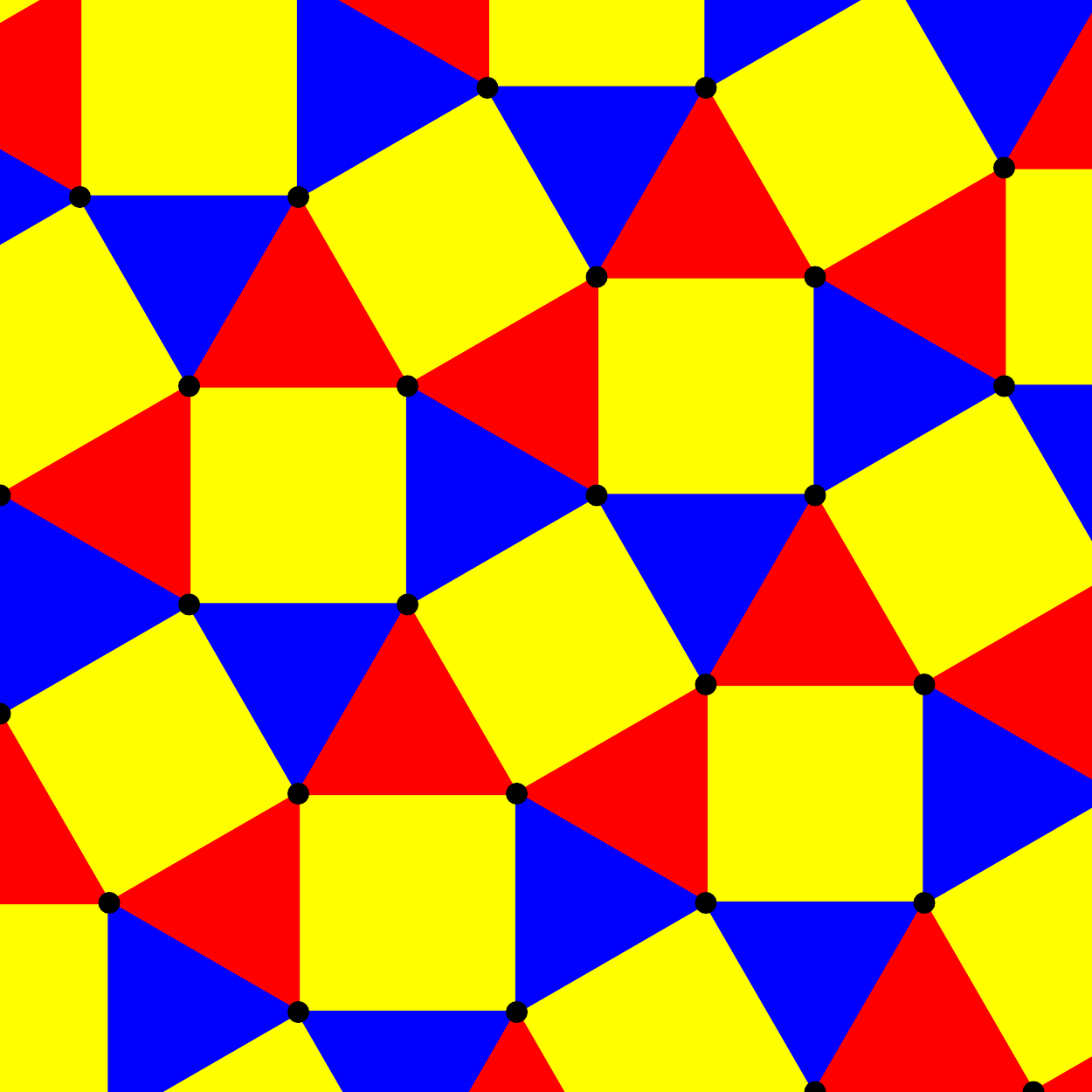}\hskip0.5cm\includegraphics[scale=0.3]{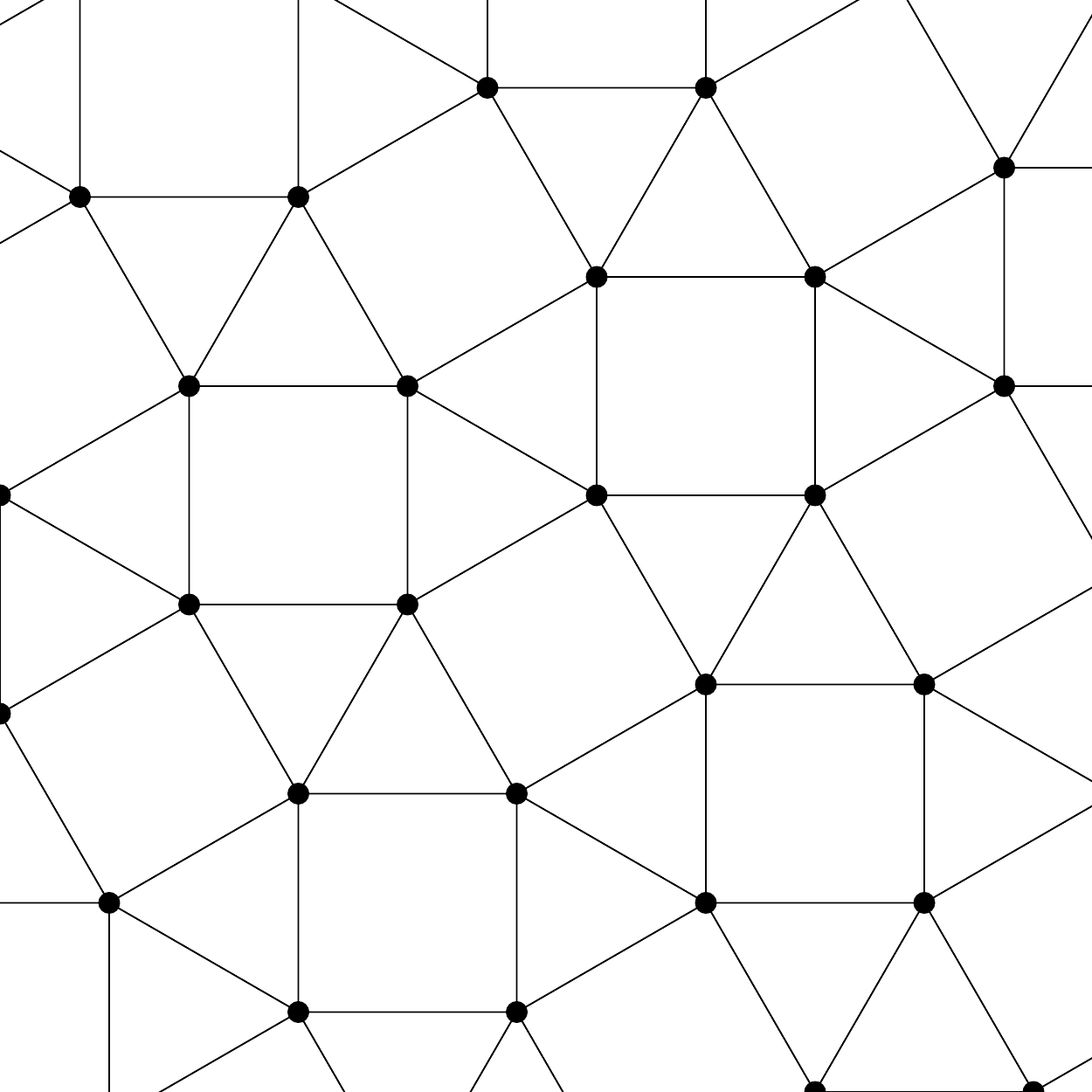}\\\vskip0.5cm
 \includegraphics[scale=0.3]{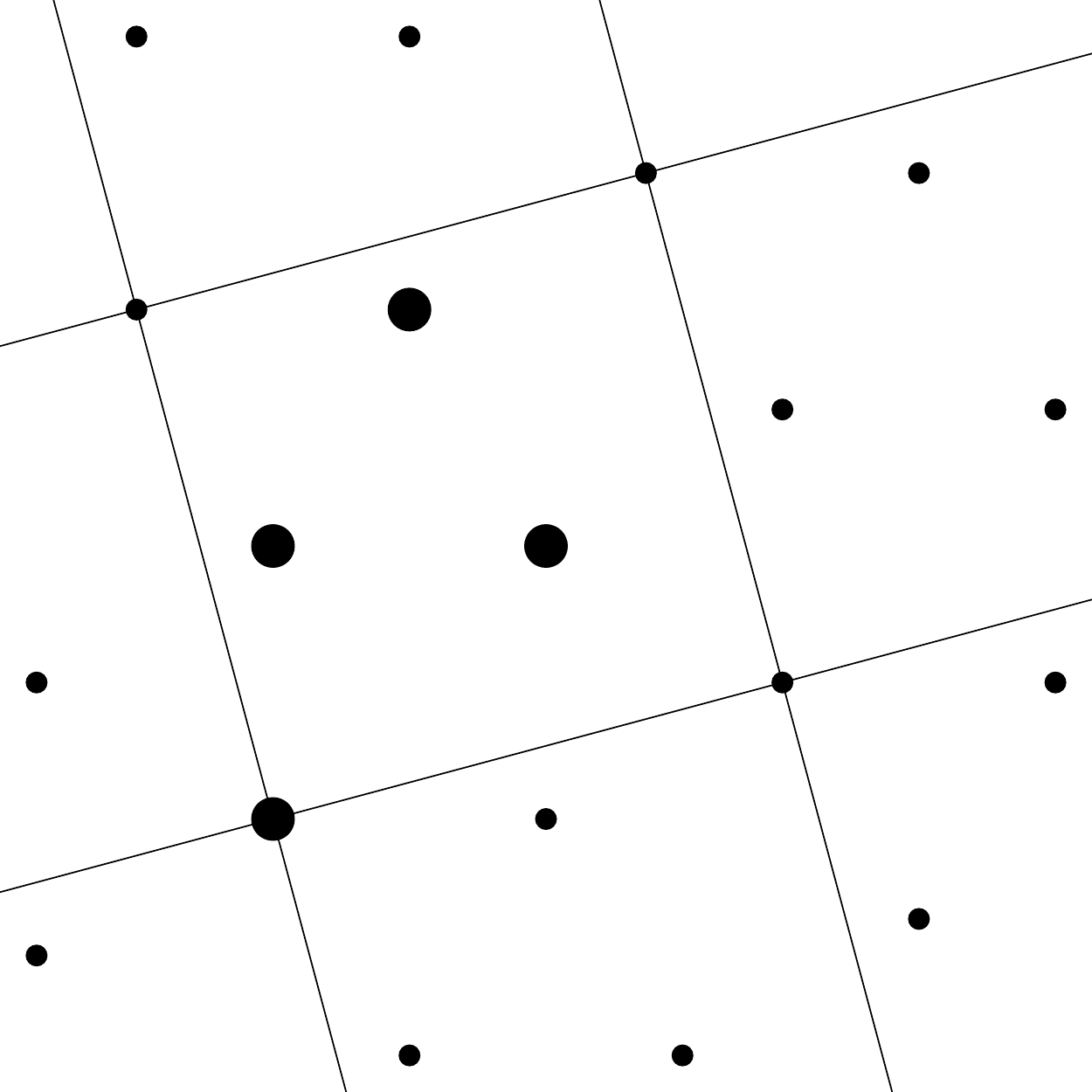}\hskip0.5cm\includegraphics[scale=0.3]{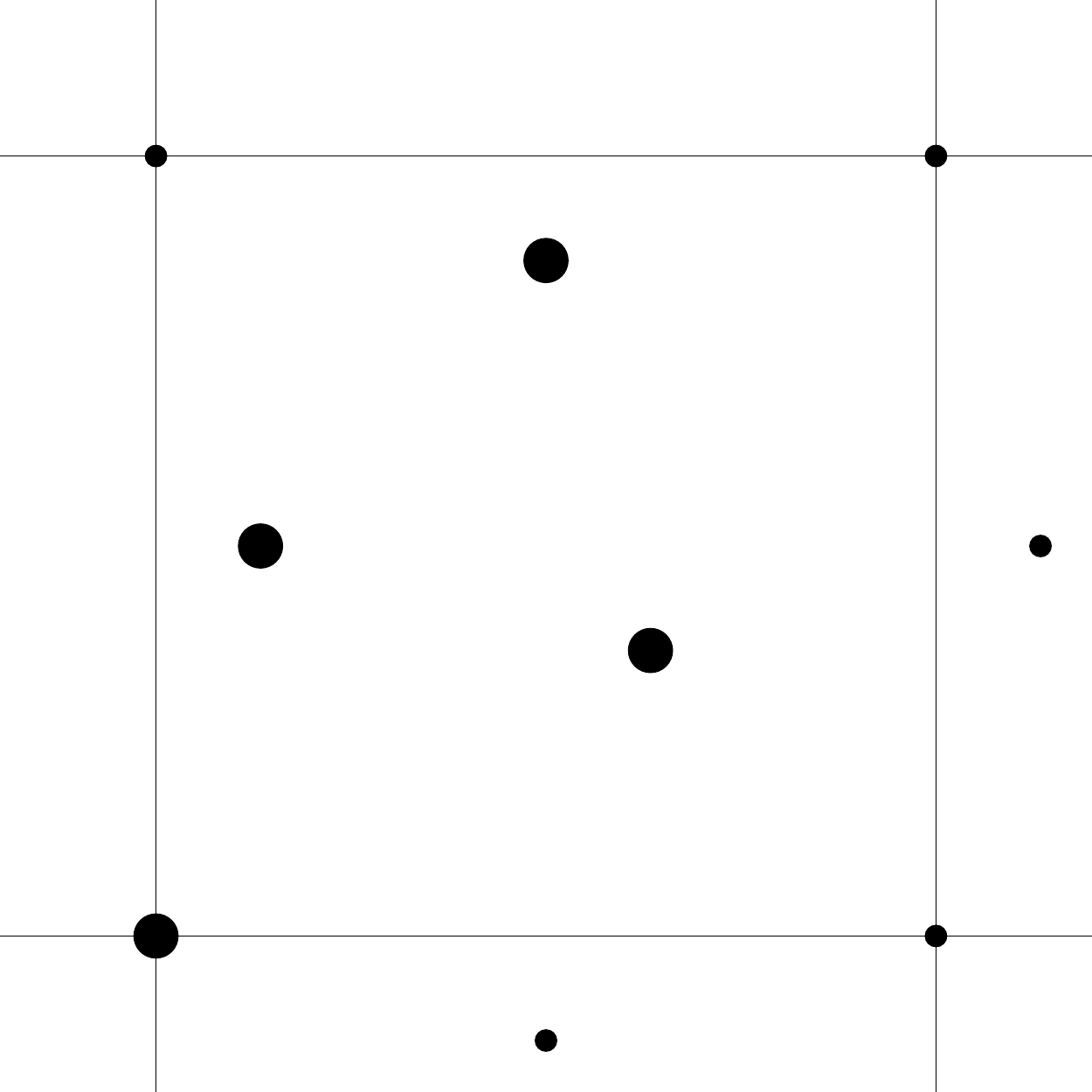}\\
\caption{Snub square tiling}\end{figure}

$\Lambda=\cup_{j=1}^4 L^*(u_j+  \ZZ^2)$
 where
  $$
\begin{array}{ll}
u_1=( 0, 0) &
 u_2=(1-\frac{\sqrt{3}}{2},\frac 1 2) \\
 u_3=(\frac{1}{2}(3-\sqrt{3}),\frac{1}{2}(-1+\sqrt{3})) &
 u_4=(\frac{1}{2}, \frac{\sqrt{3}}{2}) \\
\end{array}$$
and$$
L^*=\left(
\begin{array}{cc}
 1+\frac{\sqrt{3}}{2} & -\frac{1}{2} \\
 \frac{1}{2} & 1+\frac{\sqrt{3}}{2} \\
\end{array}
\right).$$

By a direct computation, for every $\{v_1,\dots,v_4\}$ such that $\cup_{k=1}^4 v_k+\Omega_0$ is a connected set, the condition (A2) is satisfied -- see Figure \ref{gooddomains_snub}.

\begin{figure}[h!]\captionsetup[subfig]{\hmargin=0,\hpar=0,\checksingleline=true,width=\linewidth} \vskip0.7cm
\begin{tabular}{lllll}
\begin{subfigure}{0.1\textwidth}
\begin{center}\includegraphics[width=1.2\linewidth]{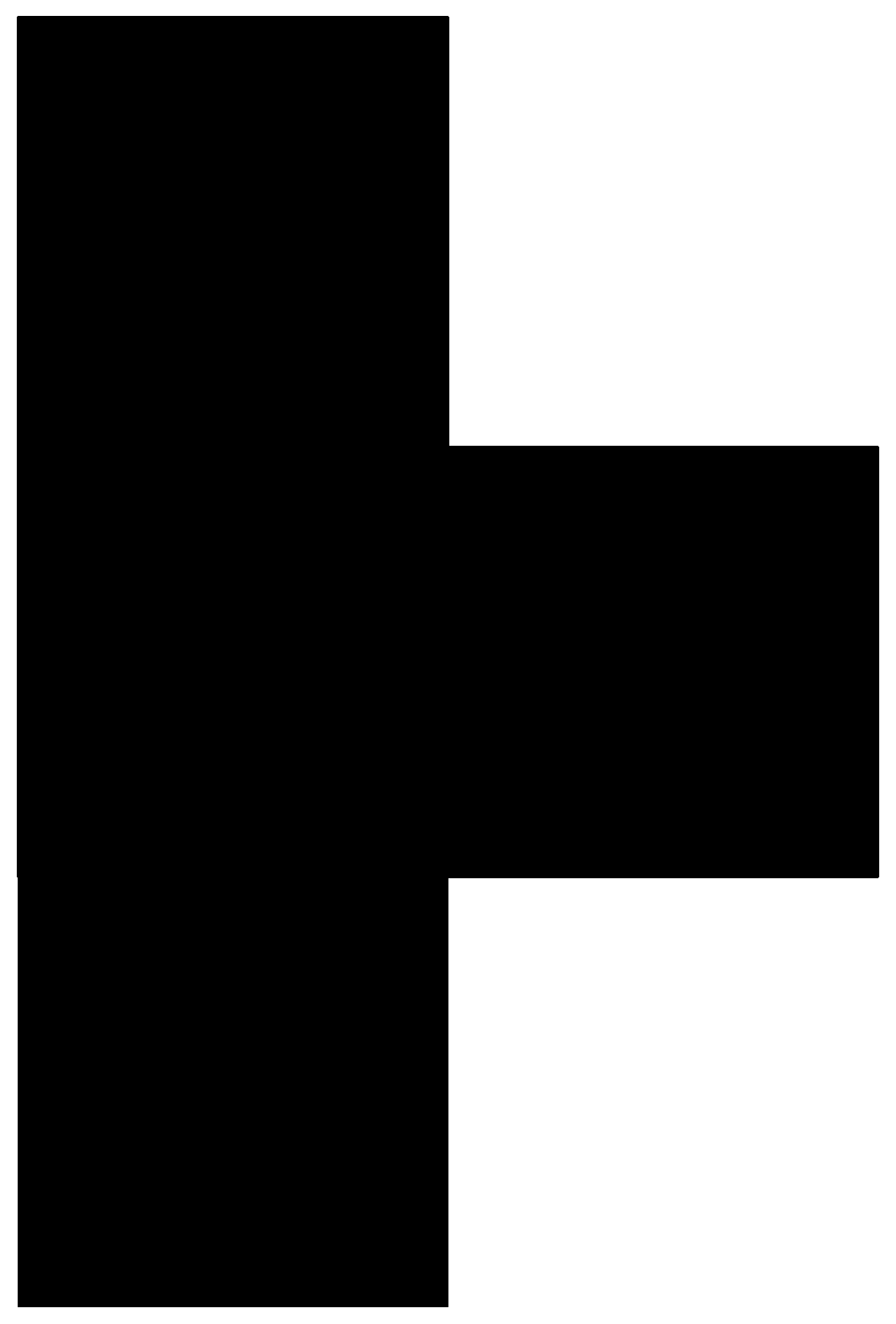}
  \captionsetup{width=1.2\linewidth,labelsep=newline,parindent=0cm,justification=centering}
    \caption{\small$c_1=1.03$ $c_2=6.66$}\end{center}
   \end{subfigure}\hskip1cm&
\begin{subfigure}{0.1\textwidth}
\begin{center}\includegraphics[width=1.2\linewidth]{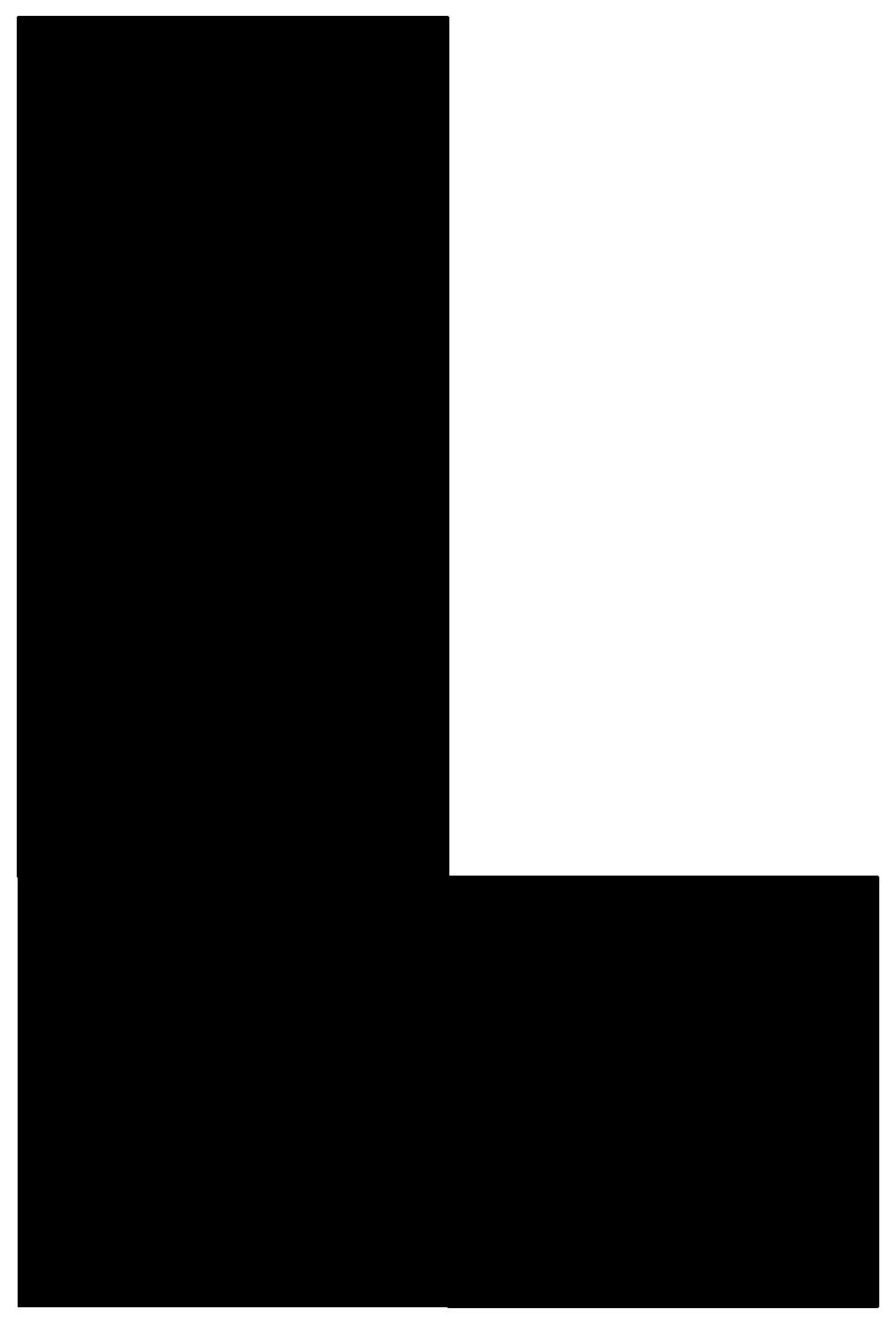}
  \captionsetup{width=1.2\linewidth,labelsep=newline,parindent=0cm,justification=centering}
    \caption{\small$c_1=0.16$ $c_2=7.83$}\end{center}
  \end{subfigure}\hskip1cm&
\begin{subfigure}{0.1\textwidth}
\vskip-0.7cm\begin{center}\hskip0.1cm\includegraphics[width=0.6\linewidth]{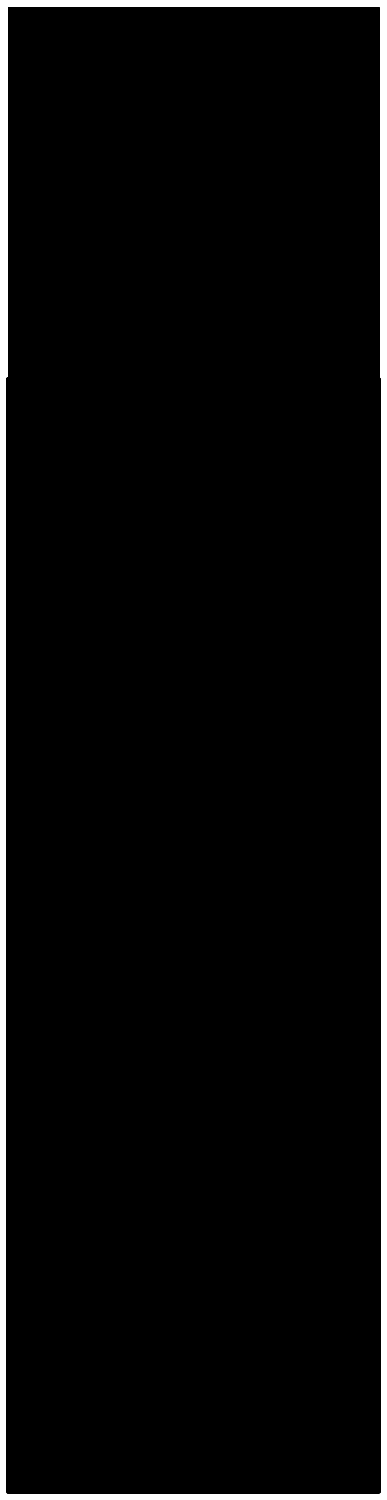}
  \captionsetup{width=1.2\linewidth,labelsep=newline,parindent=0cm,justification=centering}
   \caption{\small$c_1=1.33$ $c_2=6.66$}\end{center}
  \end{subfigure}
  \hskip1cm&
\begin{subfigure}{0.1\textwidth}
\begin{center}\includegraphics[width=1.2\linewidth]{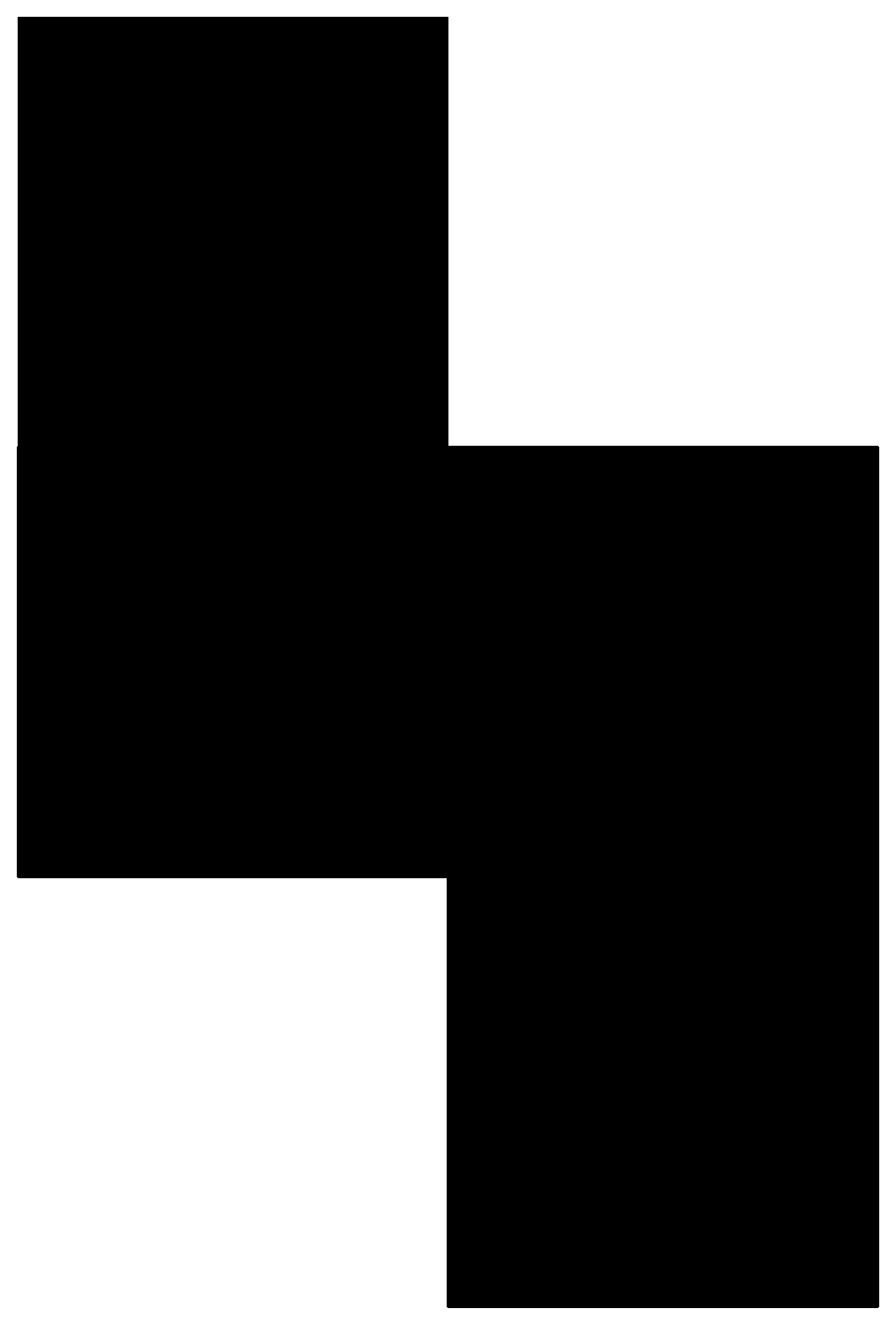}
  \captionsetup{width=1.2\linewidth,labelsep=newline,parindent=0cm,justification=centering}
  \caption{\small$c_1=1.12$ $c_2=6.87$}\end{center}
  \end{subfigure}
  \hskip1cm& \begin{subfigure}{0.1\textwidth}
  \vskip0.7cm
\begin{center}\hskip0.1cm\includegraphics[width=1.2\linewidth]{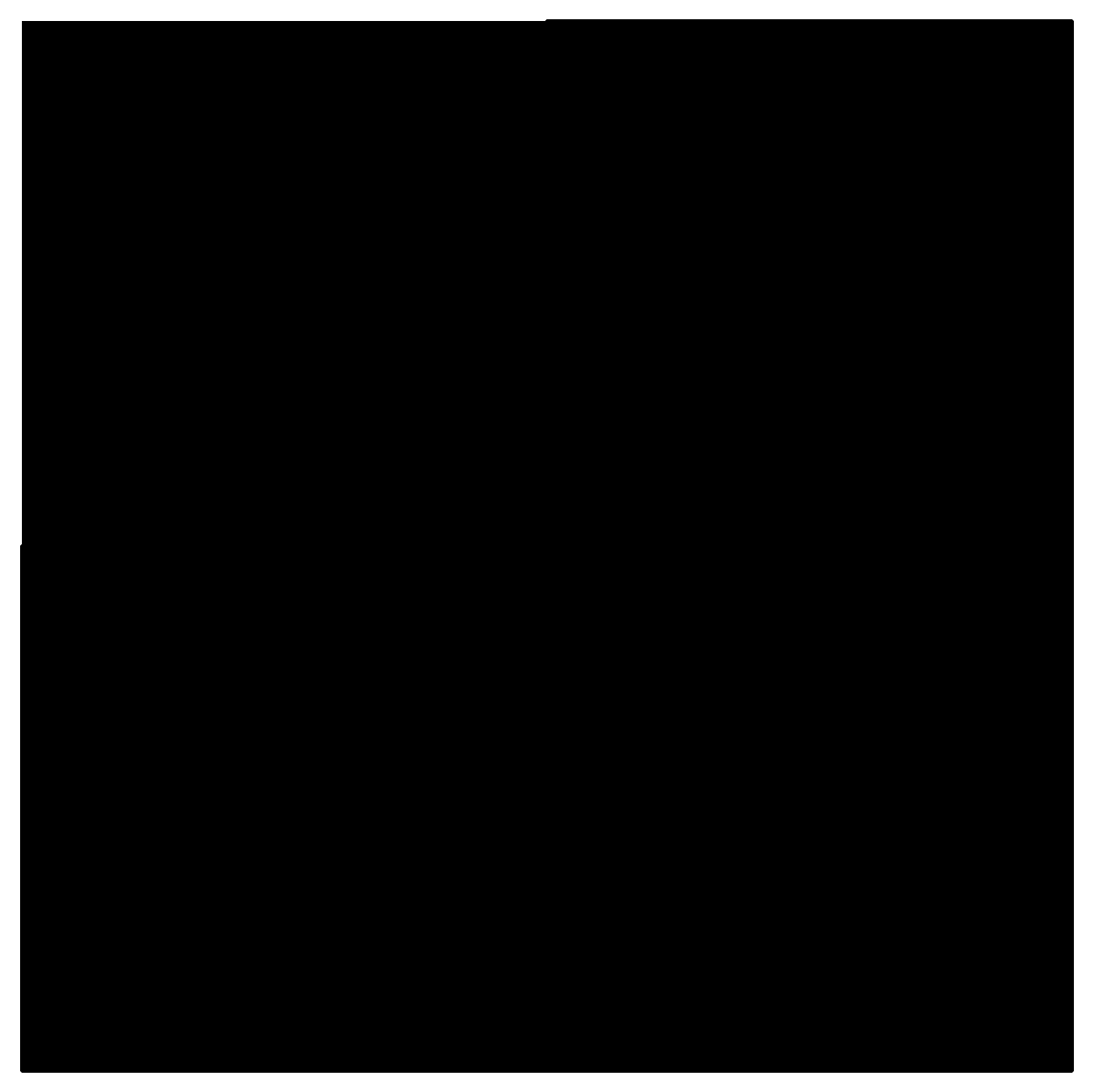}
  \captionsetup{width=1.2\linewidth,labelsep=newline,parindent=0cm,justification=centering}
    \caption{\small$c_1=0.54$ $c_2=2.16$}\end{center}
  \end{subfigure}  
  \end{tabular}
\caption{\label{gooddomains_snub}Snub square tiling: the five representatives of the 19 connected domains satisfying (A2), and related constants. 
In this case, constants $c_1$ and $c_2$ are invariant with respect to reflections and rotations. The best ratio $c_2/c_1=6.12$ is achieved by the configuration (D).}
    \end{figure}

 \begin{figure}[h!]
 \includegraphics[scale=0.5]{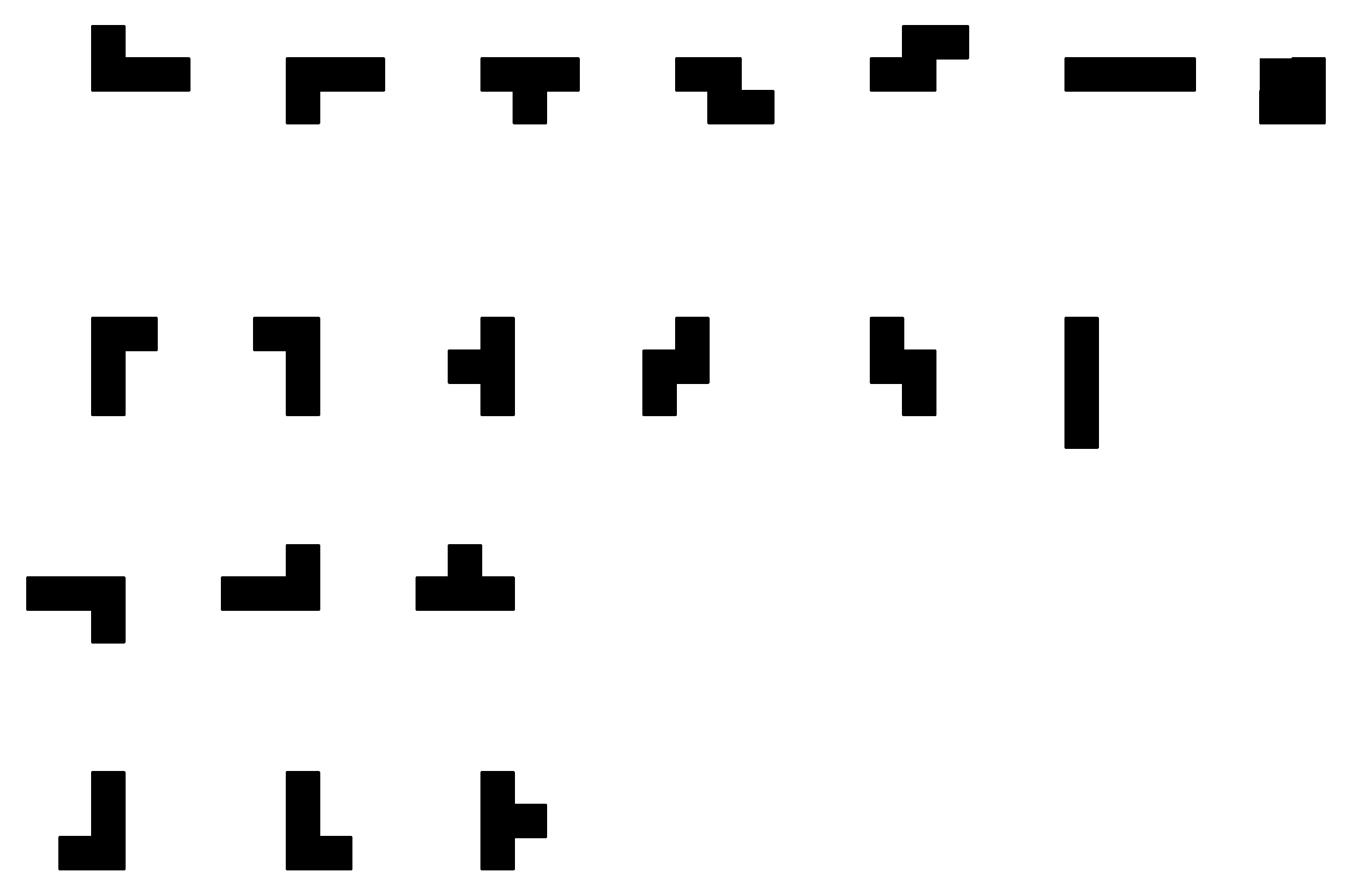}
 \caption{Fixed tetrominoes, i.e., the possible connected domains when $M=4$ \label{tetrominoes}.}
\end{figure}
\vskip5cm
\begin{figure}[h!]
 \includegraphics[scale=1]{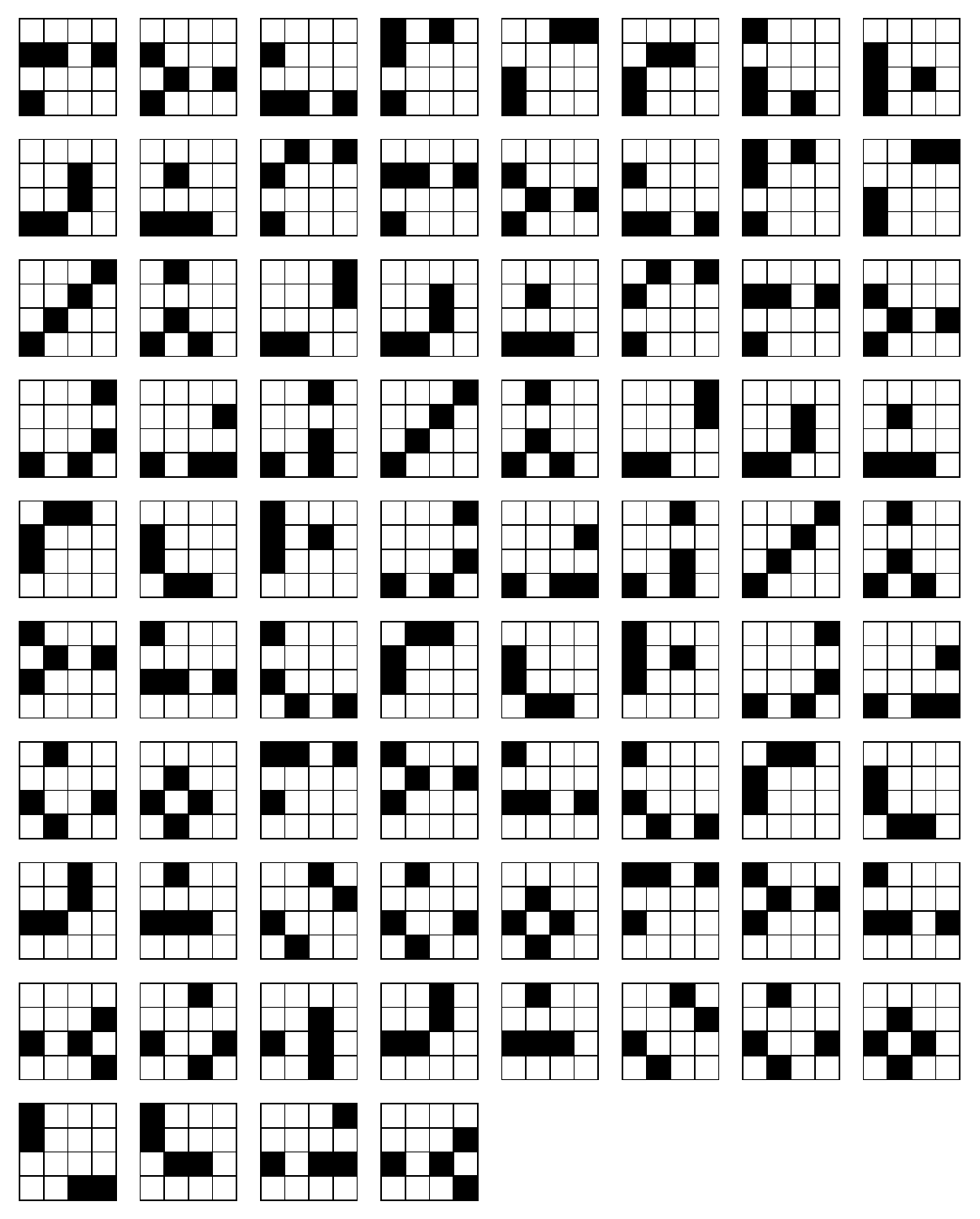}
 \caption{Snub square tiling: the 76 (over 1820) domains of the form $\cup_{k=1}^4 \Omega_0+v_k$ not satisfying condition (A2).\label{baddomains_3}}
\end{figure}
\clearpage
\subsection{Truncated square tiling}\.\\
\begin{figure}[h!]
 \includegraphics[scale=0.3]{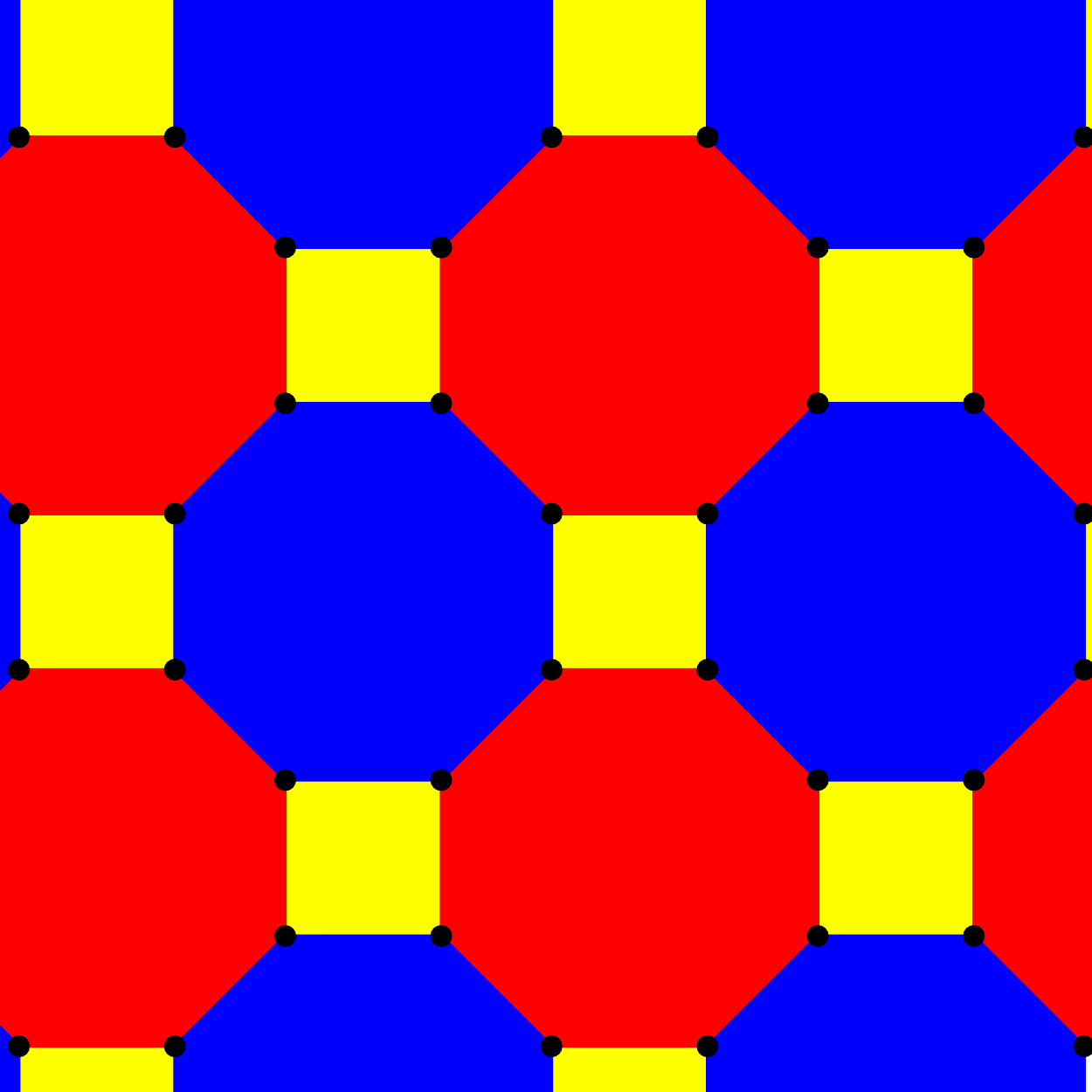}\hskip0.5cm\includegraphics[scale=0.3]{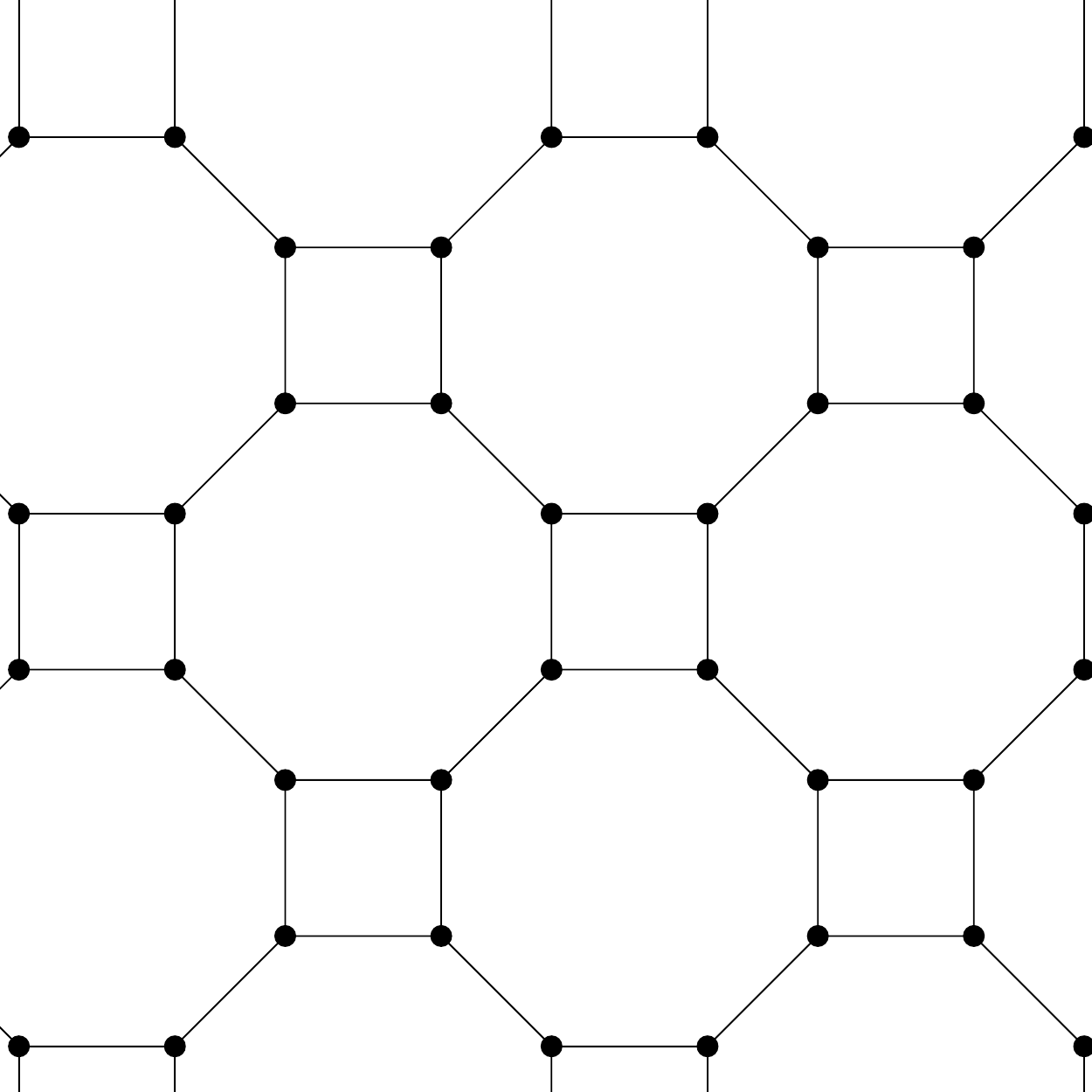}\\\vskip0.5cm
 \includegraphics[scale=0.3]{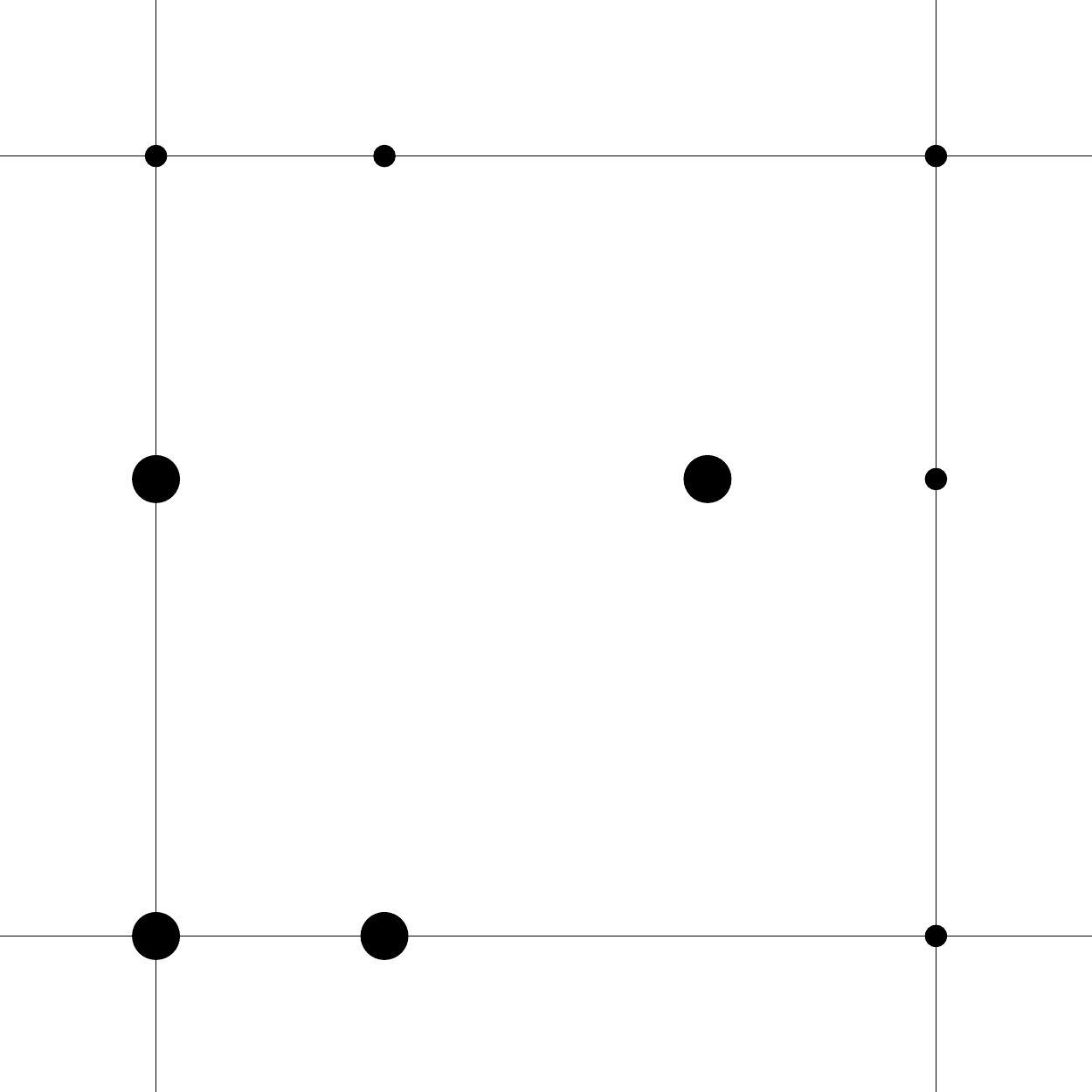}\hskip0.5cm\includegraphics[scale=0.3]{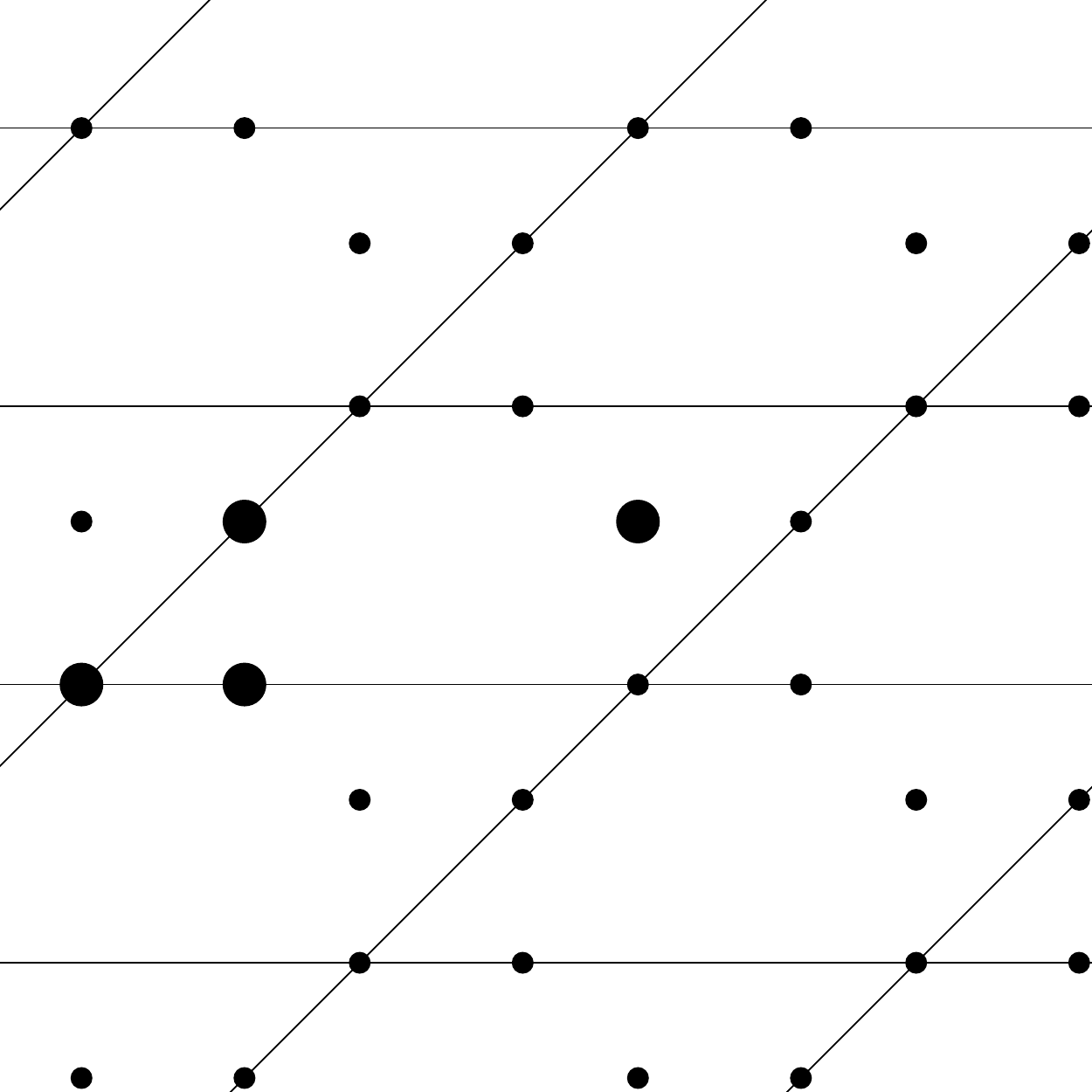}\\
 \caption{Truncated square tiling}\end{figure}
 \\
$\Lambda=\cup_{j=1}^4  L^*( u_j+ \ZZ^2)$
 where
  $$
\begin{array}{l}
u_1=( 0 , 0 )\\
 u_2=(1-\frac{\sqrt{2}}{2} , 0  )\\
 u_3=(0 , 2-\sqrt{2}  )\\
 u_4=(\frac{\sqrt{2}}{2} , 2-\sqrt{2}  )\\
\end{array}
$$
and 
$$L^*=\left(
\begin{array}{cc}
 2+\sqrt{2} & 1+\frac{\sqrt{2}}{2} \\
 0 & 1+\frac{\sqrt{2}}{2} \\
\end{array}
\right).$$
By a direct computation, the set of connected domains of the form $\cup_{k=1}^4 \Omega_0+v_k$ with $(v_k)$ satisfying condition (A2) contains 9 elements, depicted in Figure \ref{gooddomains_5}. See also 
Figure \ref{tetrominoes} for the complete list of connected domains (up to translations) and Figure \ref{baddomains_5} for some examples of (possibly disconnected) domains not satisfying condition (A2).
 \begin{figure}[h!]\captionsetup[subfig]{\hmargin=0,\hpar=0,\checksingleline=true,width=\linewidth} 
\begin{tabular}{lllllll}
\begin{subfigure}{0.1\textwidth}
\begin{center}\includegraphics[width=1.2\linewidth]{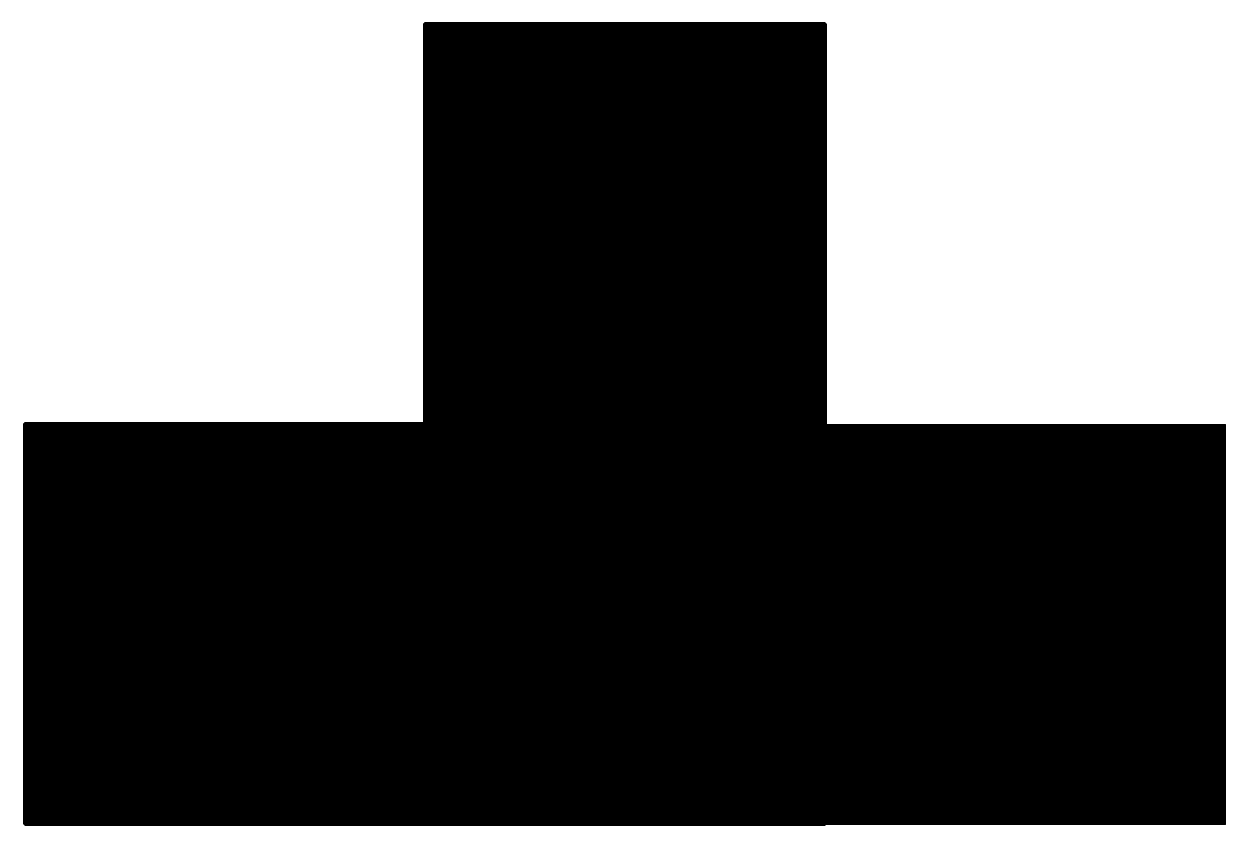}
  \captionsetup{width=1.2\linewidth,labelsep=newline,parindent=0cm,justification=centering}
     \end{center}
  \end{subfigure}\quad&
\begin{subfigure}{0.1\textwidth}
\begin{center}\includegraphics[width=1.2\linewidth]{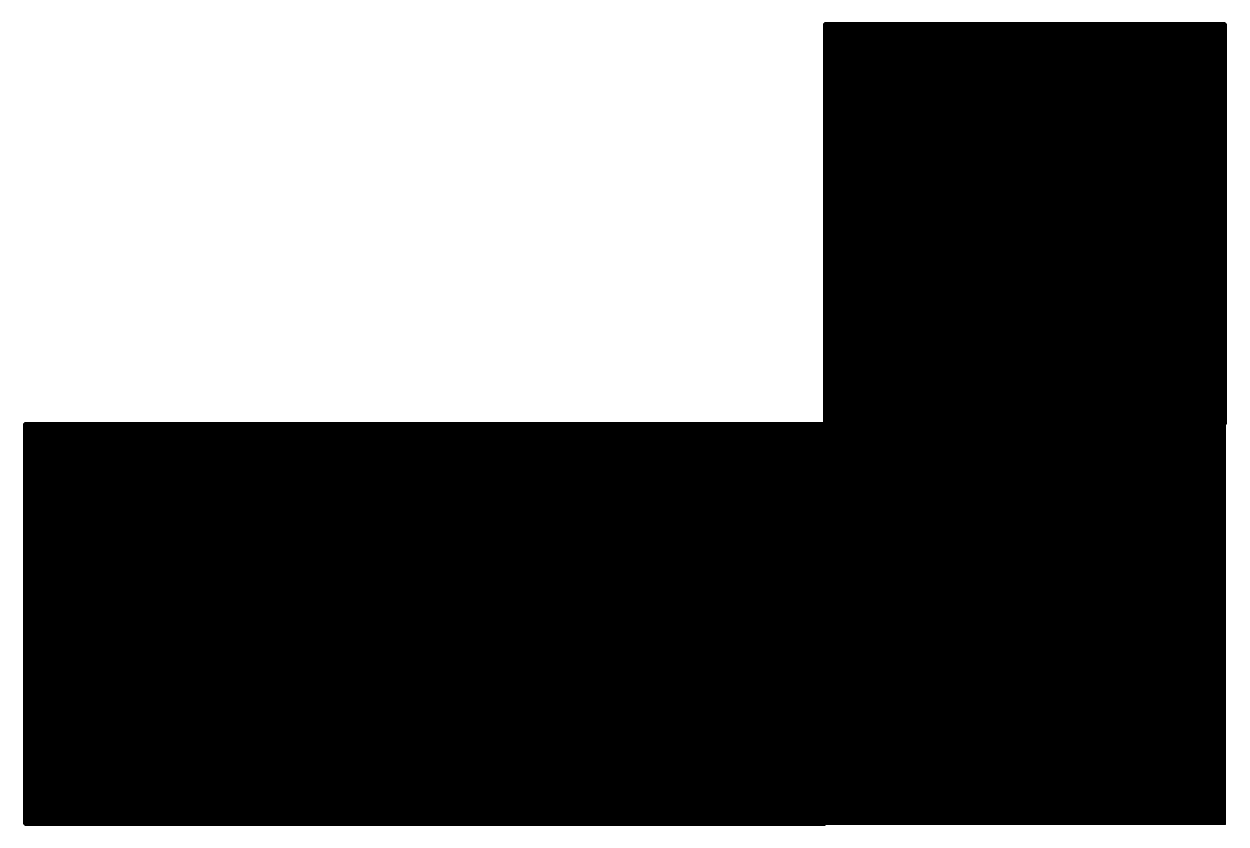}
  \captionsetup{width=1.2\linewidth,labelsep=newline,parindent=0cm,justification=centering}
\end{center}
  \end{subfigure}\quad&
  \begin{subfigure}{0.1\textwidth}
\begin{center}\includegraphics[width=1.2\linewidth]{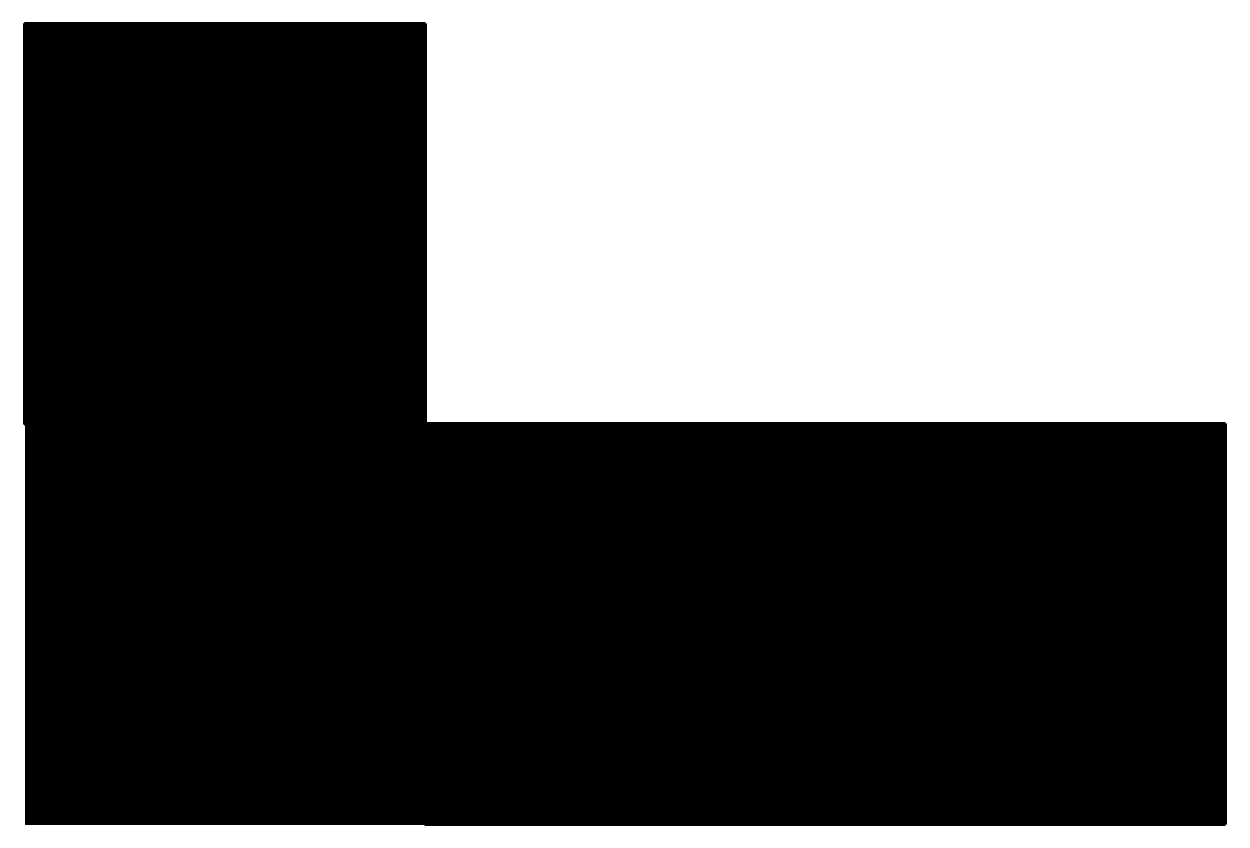}
  \captionsetup{width=1.2\linewidth,labelsep=newline,parindent=0cm,justification=centering}
\end{center}
 \end{subfigure}\quad&
\begin{subfigure}{0.1\textwidth}
\begin{center}\hskip0.1cm\includegraphics[width=0.8\linewidth]{tetramin_13.pdf}
  \captionsetup{width=1.2\linewidth,labelsep=newline,parindent=0cm,justification=centering}
\end{center}
  \end{subfigure}
  \quad&
  
  &

& 
     \\
       \\
       \\
 \begin{subfigure}{0.1\textwidth}
\begin{center}\includegraphics[width=1.2\linewidth]{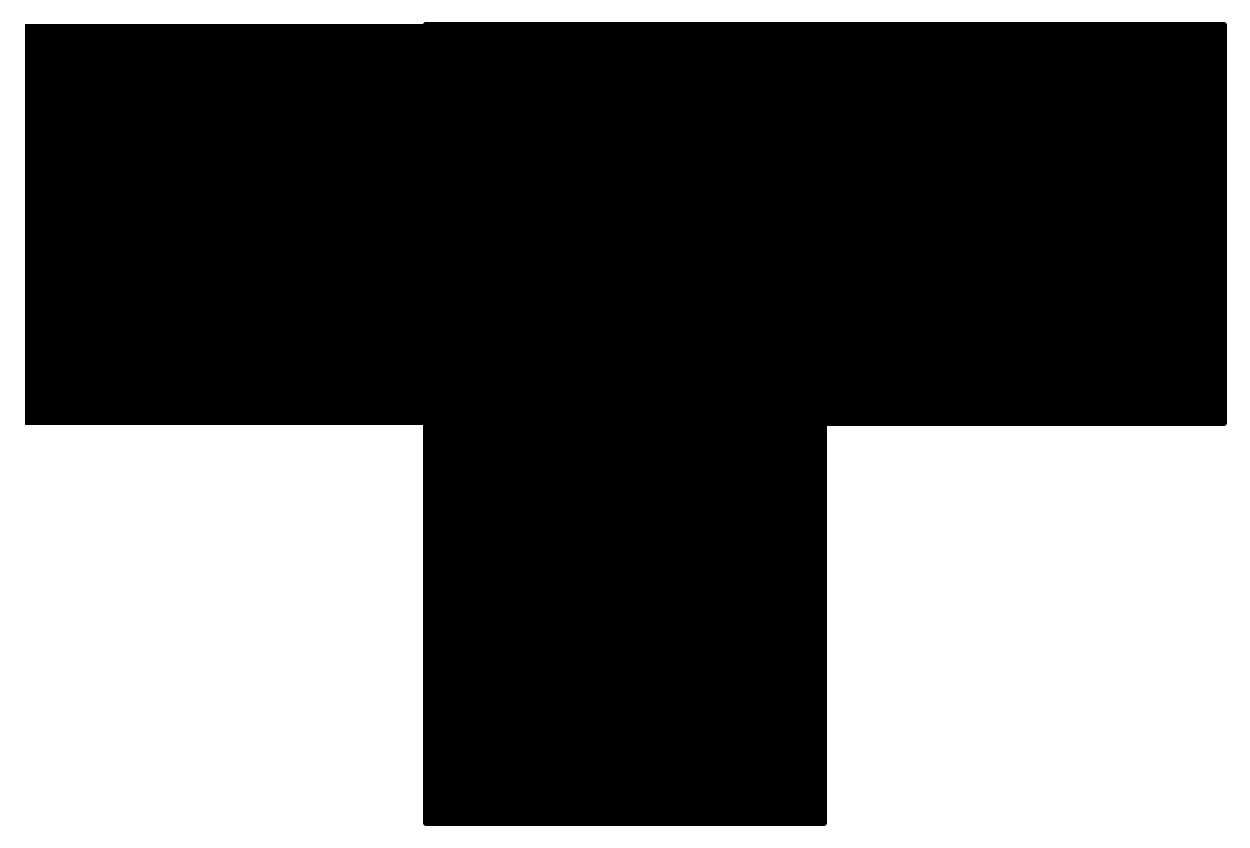}
  \captionsetup{width=1.2\linewidth,labelsep=newline,parindent=0cm,justification=centering}
\caption{\small$c_1=1.02$ $c_2=7.24$}\end{center}
  \end{subfigure}\qquad&
   \begin{subfigure}{0.1\textwidth}
\begin{center}\includegraphics[width=1.2\linewidth]{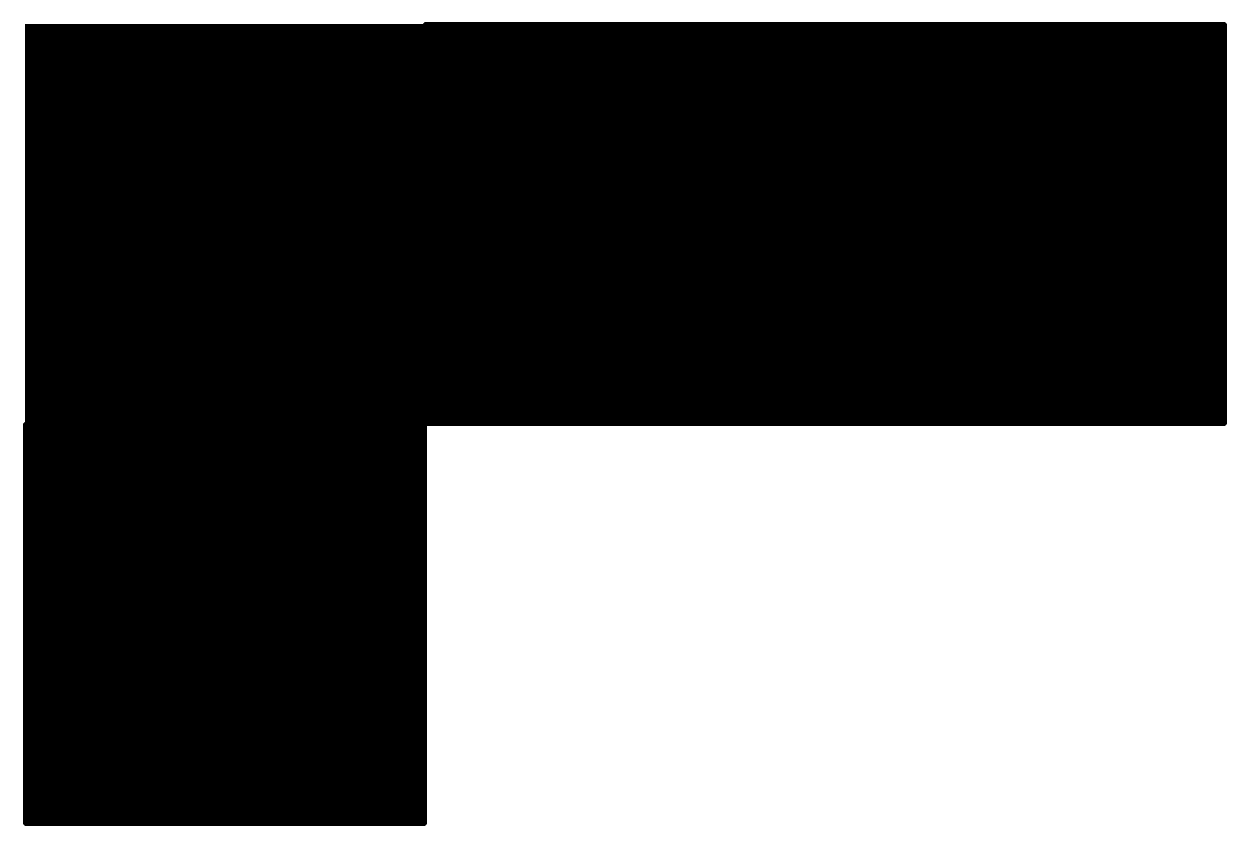}
  \captionsetup{width=1.2\linewidth,labelsep=newline,parindent=0cm,justification=centering}
  \caption{\small$c_1=0.71$ $c_2=6.23$}\end{center}
 \end{subfigure}\qquad&
  \begin{subfigure}{0.1\textwidth}
\begin{center}\includegraphics[width=1.2\linewidth]{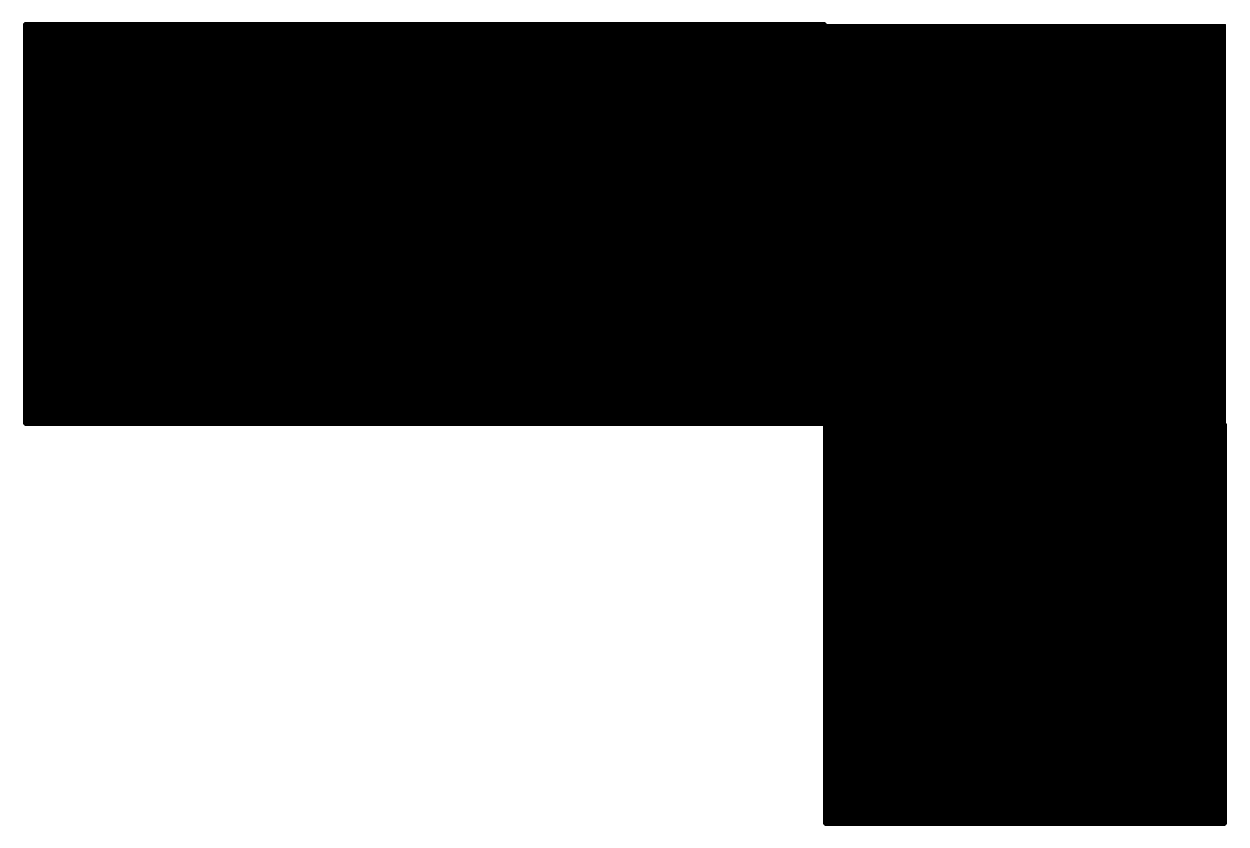}
  \captionsetup{width=1.2\linewidth,labelsep=newline,parindent=0cm,justification=centering}
\caption{\small$c_1=0.83$ $c_2=7.33$}\end{center}
 \end{subfigure}\qquad
 &
 \begin{subfigure}{0.1\textwidth}
\vskip-0.45cm\begin{center}\hskip0.1cm\includegraphics[width=0.8\linewidth]{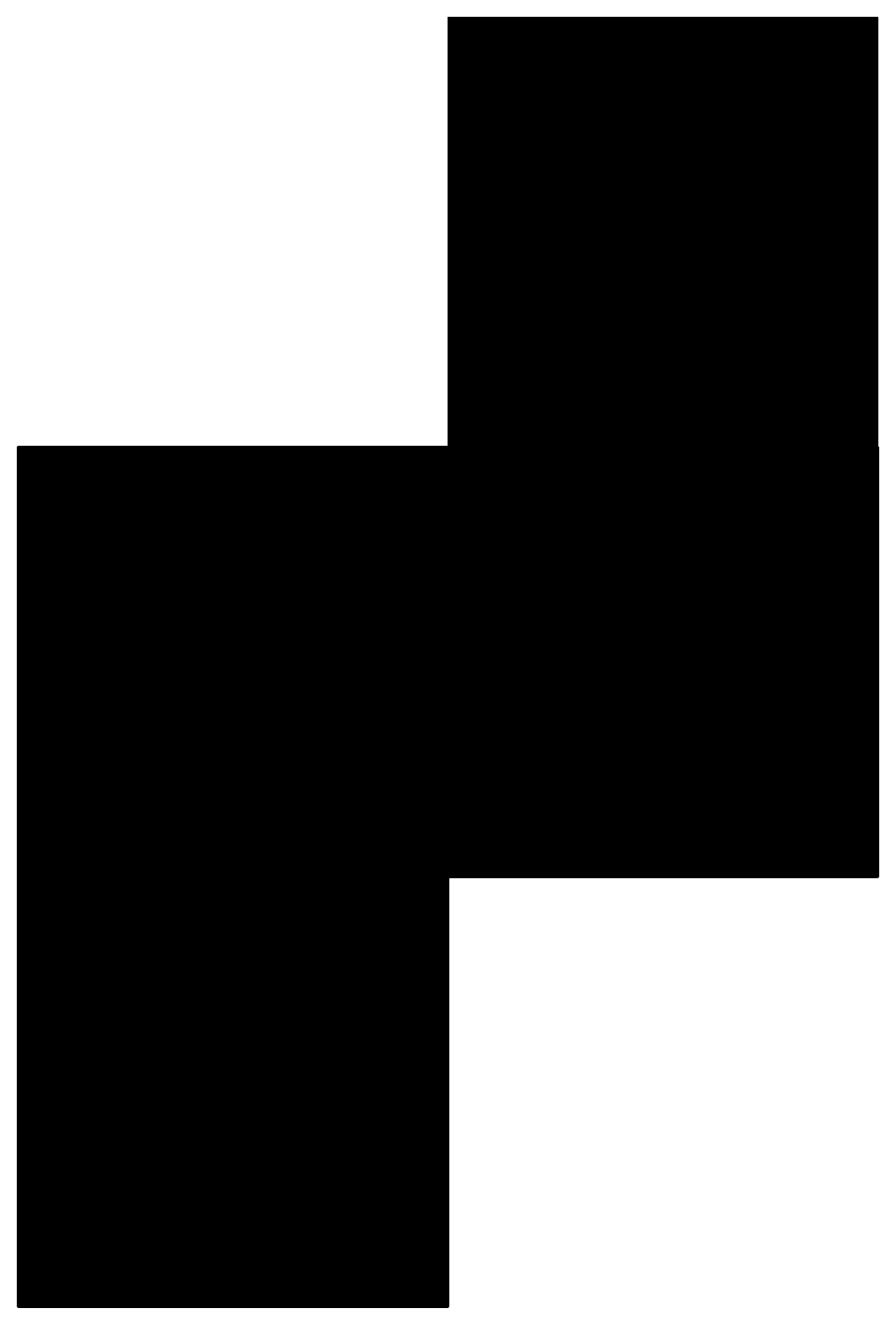}
  \captionsetup{width=1.2\linewidth,labelsep=newline,parindent=0cm,justification=centering}
  \caption{\small$c_1=1.17$ $c_2=8.02$}\end{center}
  \end{subfigure}
  \quad
 & \begin{subfigure}{0.1\textwidth}
\begin{center}\hskip0.1cm\includegraphics[width=1.2\linewidth]{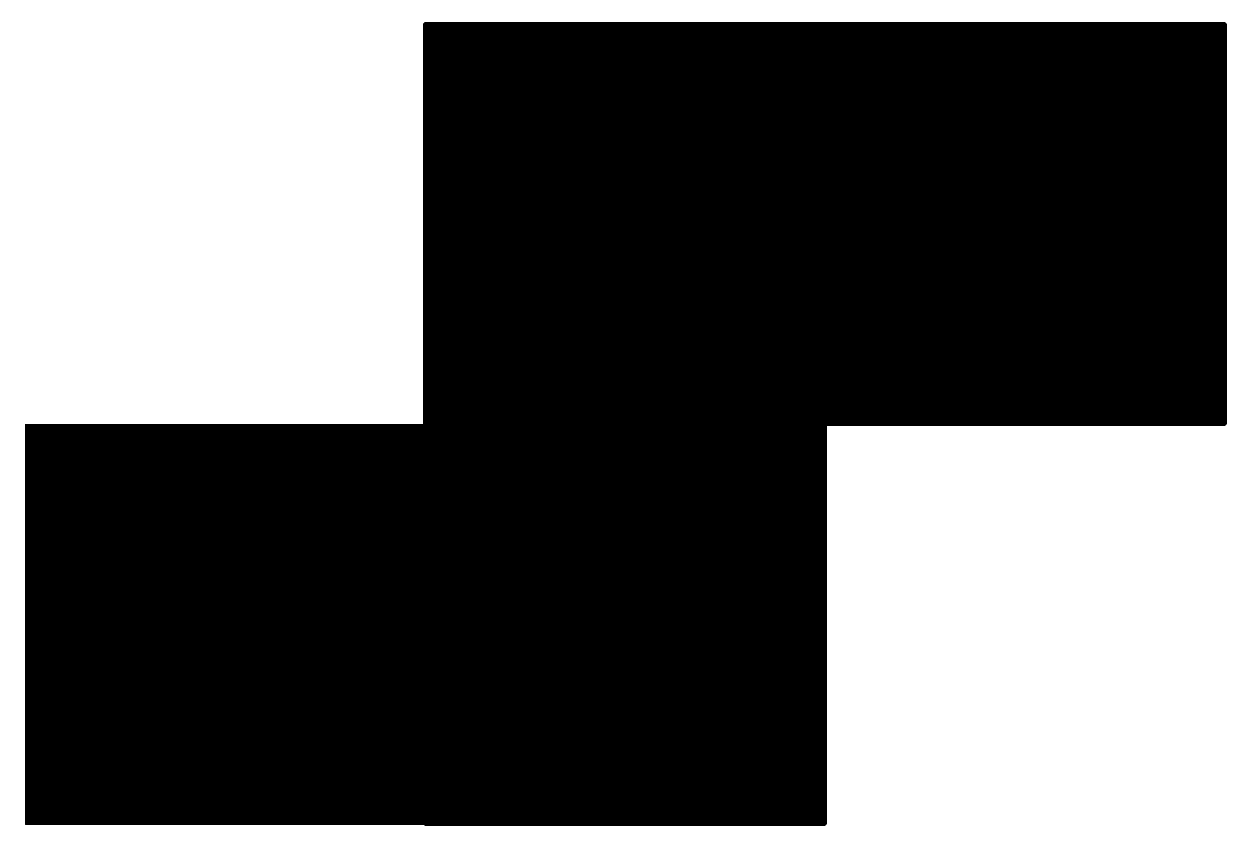}
  \captionsetup{width=1.2\linewidth,labelsep=newline,parindent=0cm,justification=centering}
  \caption{\small$c_1=1.24$ $c_2=7.53$}\end{center}
  \end{subfigure}  \qquad& 
  \begin{subfigure}{0.1\textwidth}\vskip-0.45cm
\begin{center}\hskip0.1cm\includegraphics[width=0.8\linewidth]{tetramin_8.pdf}
  \captionsetup{width=1.2\linewidth,labelsep=newline,parindent=0cm,justification=centering}
  \caption{\small$c_1=0.22$ $c_2=7.92$}\end{center}
  \end{subfigure}  & 
  \\
 
   \end{tabular}
\caption{\label{snubsquaredom1}Truncated square tiling: connected domains satisfying (A2) and related constants. The smallest ratio $c_2/c_1$ corresponds to case (C). 
\label{gooddomains_5}
}
   \end{figure}
    \begin{figure}[h!]
 \includegraphics[scale=1]{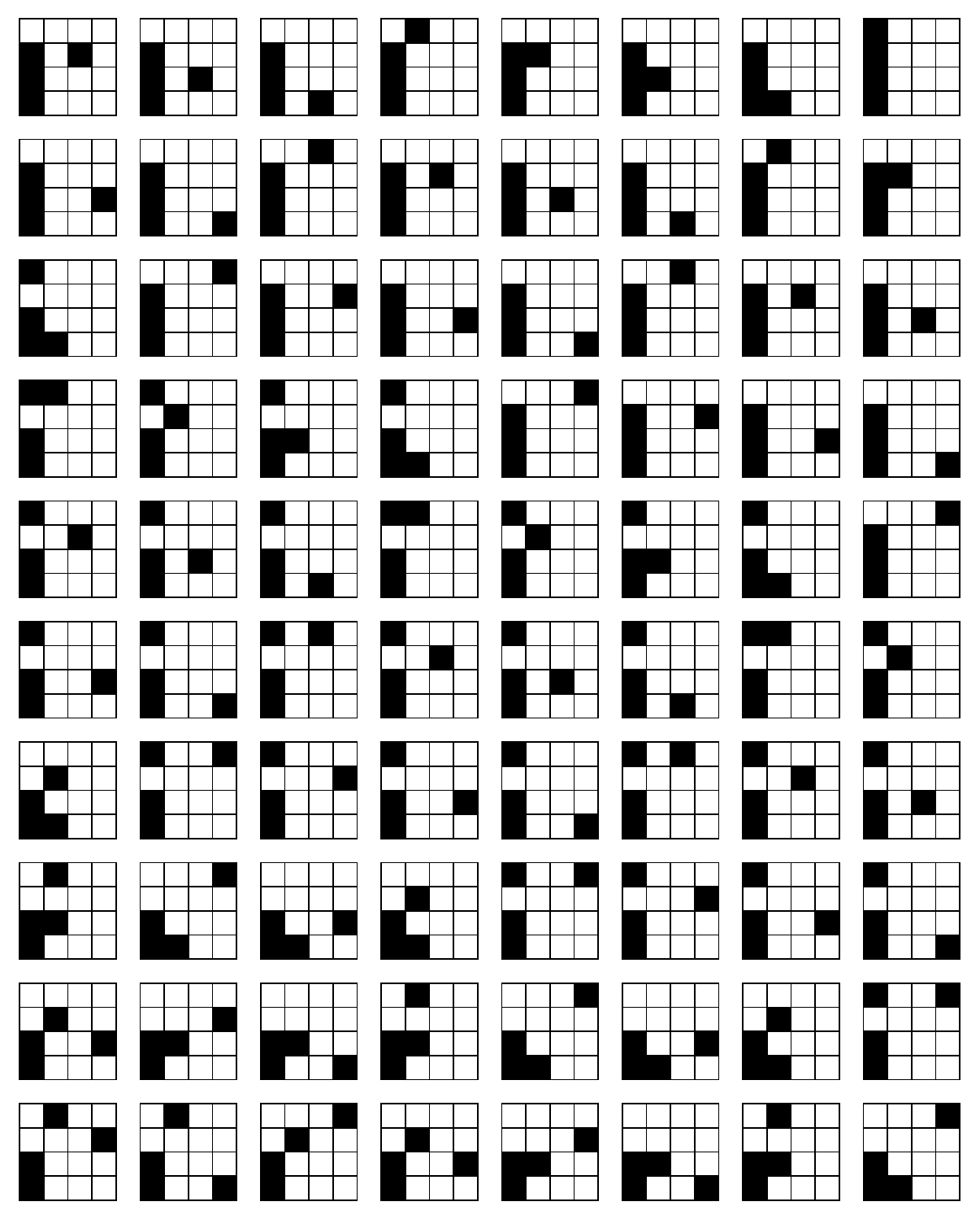}
 \caption{Truncated square tiling: first 80 of the 892 (over 1820) domains of the form $\cup_{k=1}^4 \Omega_0+v_k$ not satisfying condition (A2).\label{baddomains_5}}
\end{figure}

 \clearpage
 \subsection{Snub-hexagonal tiling}\.\\
\begin{figure}[h!]
\includegraphics[scale=0.3]{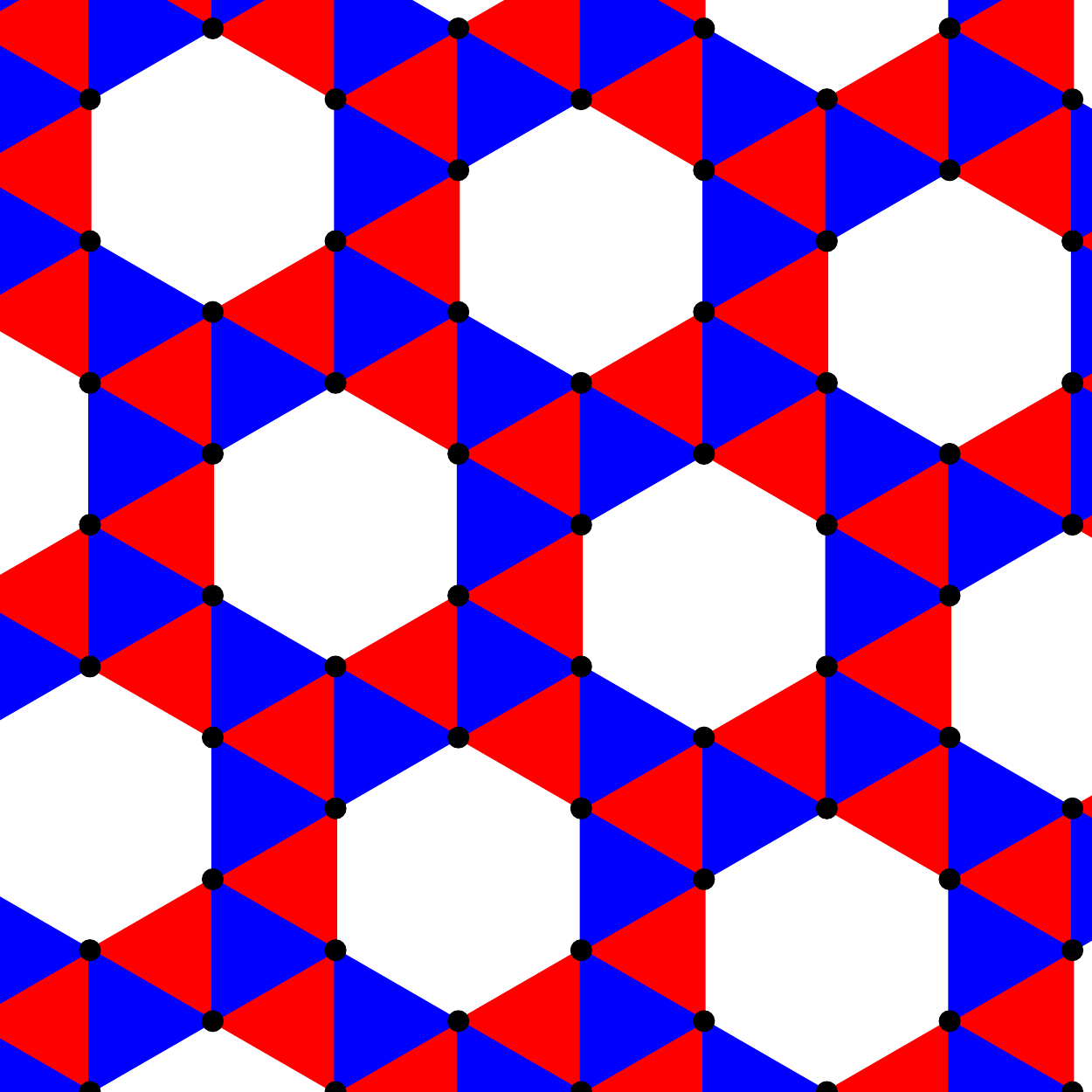}\hskip0.5cm\includegraphics[scale=0.3]{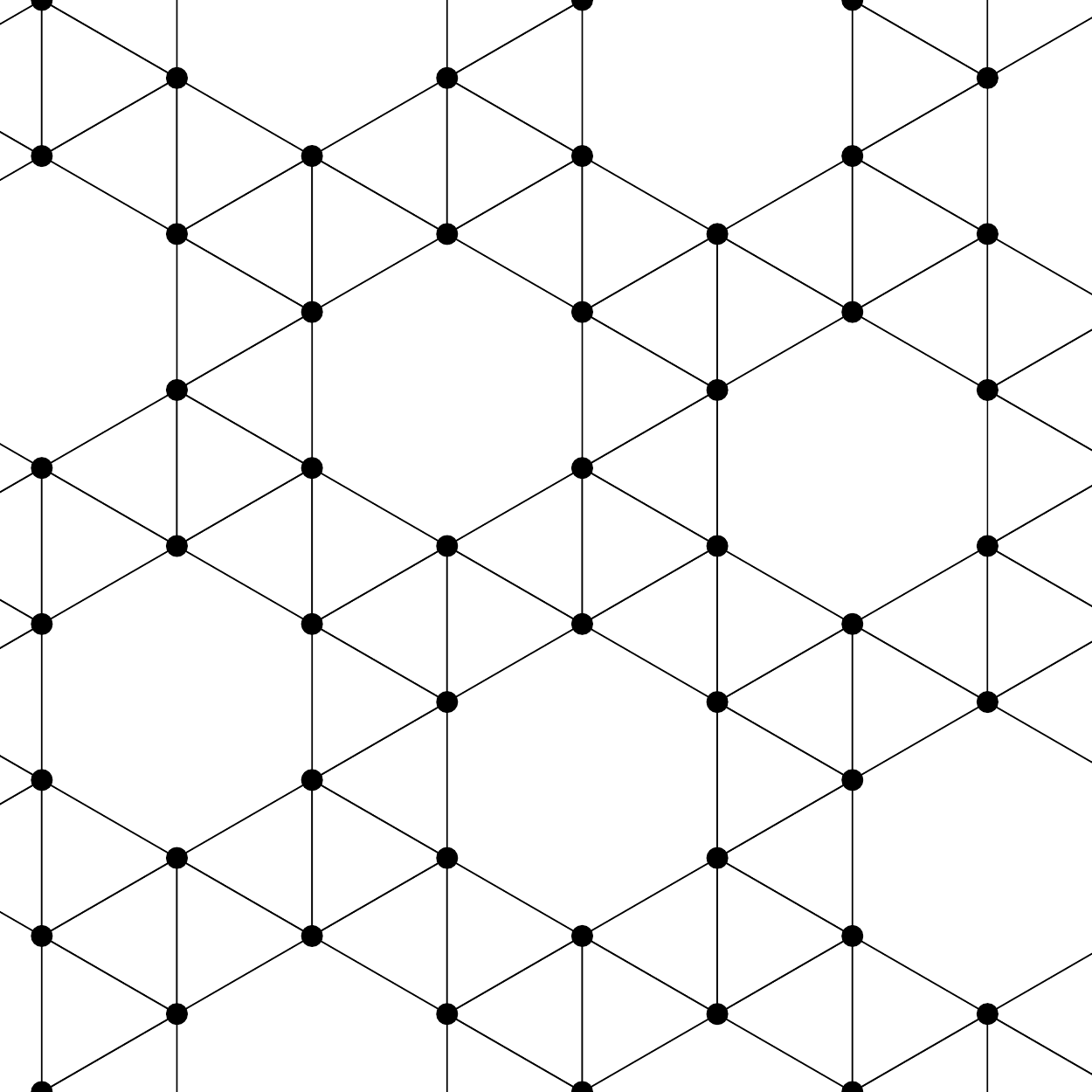}\\\vskip0.5cm
 \includegraphics[scale=0.3]{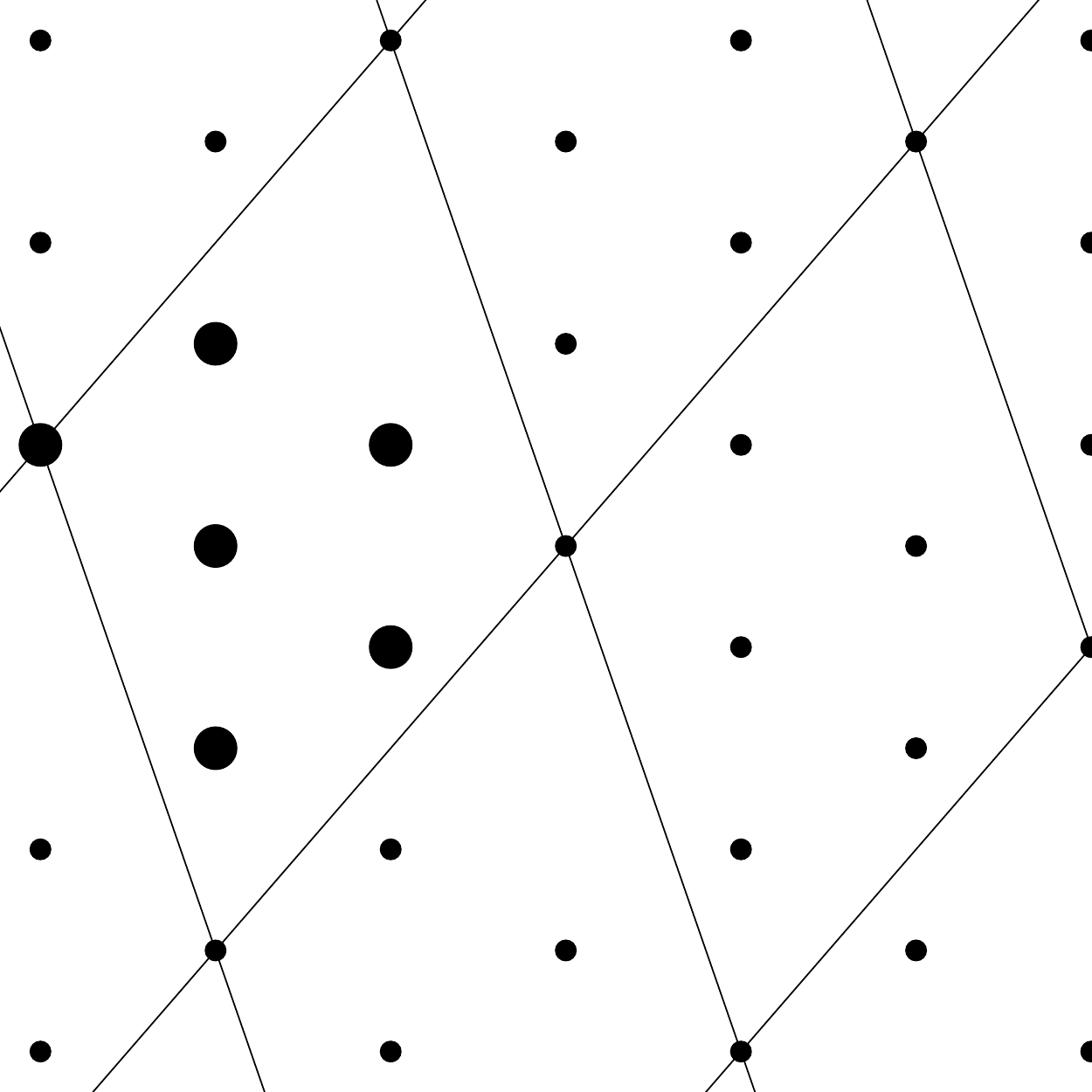}\hskip0.5cm\includegraphics[scale=0.3]{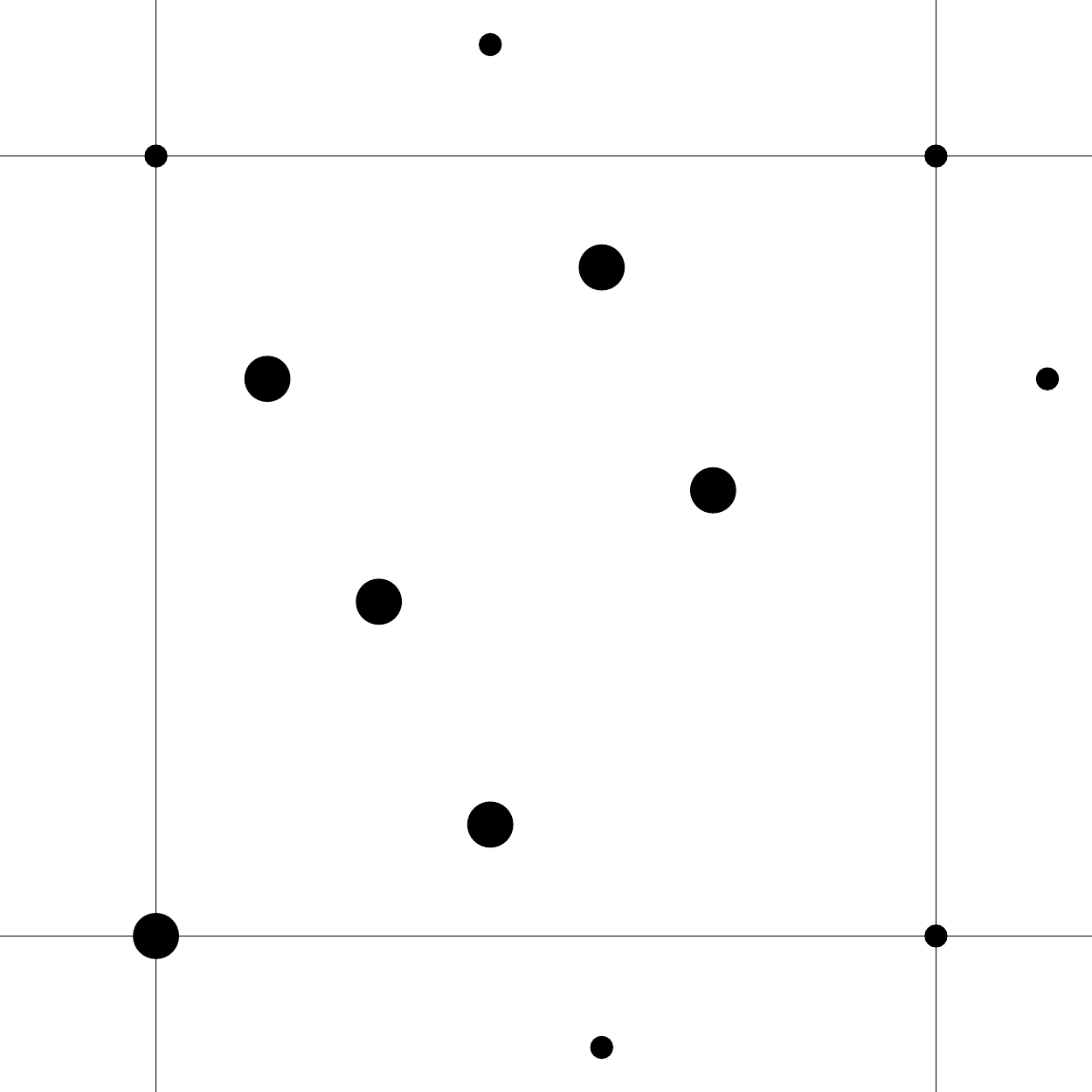}\\
 \caption{Snub-hexagonal tiling}\end{figure}\\
 $\Lambda=\cup_{j=1}^6 L^*(u_j+ \ZZ^2)$
 where
 $$
\begin{array}{ll}
u_1= \frac{1}{7}(0, 0)&  u_2= \frac{1}{7}(3,1) \\
 u_3= \frac{1}{7}(2,3)& u_4=\frac{1}{7}(5,4) \\
 u_5= \frac{1}{7}(1,5) & u_6=\frac{1}{7}(4,6) \\
\end{array}
$$
and 
$$L^*=\left(
\begin{array}{cc}
 \sqrt{3} & \frac{\sqrt{3}}{2} \\
 2 & -\frac{5}{2} \\
\end{array}
\right)$$\vskip0.2cm
\textbf{Example.} Choosing
  $$
\begin{array}{ll}
v_1=(0,0 )&
 v_2=(0,1)\\
 v_3=(0,2) &
 v_4=(0,3) \\
 v_5=(0,4) &
 v_6=(0,5) 
\end{array}
$$
condition (A2) is satisfied and the correspondig constants are $c_1=1$ and $c_2=7$.
\vskip0.5cm
\textbf{Example.} Choosing
  $$
\begin{array}{ll}
v_1=(0,0 )&
 v_2=(0,1)\\
 v_3=(0,2) &
 v_4=(0,3) \\
 v_5=(0,4) &
 v_6=(1,4) 
\end{array}
$$
condition (A2) is not satisfied.

\clearpage
\subsection{Rhombitrihexagonal tiling}\label{ssrombi}\.\\
\begin{figure}[h!]
 \includegraphics[scale=0.3]{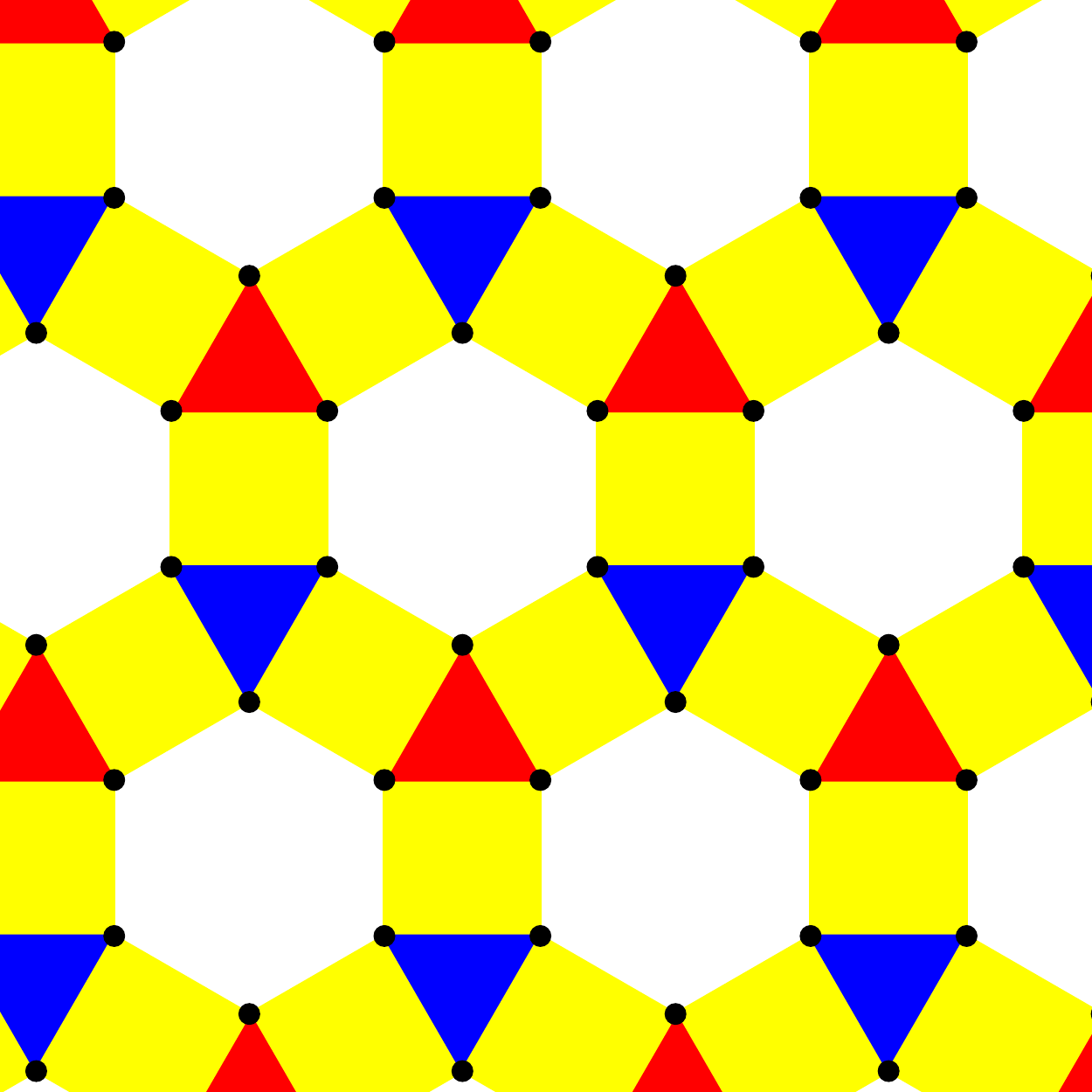}\hskip0.5cm\includegraphics[scale=0.3]{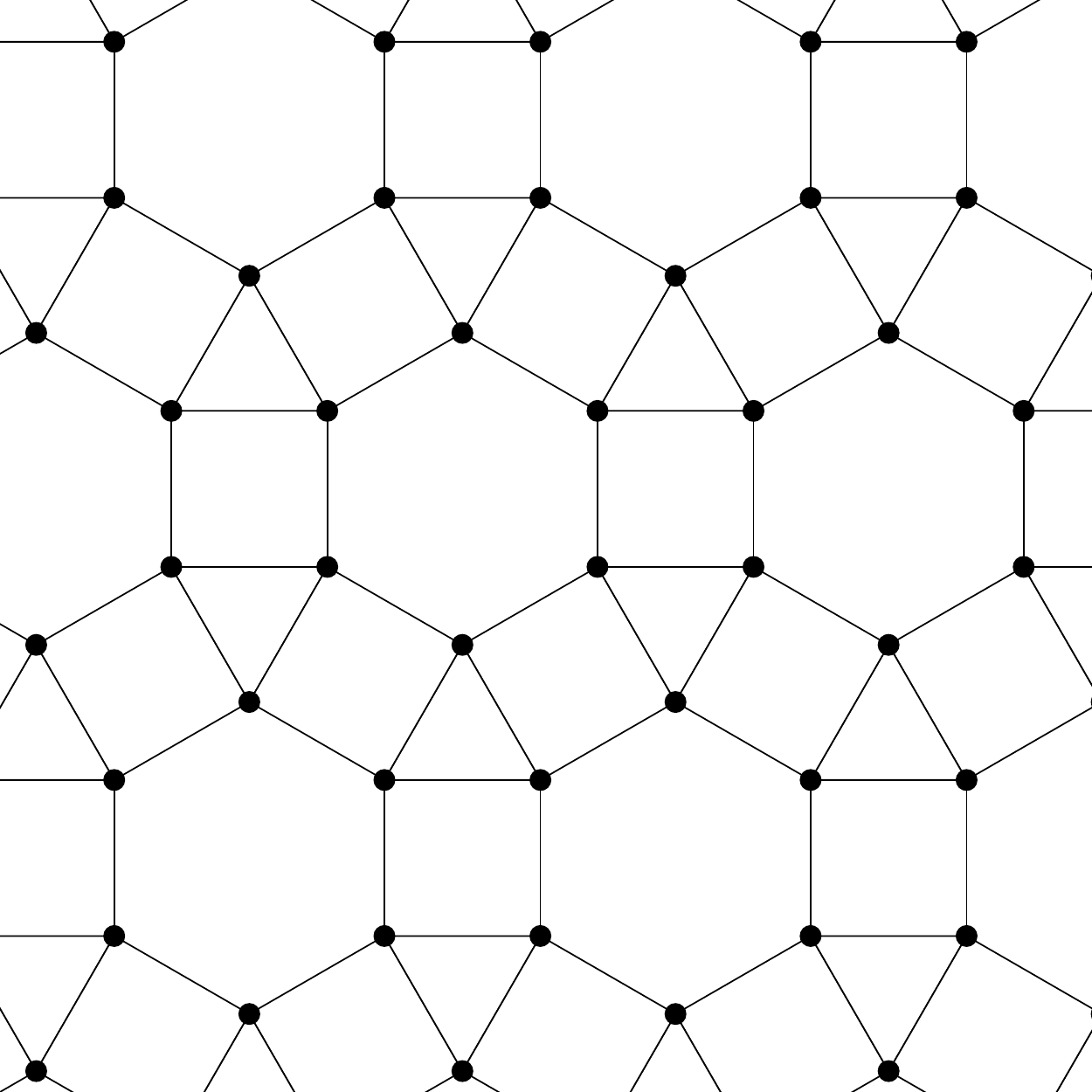}\\\vskip0.5cm
 \includegraphics[scale=0.3]{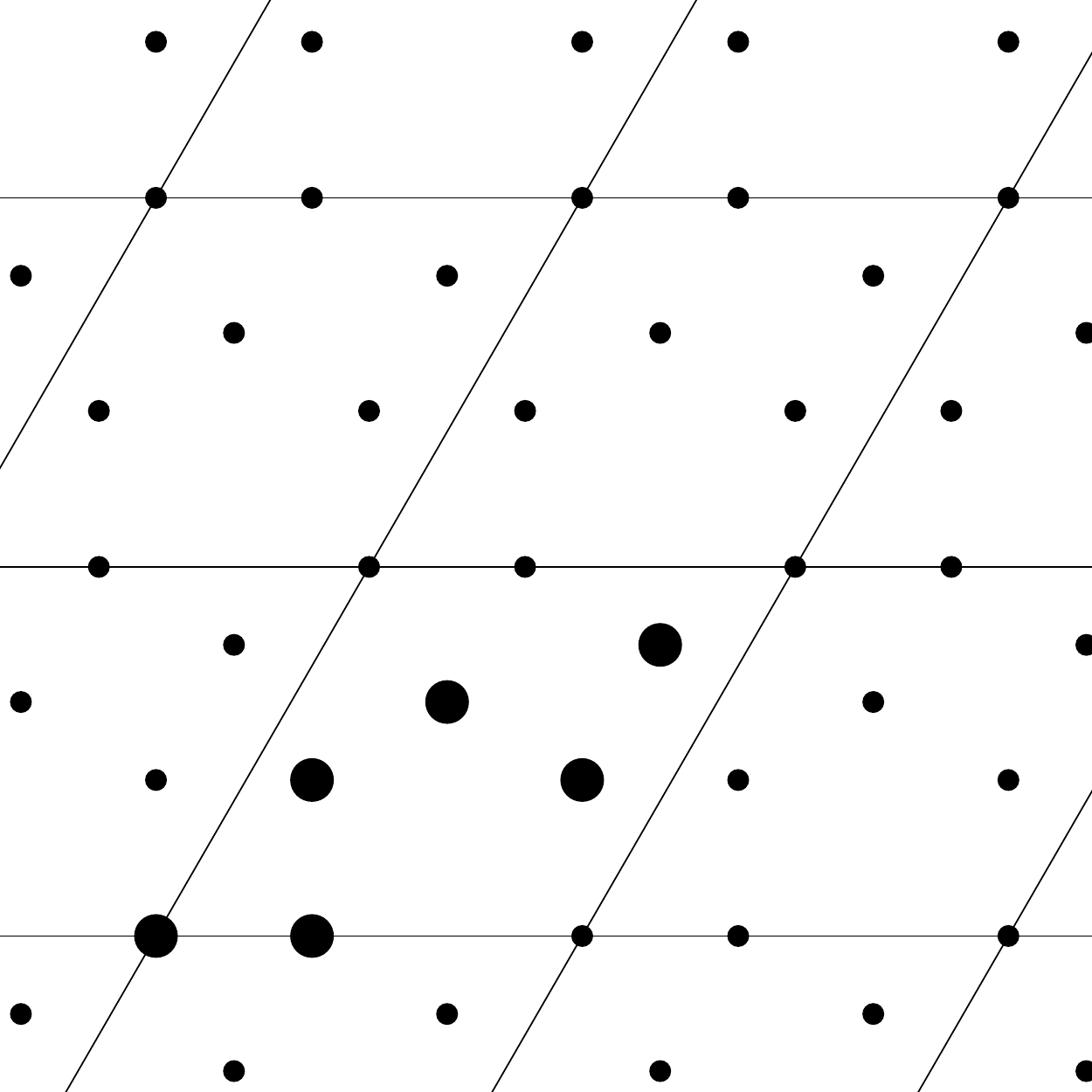}\hskip0.5cm\includegraphics[scale=0.3]{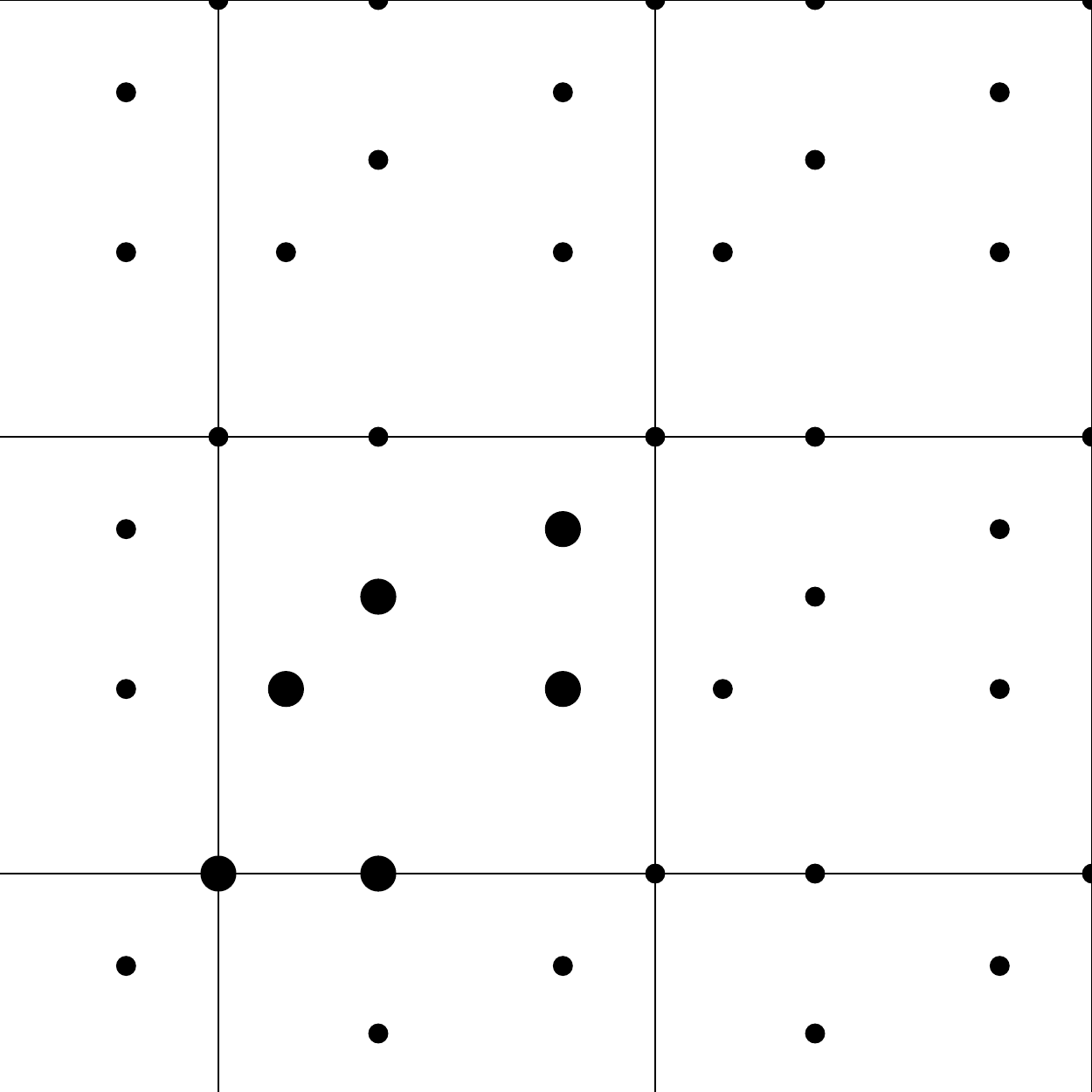}\\
 \caption{Rhombitrihexagonal tiling}\end{figure}\\
$\Lambda=\cup_{j=1}^6 L^*(u_j+ \ZZ^2)$
 where
 $$
\begin{array}{ll}
u_1= (0,0) &
u_2= (\frac{1}{2}(-1+\sqrt{3}), 0)\\
u_3= (-1+\frac{2}{3}\sqrt{3},  1-\frac{\sqrt{3}}{3})\qquad &
 u_4=(\frac{1}{2}(-1+\sqrt{3}), \frac{1}{2}(3-\sqrt{3})) \\
u_5= (\frac{1}{6}(3+\sqrt{3}), 1-\frac{\sqrt{3}}{3}) &
u_6= (\frac{1}{6}(3+\sqrt{3}),\frac{1}{6}(3+\sqrt{3})) 
\end{array}
$$
and 
$$L^*=\left(
\begin{array}{cc}
 1+\sqrt{3} & \frac{1}{2} \left(1+\sqrt{3}\right) \\
 0 & \frac{1}{2} \left(3+\sqrt{3}\right) \\
\end{array}
\right)$$
\vskip0.2cm
\textbf{Example.} Choosing
  $$
\begin{array}{ll}
v_1=(0,0 )&
 v_2=(0,1)\\
 v_3=(0,2) &
 v_4=(0,3) \\
 v_5=(0,4) &
 v_6=(0,5) 
\end{array}
\quad\text{or}\quad
 \begin{array}{ll}
v_1=(0,0 )&
 v_2=(0,1)\\
 v_3=(0,2) &
 v_4=(0,3) \\
 v_5=(0,4) &
 v_6=(1,4) 
\end{array}
$$
condition (A2) is not satisfied (see Examples of the previous section). 
\vskip0.2cm
\textbf{Example.} Choosing
  $$
\begin{array}{ll}
v_1=(0,0 )&
 v_2=(0,1)\\
 v_3=(0,2) &
 v_4=(0,3) \\
 v_5=(1,3) &
 v_6=(1,4) 
\end{array}
$$
condition (A2) is satisfied and the correspondig constants are $c_1=0.47$ and $c_2=11.92$.

\clearpage
\subsection{Truncated hexagonal tiling}\.\\

\begin{figure}[h!]
 \includegraphics[scale=0.3]{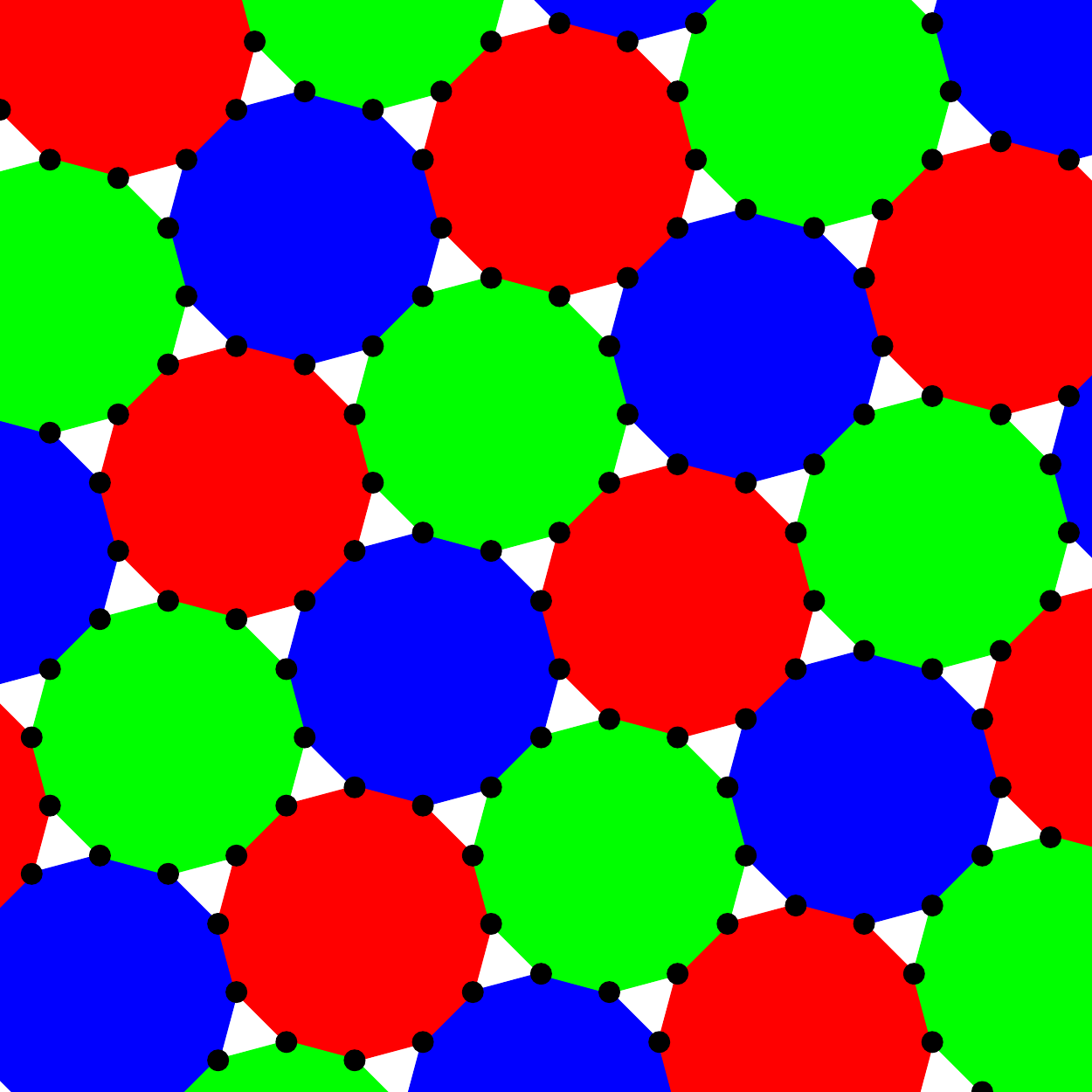}\hskip0.5cm\includegraphics[scale=0.3]{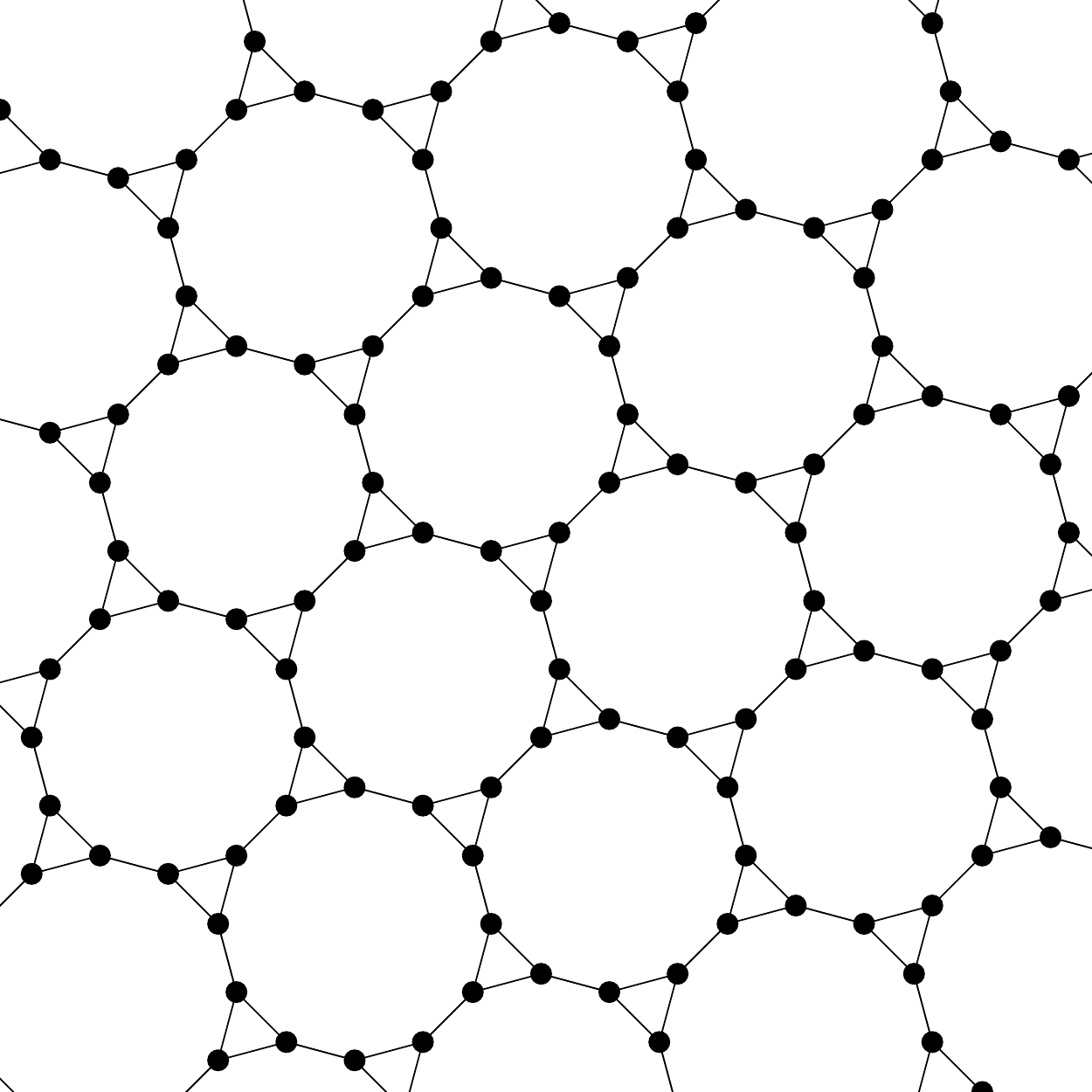}\\\vskip0.5cm
 \includegraphics[scale=0.3]{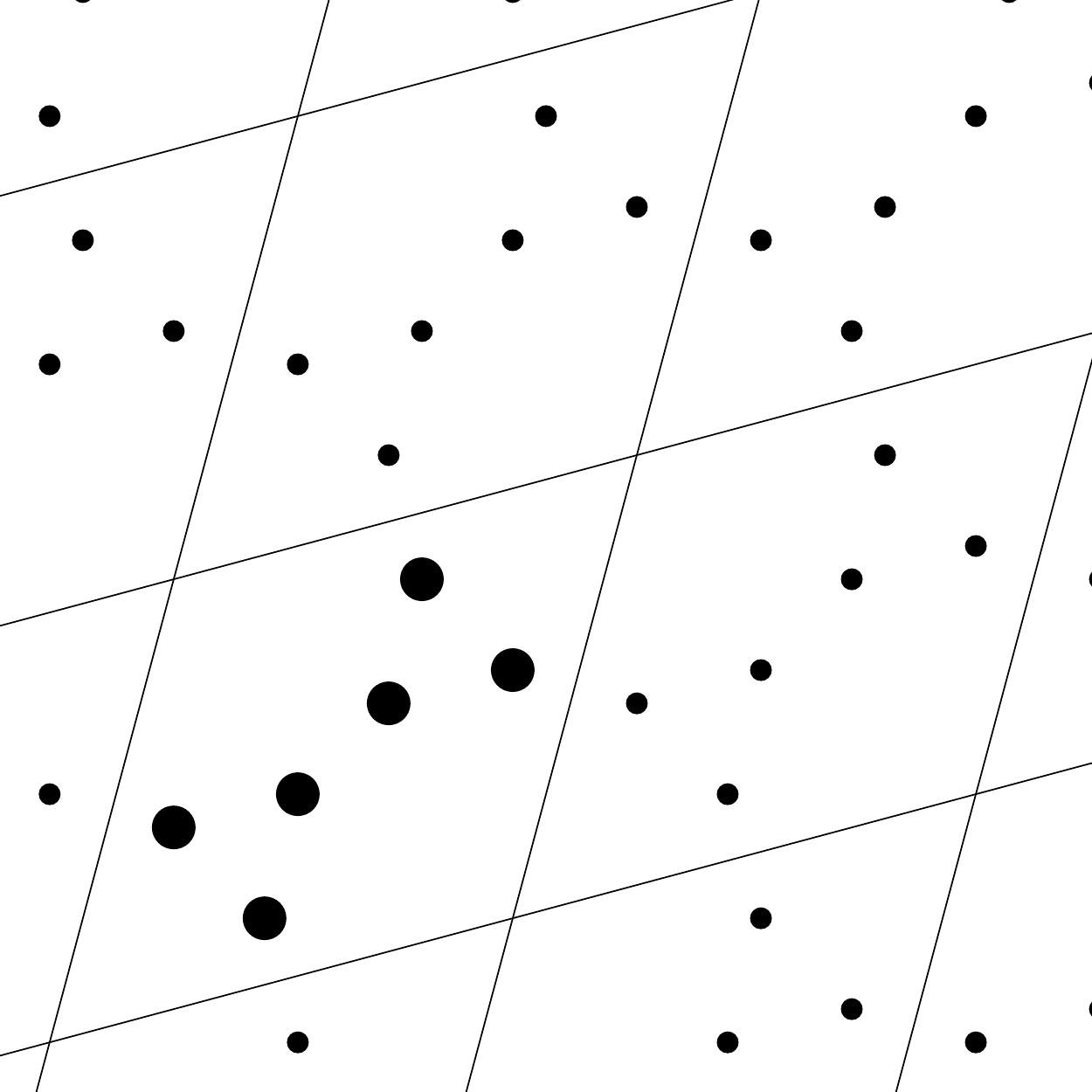}\hskip0.5cm\includegraphics[scale=0.3]{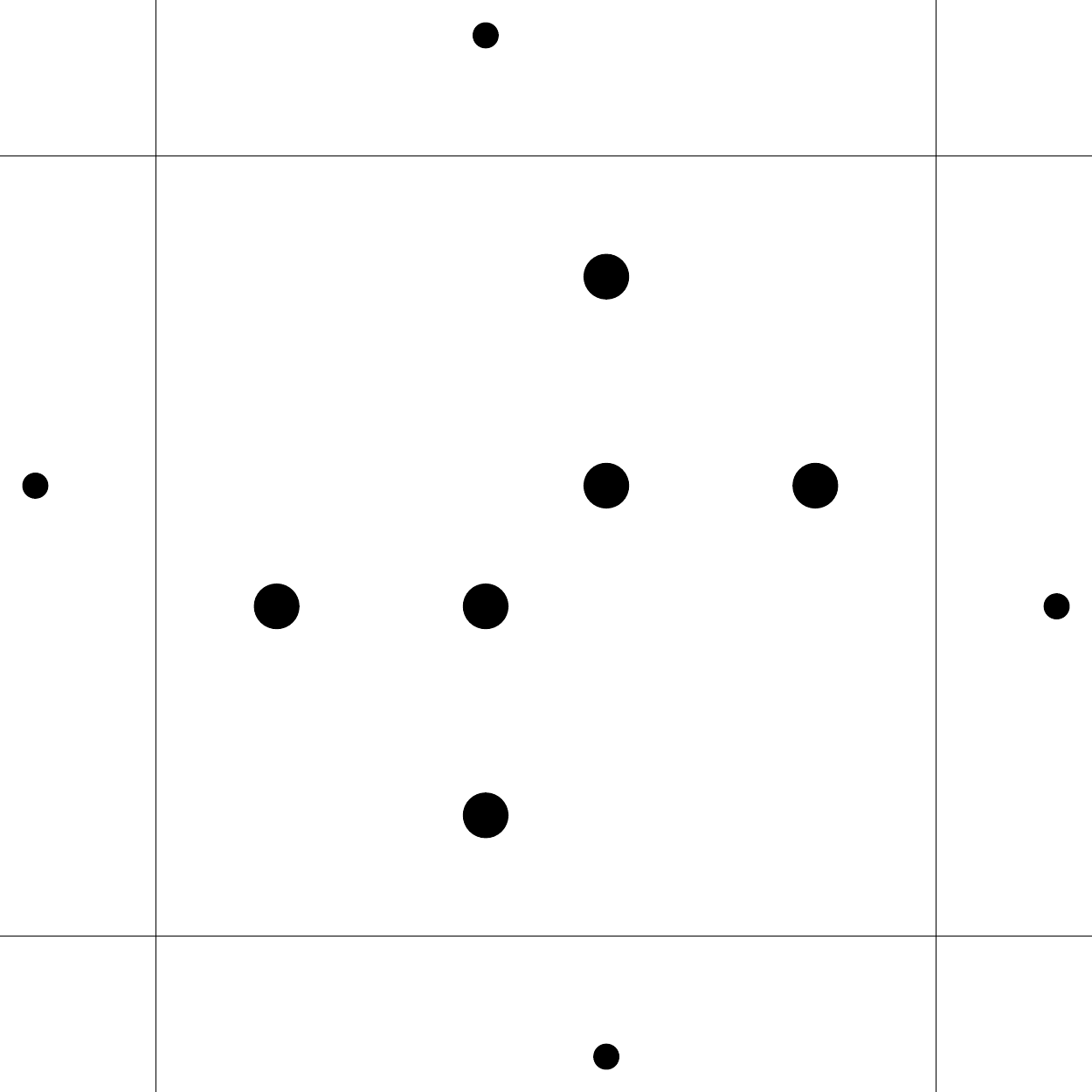}\\
 \caption{Truncated hexagonal tiling}\end{figure}
 $\Lambda=\cup_{j=1}^6L^*( u_j+ \ZZ^2)$ where
 $$
\begin{array}{ll}
u_1= (\frac{\sqrt{3}}{3},2-2\frac{\sqrt{3}}{3}) &
u_2= (1-\frac{\sqrt{3}}{3}, -1+2\frac{\sqrt{3}}{3})\\
u_3= (-1+2\frac{\sqrt{3}}{3},  1-\frac{\sqrt{3}}{3})\qquad &
 u_4=(2-2\frac{\sqrt{3}}{3}), \frac{\sqrt{3}}{3}) \\
u_5= (\frac{\sqrt{3}}{3}, \frac{\sqrt{3}}{3}) &
u_6= (1-\frac{\sqrt{3}}{3},1-\frac{\sqrt{3}}{3}) 
\end{array}
$$
and 
$$L^*=\left(
\begin{array}{cc}
 1+\frac{\sqrt{3}}{2} & \frac{1}{2} \\
 \frac{1}{2} & 1+\frac{\sqrt{3}}{2} \\
\end{array}
\right).$$
\vskip0.5cm
\textbf{Example.} As in Section \ref{ssrombi}, choosing
  $$
\begin{array}{ll}
v_1=(0,0 )&
 v_2=(0,1)\\
 v_3=(0,2) &
 v_4=(0,3) \\
 v_5=(0,4) &
 v_6=(0,5) 
\end{array}
\quad\text{or}\quad
\begin{array}{ll}
v_1=(0,0 )&
 v_2=(0,1)\\
 v_3=(0,2) &
 v_4=(0,3) \\
 v_5=(0,4) &
 v_6=(1,4) 
\end{array}
$$
condition (A2) is not satisfied. 
On the other hand, again as in Section \ref{ssrombi}, the choice
  $$
\begin{array}{ll}
v_1=(0,0 )&
 v_2=(0,1)\\
 v_3=(0,2) &
 v_4=(0,3) \\
 v_5=(1,3) &
 v_6=(1,4) 
\end{array}
$$
satisfies condition (A2). The correspondig constants are $c_1=0.15$ and $c_2=15.6$.

\clearpage
\subsection{Truncated trihexagonal tiling}
$\Lambda=\cup_{j=1}^{12} u_j+ L^* \ZZ^2$
 where
\begin{figure}[h!]
 \includegraphics[scale=0.3]{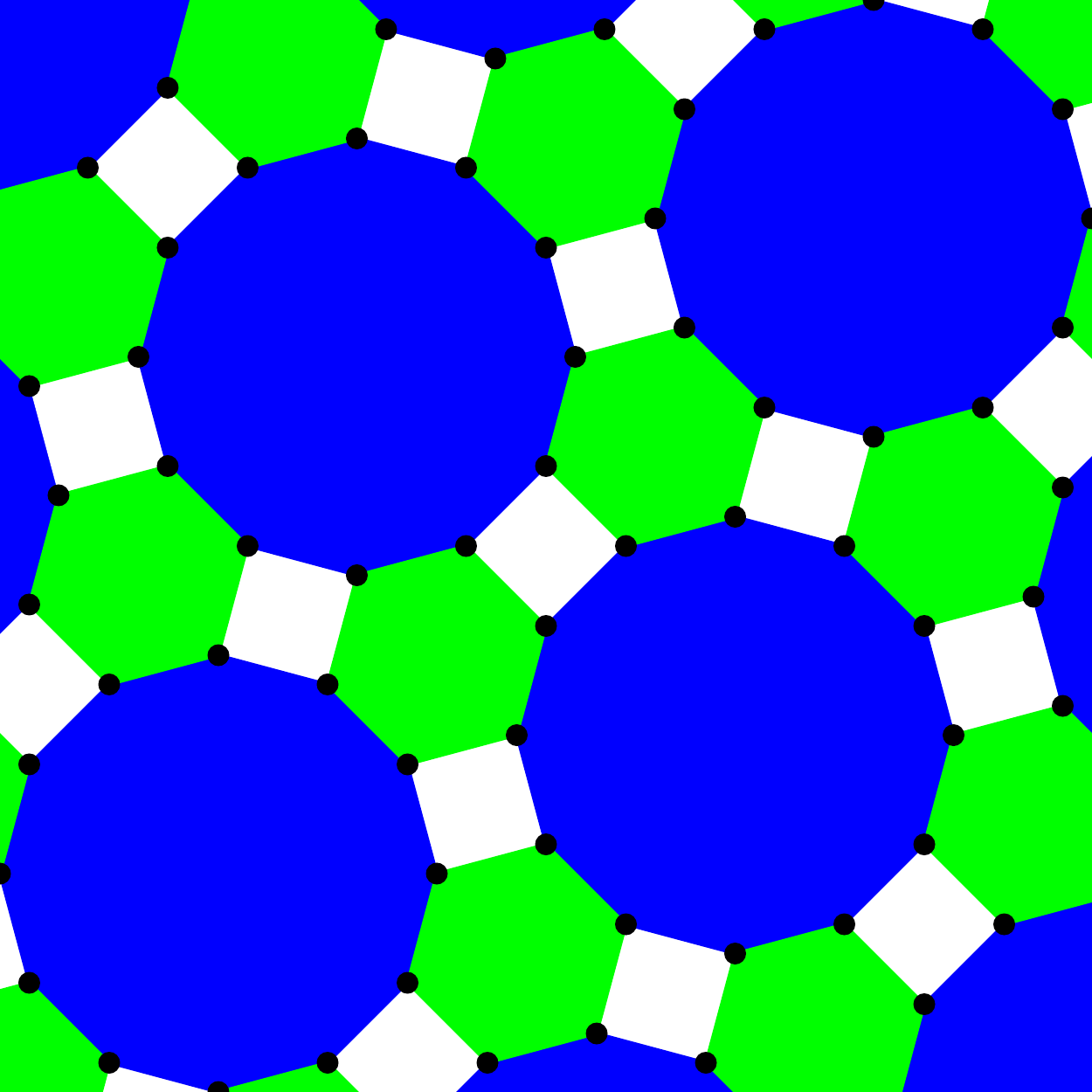}\hskip0.5cm\includegraphics[scale=0.3]{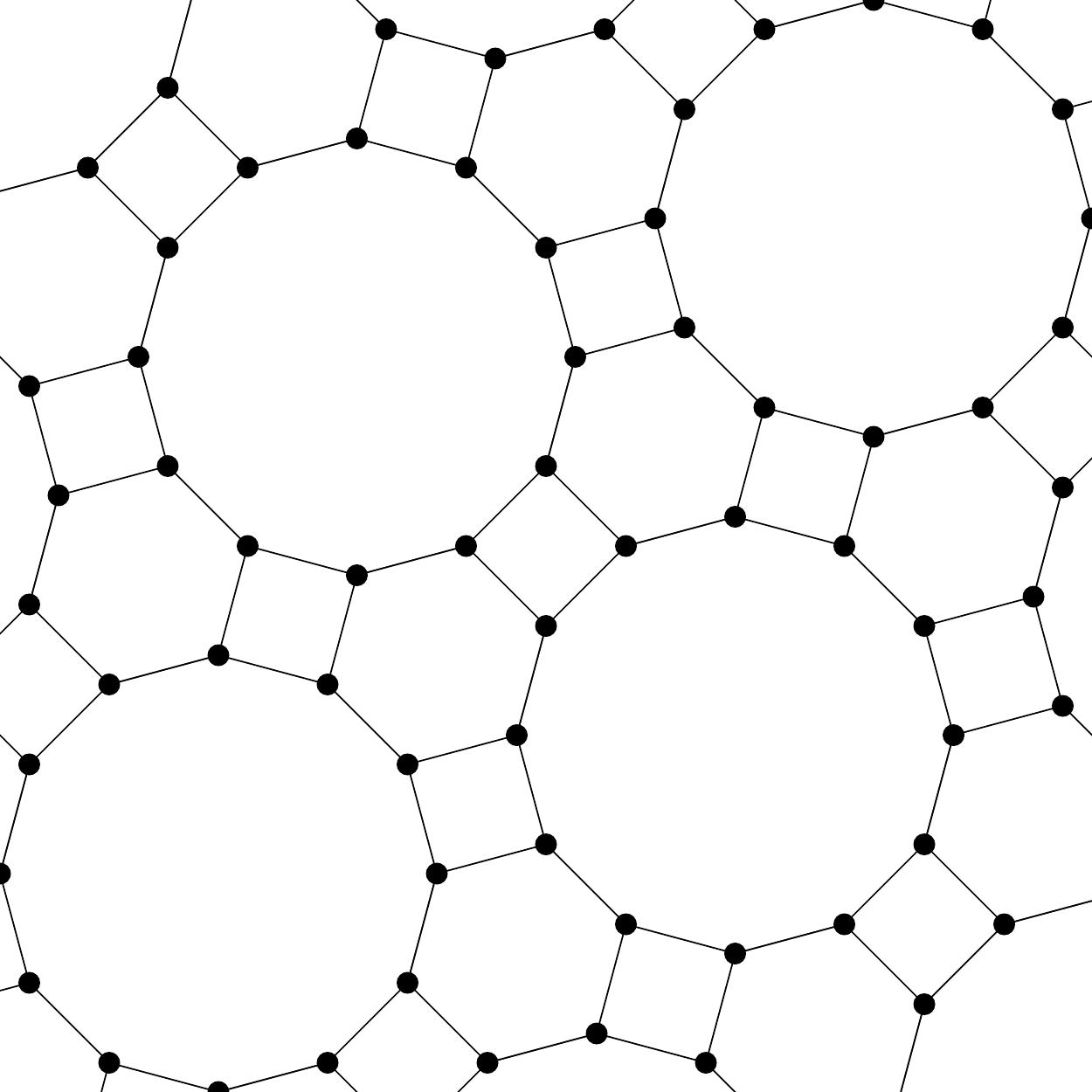}\\\vskip0.5cm
 \includegraphics[scale=0.3]{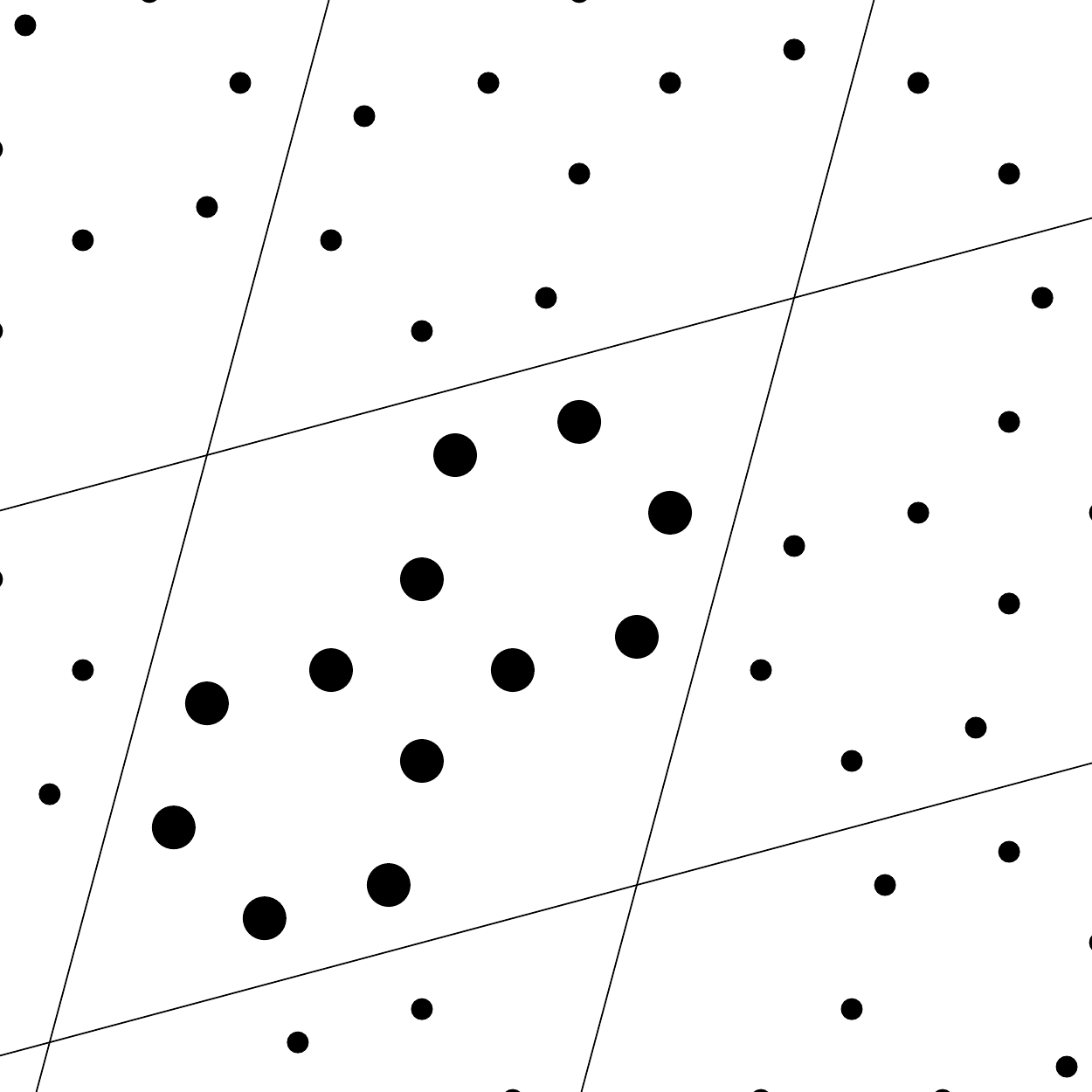}\hskip0.5cm\includegraphics[scale=0.3]{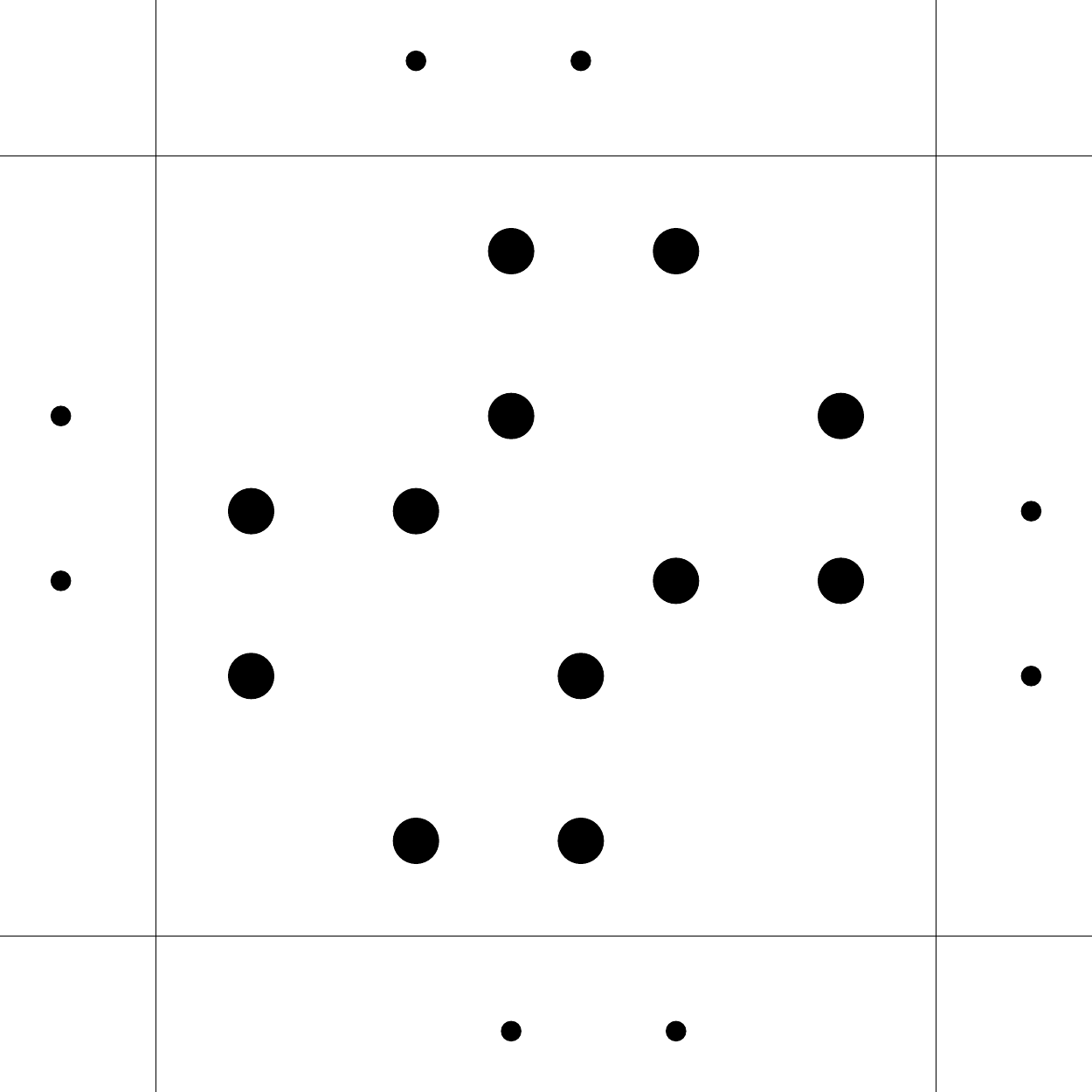}\\
 \caption{Truncated trihexagonal tiling}\end{figure}
  $$
\begin{array}{ll}
 u_1=\frac{1}{6}(2, 5-\sqrt{3}) &
 u_2=\frac{1}{6}(-1+\sqrt{3}, 5-\sqrt{3}) \\
 u_3=\frac{1}{6}(-1+\sqrt{3},2)&
 u_4=\frac{1}{6}( 2, -1+\sqrt{3}) \\
 u_5=\frac{1}{6}(5-\sqrt{3},-1+\sqrt{3}) &
 u_6=\frac{1}{6}(5-\sqrt{3} , 2) \\
 u_7=\frac{1}{6}( 4 ,7-\sqrt{3})&
 u_8=\frac{1}{6}(1+\sqrt{3} , 7-\sqrt{3}) \\
 u_9=\frac{1}{6}(1+\sqrt{3} ,4)&
 u_{10}=\frac{1}{6}(4, 1+\sqrt{3}) \\
 u_{11}=\frac{1}{6}(7-\sqrt{3},1+\sqrt{3})&
 u_{12}=\frac{1}{6}(7-\sqrt{3}, 4)
\end{array}
$$
and 
$$L^*=\left(
\begin{array}{cc}
 \frac{1}{2} \left(3+\sqrt{3}\right) & \frac{1}{2}
   \left(3-\sqrt{3}\right) \\
 \frac{1}{2} \left(3-\sqrt{3}\right) & \frac{1}{2}
   \left(3+\sqrt{3}\right) \\
\end{array}
\right).$$
\vskip0.2cm
\textbf{Example.} Choosing
  $$
\begin{array}{ll}
v_1=(0,0 )&
 v_2=(0,1)\\
 v_3=(1,0) &
 v_4=(1,1) \\
 v_5=(2,0) &
 v_6=(2,1) \end{array}\quad
 \begin{array}{ll}
 v_7=(3,0) &
v_8=(3,1) \\
 v_9=(4,0) &
 v_{10}=(4,1)\\
 v_{11}=(5,0) &
 v_{12}=(5,1) \\
\end{array}
$$
condition (A2) is satisfied with constants $c_1=2.71$ and $c_2=28.02$. 
\vskip0.2cm
\textbf{Example.} Choosing 
 $$
\begin{array}{ll}
v_1=(0,0 )&
 v_2=(0,1)\\
 v_3=(0,2) &
 v_4=(0,3) \\
 v_5=(1,0) &
 v_6=(1,1) \end{array}\quad
 \begin{array}{ll}
 v_7=(1,2) &
v_8=(1,3) \\
 v_9=(2,0) &
 v_{10}=(2,1)\\
 v_{11}=(2,2) &
 v_{12}=(2,3) \\
\end{array}
$$
condition (A2) is not satisfied.

\end{document}